\documentclass[11pt]{amsart}

\usepackage{amsfonts,amssymb,amsmath,amsthm,latexsym,graphics,epsfig}
\usepackage{verbatim,enumerate,array,booktabs,color,bigstrut,prettyref,tikz-cd}
\usepackage{multirow}
\usepackage[all]{xy}
\usepackage[backref]{hyperref}
\usepackage{ctable}

\usepackage{mathtools}

\newcommand\myiso{\stackrel{\mathclap{\normalfont\mbox{\small $p$}}}{-}}
\newcommand\myisot{\stackrel{\mathclap{\normalfont\mbox{\small $3$}}}{-}}

\newrefformat{eq}{\textup{(\ref{#1})}}
\newrefformat{prty}{\textup{(\ref{#1})}}

\definecolor{mylinkcolor}{rgb}{0.8,0,0}
\definecolor{myurlcolor}{rgb}{0,0,0.8}
\definecolor{mycitecolor}{rgb}{0,0,0.8}
\hypersetup{colorlinks=true,urlcolor=myurlcolor,citecolor=mycitecolor,linkcolor=mylinkcolor,linktoc=page,breaklinks=true}

\newtheorem{defn}{Definition}[section]

\newtheorem{lemma}[defn]{Lemma}

\newtheorem{thm}[defn]{Theorem}

\newtheorem{cor}[defn]{Corollary}

\theoremstyle{definition}

\newtheorem*{ack}{Acknowledgements}
\newtheorem{remark}[defn]{Remark}

\newtheorem{example}[defn]{Example}

\newcommand{\Z}{\mathbb Z}
\newcommand{\Q}{\mathbb{Q}}

\newcommand{\FF}{\mathbb F}

\newcommand{\PP}{\mathbb P}

\newcommand{\Rats}{\Q}

\newcommand{\arrow}{\longrightarrow}

\newcommand{\Gal}{\operatorname{Gal}}

\newcommand{\GQ}{\Gal(\overline{\Rats}/\Rats)}

\newcommand{\GL}{\operatorname{GL}}
\newcommand{\PGL}{\operatorname{PGL}}

\newcommand{\E}{\mathcal{E}}

\newcommand{\tor}{\mathrm{tors}}

\begin{document}
	
	
	
	\title[Isogeny-Torsion Graphs]{A classification of isogeny-torsion graphs of $\Q$-isogeny classes of elliptic curves}
	
	\author{Garen Chiloyan}
	\address{Department of Mathematics, University of Connecticut, Storrs, CT 06269, USA}
	\email{garen.chiloyan@uconn.edu}
	
	\author{\'Alvaro Lozano-Robledo}
	\address{Department of Mathematics, University of Connecticut, Storrs, CT 06269, USA}
	\email{alvaro.lozano-robledo@uconn.edu}
	\urladdr{\url{http://alozano.clas.uconn.edu/}}

	
	\begin{abstract}
		Let $\mathcal{E}$ be a $\Q$-isogeny class of elliptic curves defined over $\Q$. The isogeny graph associated to $\mathcal{E}$ is a graph which has a vertex for each elliptic curve in the $\Q$-isogeny class $\mathcal{E}$, and an edge for each cyclic $\Q$-isogeny of prime degree between elliptic curves in the isogeny class, with the degree recorded as a label of the edge. In this paper, we define an isogeny-torsion graph to be an isogeny graph where, in addition, we label each vertex with the abstract group structure of the torsion subgroup over $\Q$ of the corresponding elliptic curve. Then, the main result of the article is  a classification of all the possible isogeny-torsion graphs that occur for $\Q$-isogeny classes of elliptic curves defined over the rationals.  
	\end{abstract}
	
	\maketitle
	
	\section{Introduction}
	Let $E$ be an elliptic curve defined over $\Q$. The Mordell--Weil theorem shows that $E(\Q)$ is a finitely generated abelian group, and therefore $E(\Q)_\text{tors}$ is a finite abelian group. By Mazur's theorem, there are precisely $15$ different isomorphism types of torsion groups that occur over $\Q$ (which we recall in Theorem \ref{thm-mazur} below). Now, suppose that $E$ admits an isogeny $\phi\colon E\to E'$ defined over $\Q$, where $E'/\Q$ is another elliptic curve defined over $\Q$. In this article, we are interested in the possible pairs of torsion subgroups that can occur as $(E(\Q)_\text{tors},E'(\Q)_\text{tors})$.
	
	More generally, let $\mathcal{E}$ be a $\Q$-isogeny class of elliptic curves defined over $\Q$. The isogeny graph associated to $\mathcal{E}$ is a graph which has a vertex for each elliptic curve in the $\Q$-isogeny class $\mathcal{E}$, and an edge for each cyclic $\Q$-isogeny of prime degree between elliptic curves in the isogeny class, with the degree recorded as a label of the edge. In this article, we define an isogeny-torsion graph to be an isogeny graph where, in addition, we label each vertex with the abstract group structure of the torsion subgroup over $\Q$ of the corresponding elliptic curve $E$ (that is, the vertex associated to $E$ is labeled with the abstract group structure of $E(\Q)_{\text{tors}}$). Thus, the main goal of this article is to classify all the possible isogeny-torsion graphs that occur for $\Q$-isogeny classes of elliptic curves defined over $\Q$. For the remainder of the article, for each positive integer $n$, the term $n$-isogeny is strictly reserved for a $\Q$-isogeny whose kernel is cyclic of order $n$.
	
			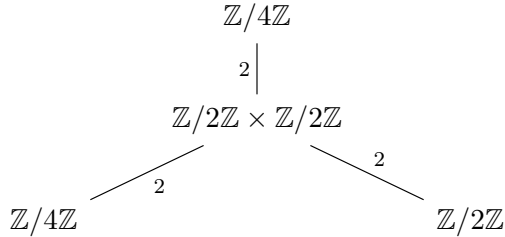
\begin{figure}[h!] 
		\begin{tikzcd}
			& \Z/4\Z                             &       \\
			& \Z/2\Z\times \Z/2\Z \arrow["2", u, no head] \arrow["2", ld, no head] \arrow["2", rd, no head] &       \\
			\Z/4\Z &                                   & \Z/2\Z
		\end{tikzcd}
		\caption{The isogeny-torsion graph for the isogeny class \texttt{17.a}.}
		\label{fig-17.a2}
	\end{figure} 
	
	\begin{example}
		Let $E/\Q$ be the elliptic curve $y^2+xy+y=x^3-x^2-6x-4$ with LMFDB \cite{lmfdb} label \texttt{17.a2}. The curve $E$ admits three distinct $2$-isogenies, and no other $\Q$-isogenies of prime degree. The isogeny class \texttt{17.a} consists of four elliptic curves. The torsion subgroup of $E$ is isomorphic to $\Z/2\Z\times \Z/2\Z$, while those curves at the corners of the isogeny graph have torsion subgroups $\Z/2\Z$, $\Z/4\Z$, and $\Z/4\Z$, respectively. Thus, the isogeny-torsion graph associated to the isogeny class \texttt{17.a} is as it appears in Figure \ref{fig-17.a2}. 

We will abbreviate $\Z/a\Z$ and $\Z/a\Z\times \Z/b\Z$ by $[a]$ and $[a,b]$, respectively, and we will denote an isogeny-graph as the one in Figure \ref{fig-17.a2} by $([2,2],[4],[4],[2])$ and say it is of $T_4$ type.	
	\end{example}

Our results are as follows.

\begin{thm}\label{thm-mainisogenygraphs}
	There are $26$ isomorphism types of isogeny graphs that are associated to elliptic curves defined over $\Q$. More precisely, there are $16$ types of (linear) $L_k$ graphs, $3$ types of (non-linear two-primary torsion) $T_k$ graphs, $6$ types of (rectangular) $R_k$ graphs, and $1$ type of (special) $S$ graph. The possible configurations and labels of isogeny graphs are listed in the first two columns of Tables \ref{L_k graphs} through \ref{S_k graphs} (see Section \ref{sec-tables}). Moreover, there are $11$ isomorphism types of isogeny graphs that are associated to elliptic curves over $\Q$ with complex multiplication, namely the types $L_2(p)$ for $p=2,3,11,19,43,67$, or $163$, and types $L_4$, $T_4$, $R_4(6)$, and $R_4(14)$ (and examples are given in Table \ref{tab-CMgraphs}). Finally, the type $L_4$ occurs exclusively for elliptic curves with CM.
	\end{thm} 

While the results of Theorem \ref{thm-mainisogenygraphs} are essentially already known, we could not find a suitable reference, so we offer a full proof here for the sake of completeness (in Section \ref{sec-whatgraphs}). The classification of isogeny graphs is necessary for our main theorem, which details a complete classification of isogeny-torsion graphs over $\Q$.

\begin{thm}\label{thm-main2} There are $52$ isomorphism types of isogeny-torsion graphs that are associated to elliptic curves defined over $\Q$. In particular, there are $23$ types of $L_k$ graphs, $13$ types of $T_k$ graphs, $12$ types of $R_k$ graphs, and $4$ types of $S$ graphs. The possible configurations of isogeny-torsion graphs are listed in the third column of Tables \ref{L_k graphs} through \ref{S_k graphs}. Moreover, there are $16$ isomorphism types of isogeny-torsion graphs that are associated to elliptic curves over $\Q$ with complex multiplication (and examples are given in Table \ref{tab-CMgraphs}).
	\end{thm}

The degrees of rational isogenies of elliptic curves defined over $\Q$ have been described completely in the literature. One of the most important milestones in the classification was \cite{mazur1}, where Mazur dealt with the case of rational isogenies of elliptic curves defined over $\Q$ of prime degree. The complete classification of the degrees of rational isogenies of elliptic curves defined over $\Q$, for prime or composite degree, was completed due to work of Fricke, Kenku, Klein, Kubert, Ligozat, Mazur  and Ogg, among others (see Theorem \ref{thm-ratnoncusps} below, and \cite{lozano0}, Section 9). In particular, the work \cite{kenku} of Kenku shows that there are at most $8$ elliptic curves in each isogeny class over $\Q$ (see Theorem \ref{thm-kenku} below). Theorem \ref{thm-mainisogenygraphs} follows directly from the classification of rational isogenies over $\Q$.

  In Section \ref{sec-tables}, we include tables with every type of isogeny and isogeny-torsion graphs. The classification of isogeny-torsion graphs begins in  Section \ref{sec-CM}, for elliptic curves with complex multiplication (in particular, we show that only $16$ isogeny-torsion graphs are possible for elliptic curves with CM; see Table \ref{tab-CMgraphs}).  After some preliminary results in Sections \ref{sec-preliminaries} and \ref{sec-lemmas}, and building on Kenku's work \cite{kenku}, we shall show in Section \ref{sec-whatgraphs} that there are $26$ types of isogeny graphs (including the trivial graph with one vertex), which we denote linear ($L_1$, $L_2(p)$ and $L_3(q^2)$ for various primes $p=2,3,5,7,11,13,17,19,37,43,67$, or $163$, and $q=3$, or $5$, and $L_4=L_4(27)$), rectangular ($R_4(p\cdot q)$ for $pq=6,10,14,15,21$, and $R_6=R_6(18)$), $T_{k}$-graphs ($T_4$, $T_6$, and $T_8$) and the special graph of type $S$. We remark here that the degree of the maximal cyclic $\Q$-isogeny in the graph is added to the notation inside the parentheses when necessary to distinguish types, and omitted otherwise. After that, we treat linear and rectangular graphs in Sections \ref{sec-linear} and \ref{sec-rectangular}, respectively. In Section \ref{sec-T4} we classify graphs of type $T_4$ and in Section \ref{sec-morelemmas} we show a few more intermediary results. The graphs of type $T_6$ and $T_8$ are classified in Sections \ref{sec-T6graphs} and \ref{sec-T8graphs}. Finally, the special graphs of type $S$ are classified in Section \ref{sec-Sgraphs}. 
 
 We would like to stress that the question of classifying isogeny graphs and isogeny-torsion graphs depends only on the $\Q$-isogeny class and not on an individual element of the $\Q$-isogeny class. During our discussions and proofs we will explain which elliptic curve in the $\Q$-isogeny class is the easiest to focus on to generate or identify the isogeny-torsion graph but it is not required to start with said curve. In fact, generating the isogeny-torsion graph using any of the elliptic curves in the $\Q$-isogeny class will give the same isogeny-torsion graph.
 
 \begin{remark}
 	It is worth pointing out that some isogeny-torsion graphs appear for infinitely many isomorphism classes of elliptic curves over $\Q$ (for example, the $R_4$ graphs of type $([10],[10],[2],[2])$) while other isogeny-torsion graphs only occur for finitely many $j$-invariants (for example, the $L_4$ graph of type $([3],[3],[3],[1])$).  In addition, the $L_4$ graphs only occur in the case of elliptic curves with complex multiplication (as they require a rational cyclic $27$-isogeny). The classification of what isogeny-torsion graphs appear for infinitely many different $j$-invariants will be the subject of a follow up paper by the first author.
 \end{remark}
 
 \begin{ack}
 	The authors would like to express their gratitude to Harris Daniels, who helped us find the rational points on the modular curve of genus $13$ that appears in Section \ref{sec-elusive} (using a method and Magma code similar to that outlined in \cite{daniels}). The authors would also like to thank Enrique Gonz\'alez-Jim\'enez, Filip Najman, Drew Sutherland, and John Voight for helpful comments and suggestions on an earlier draft of this paper. Finally, we would like to thank the referees for a thorough reading and many helpful comments and suggestions.
 \end{ack}
 
 \section{Tables}\label{sec-tables} In this section we include tables of all the possible isogeny-torsion graphs that occur for $\Q$-isogeny classes of elliptic curves over $\Q$. In each table we include an example of a concrete isogeny class that realizes said isogeny-torsion graph. The possible $L_k$ isogeny graphs appear in Table \ref{L_k graphs} below, while the $T_k$, $R_k$, and $S$ graphs appear in Tables  \ref{T_k graphs},  \ref{R_k graphs}, and \ref{S_k graphs}, respectively. Our notation for an isogeny-torsion graph is $(G_1,\ldots,G_n)$, where $n$ is the number of vertices, with $G_i=[a]$ or $[a,b]$ which represent $\Z/a\Z$ or $\Z/a\Z\times \Z/b\Z$, respectively, and $G_i\cong E_i(\Q)_{\text{tors}}$. The numbering of the curves $E_i$ follows the numbering of the vertices in the isogeny graphs that appear in the first column of the tables. We remark here that graphs are considered up to graph isomorphism so, for example, $([2,2],[4],[2],[2])$ and $([2,2],[2],[4],[2])$ represent the same type.

 The rest of the article is devoted to prove that these tables are complete.
 
 \section{Graphs in the CM case}\label{sec-CM}
 
 In this section we quickly determine the possible isogeny-torsion graphs attached to elliptic curves with complex multiplication. Let $E/\Q$ be an elliptic curve with CM by an imaginary quadratic field $K$ of discriminant $d_K$. In \cite{enriquealvaro}, Table 1 of Section 7, one can find the degrees of the isogenies for each CM $j$-invariant over $\Q$, therefore determining the type of isogeny graph. In addition, by \cite{olson} (see also \cite{clarkcooketal}), the torsion subgroup of an elliptic curve over $\Q$ with CM is isomorphic to one of the following groups: $0$, $\Z/2\Z$, $\Z/3\Z$, $\Z/4\Z$, $\Z/6\Z$, or $\Z/2\Z\oplus \Z/2\Z$.  In particular, it follows that the possible types of isogeny graphs for CM elliptic curves are $L_2(p)$ for $p=2,3,11,19,43,67,163$, $L_4$, $T_4$, $R_4(6)$, and $R_4(14)$.   In Table \ref{tab-CMgraphs}, we have broken the case of each rational $j$-invariant with CM into subcases according to the torsion subgroup of the model, and the number of finite-degree, $\Q$-rational, cyclic, isogenies. Thus, after a finite computation, the table represents the complete list of isogeny-torsion graphs that are possible for elliptic curves with CM. In addition, we have included the LMFDB label of a $\Q$-isogeny class of an example in each category.
 
 Thus, from now on, we shall assume that our elliptic curves do not have complex multiplication.

		\begin{table}[h!]
	\renewcommand{\arraystretch}{1.2}
	\begin{tabular}{ |c|c|c|c| }
		\hline
		Isogeny Graph & Label & Isomorphism Types & LMFDB Label (Isogeny Class) \\
		
		\hline
		$E_1$ & $L_{1}$ & ([1]) & 37.a \\
		\hline
		
		\multirow{12}*{$E_1 \myiso E_2$} & {$L_{2}(2)$} & ([2],[2]) & $46.a$ \\
		\cline{2-4} 
		& \multirow{2}*{$L_{2}(3)$} & $([1],[1])$ & $196.a$ \\
		\cline{3-4}
		& & $([3],[1])$ & $44.a$ \\
		\cline{2-4} 
		& \multirow{2}*{$L_{2}(5)$} & $([1],[1])$ & $75.c$ \\
		\cline{3-4}
		& & $([5],[1])$ & $38.b$ \\
		\cline{2-4} 
		& \multirow{2}*{$L_{2}(7)$} & $([1],[1])$ & $208.d$ \\
		\cline{3-4}
		& & $([7],[1])$ & $26.b$ \\
		\cline{2-4}
		& $L_{2}(11)$ & $([1],[1])$ & $121.a$ \\
		\cline{2-4}
		& $L_{2}(13)$ & $([1],[1])$ & $147.b$ \\
		\cline{2-4}
		& $L_{2}(17)$ & $([1],[1])$ & $14450.b$ \\
		\cline{2-4}
		& $L_{2}(19)$ & $([1],[1])$ & $361.a$ \\
		\cline{2-4}
		& $L_{2}(37)$ & $([1],[1])$ & $1225.b$ \\
		\cline{2-4}
		& $L_{2}(43)$ & $([1],[1])$ & $1849.b$ \\
		\cline{2-4}
		& $L_{2}(67)$ & $([1],[1])$ & $4489.b$ \\
		\cline{2-4}
		& $L_{2}(163)$ & $([1],[1])$ & $26569.b$ \\
		\hline
		
		\multirow{5}*{$E_{1}\myiso E_{2}\myiso E_{3}$} & \multirow{3}*{$L_{3}(9)$} & $([1],[1],[1])$ & $175.b$ \\
		\cline{3-4}
		& & $([3],[3],[1])$ & $19.a$ \\
		\cline{3-4}
		& & $([9],[3],[1])$ & $54.b$ \\
		\cline{2-4}
		& \multirow{2}*{$L_{3}(25)$} & $([1],[1],[1])$ & $99.d$ \\
		\cline{3-4}
		& & $([5],[5],[1])$ & $11.a$ \\
		\hline
		
		\multirow{2}*{$E_{1} \myisot E_{2} \myisot E_{3} \myisot E_{4}$} & \multirow{2}*{$L_{4}$} & $([1],[1],[1],[1])$ & $432.e$ \\
		\cline{3-4}
		& & $([3],[3],[3],[1])$ & $27.a$ \\
		\hline
	\end{tabular}
	\caption{The list of all $L_{k}$ rational isogeny-torsion graphs}
	\label{L_k graphs}
\end{table}

 \begin{table}[h!]
 	\renewcommand{\arraystretch}{1.3}
 	\begin{tabular}{ |c|c|c|c| }
 		\hline
 		Graph Type & Label & Isomorphism Types & LMFDB Label  \\
 		\hline
 		
 		\multirow{4}*{\includegraphics[scale=0.21]{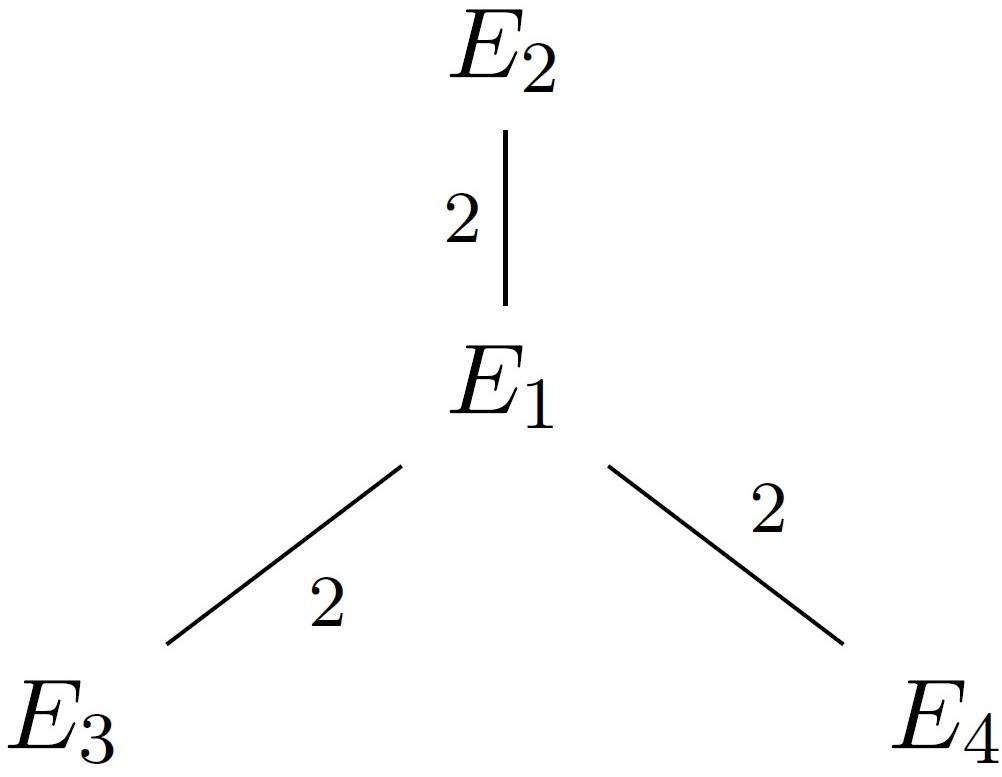}}& \multirow{4}*{$T_{4}$} & ([2,2], [2], [2], [2]) & 120.a \\
 		\cline{3-4}
 		& & ([2,2], [4], [2], [2]) & 33.a \\
 		\cline{3-4}
 		& & ([2,2], [4], [4], [2]) & 17.a \\
 		& & & \\
 		\hline
 		
 		\multirow{4}*{\includegraphics[scale=0.21]{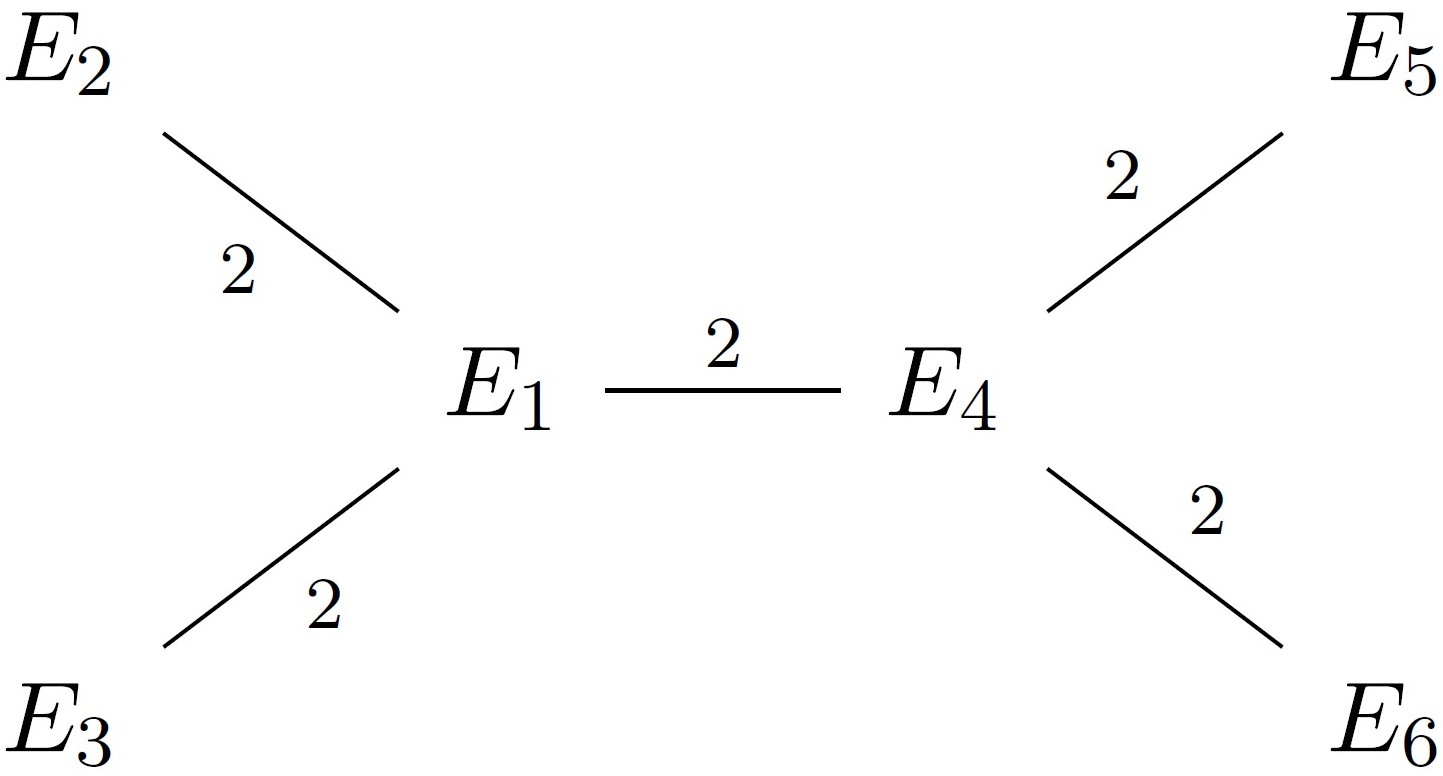}}& \multirow{4}*{$T_{6}$} & ([2,4],[4],[4],[2,2],[2],[2]) & 24.a \\
 		\cline{3-4}
 		& & ([2,4],[8],[4],[2,2],[2],[2]) & 21.a \\
 		\cline{3-4}
 		& & ([2,2],[2],[2],[2,2],[2],[2]) & 126.a \\
 		\cline{3-4}
 		& & ([2,2],[4],[2],[2,2],[2],[2]) & 63.a \\
 		\hline
 		
 		\multirow{6}*{\includegraphics[scale=0.32]{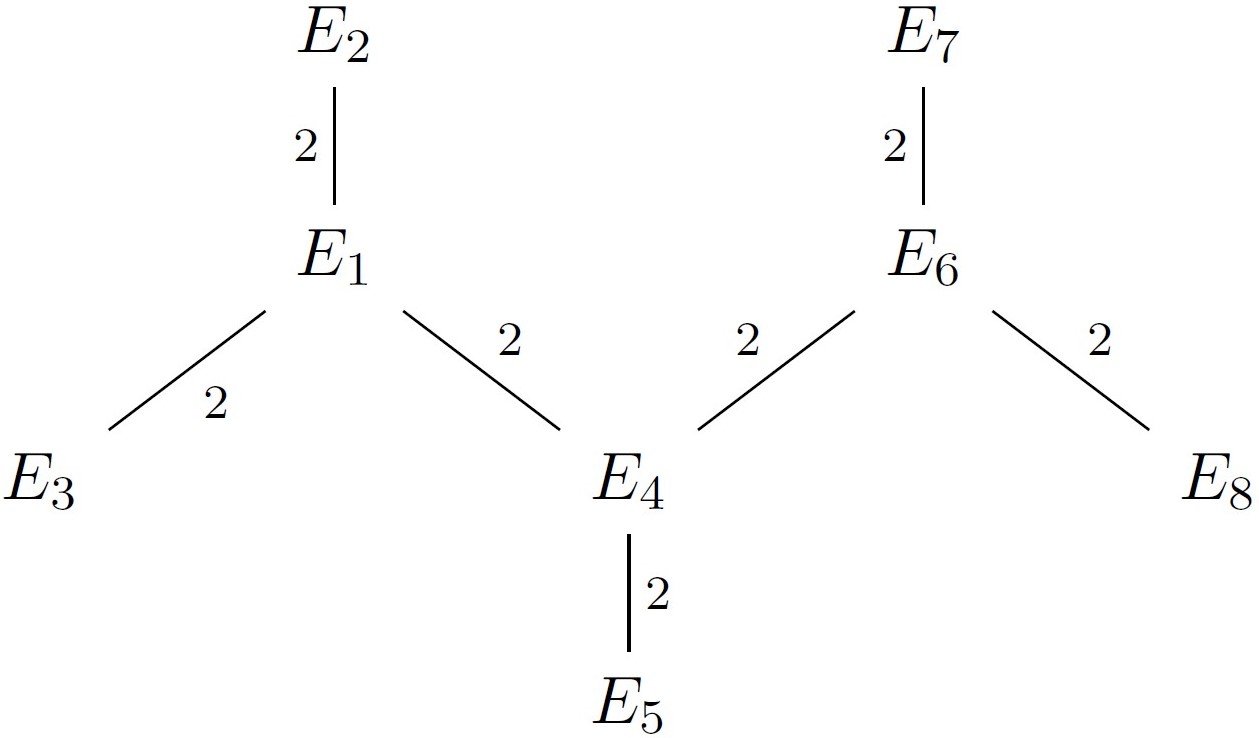}}& \multirow{6}*{$T_{8}$} & ([2,8],[8],[8],[2,4],[4],[2,2],[2],[2]) & 210.e \\
 		\cline{3-4}
 		& & ([2,4],[4],[4],[2,4],[4],[2,2],[2],[2]) & 195.a \\
 		\cline{3-4}
 		& & ([2,4],[4],[4],[2,4],[8],[2,2],[2],[2]) & 15.a \\
 		\cline{3-4}
 		& & ([2,4],[8],[4],[2,4],[4],[2,2],[2],[2]) & 1230.f \\
 		\cline{3-4}
 		& & ([2,2],[2],[2],[2,2],[2],[2,2],[2],[2]) & 45.a \\
 		\cline{3-4}
 		& & ([2,2],[4],[2],[2,2],[2],[2,2],[2],[2]) & 75.b \\
 		
 		\hline
 		
 	\end{tabular}
 	\caption{The list of all $T_{k}$ rational isogeny-torsion graphs}
 	\label{T_k graphs}
 \end{table}
 
 		\begin{table}[h!]
 	\renewcommand{\arraystretch}{1.3}
 	\begin{tabular}{ |c|c|c|c| }
 		\hline
 		Graph Type & Label & Isomorphism Types & LMFDB Label (Isogeny Class) \\
 		
 		\hline
 		\multirow{10}*{\includegraphics[scale=0.3]{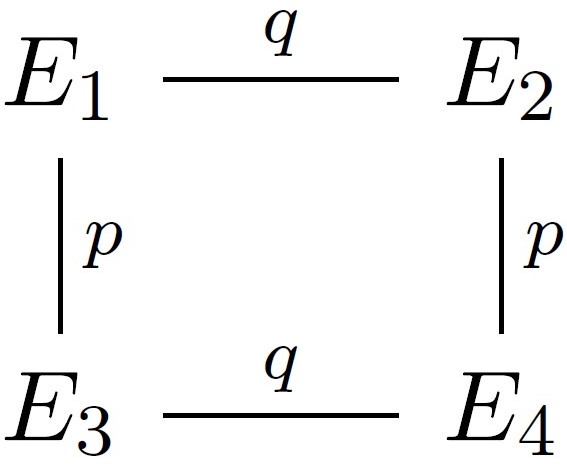}} & \multirow{2}*{$R_{4}(6)$} & $([2],[2],[2],[2])$ & $80.b$ \\
 		\cline{3-4}
 		& & $([6],[6],[2],[2])$ & $20.a$ \\
 		\cline{2-4}
 		& \multirow{2}*{$R_{4}(10)$} & $([2],[2],[2],[2])$ & $150.a$ \\
 		\cline{3-4}
 		& & $([10],[10],[2],[2])$ & $66.c$ \\
 		\cline{2-4}
 		& $R_{4}(14)$ & $([2],[2],[2],[2])$ & $49.a$ \\
 		\cline{2-4}
 		& \multirow{3}*{$R_{4}(15)$} & $([1],[1],[1],[1])$ & $400.d$ \\
 		\cline{3-4}
 		& & $([3],[3],[1],[1])$ & $50.a$ \\
 		\cline{3-4}
 		& & $([5],[5],[1],[1])$ & $50.b$ \\
 		\cline{2-4}
 		& \multirow{2}*{$R_{4}(21)$} & $([1],[1],[1],[1])$ & $1296.f$ \\
 		\cline{3-4}
 		& & $([3],[3],[1],[1])$ & $162.b$ \\
 		\hline
 		\multirow{3}*{\includegraphics[scale=0.25]{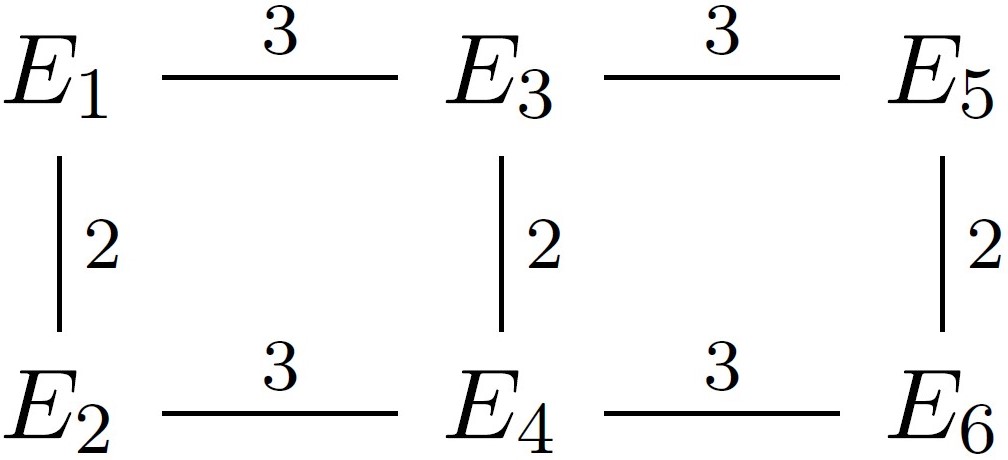}} & \multirow{3}*{$R_{6}$} & $([2],[2],[2],[2],[2],[2])$ & $98.a$ \\
 		\cline{3-4}
 		& & $([6],[6],[6],[6],[2],[2])$ & $14.a$ \\
 		& & & \\
 		\hline
 	\end{tabular}
 	\caption{The list of all $R_{k}$ rational isogeny-torsion graphs}
 	\label{R_k graphs}
 \end{table}

 \begin{table}[h!]
	\renewcommand{\arraystretch}{1.6}
	\begin{tabular} { |c|c|c|c| }
		\hline
		
		Graph Type & Label & Isomorphism Types & LMFDB Label \\
		\hline
		\multirow{4}*{\includegraphics[scale=0.25]{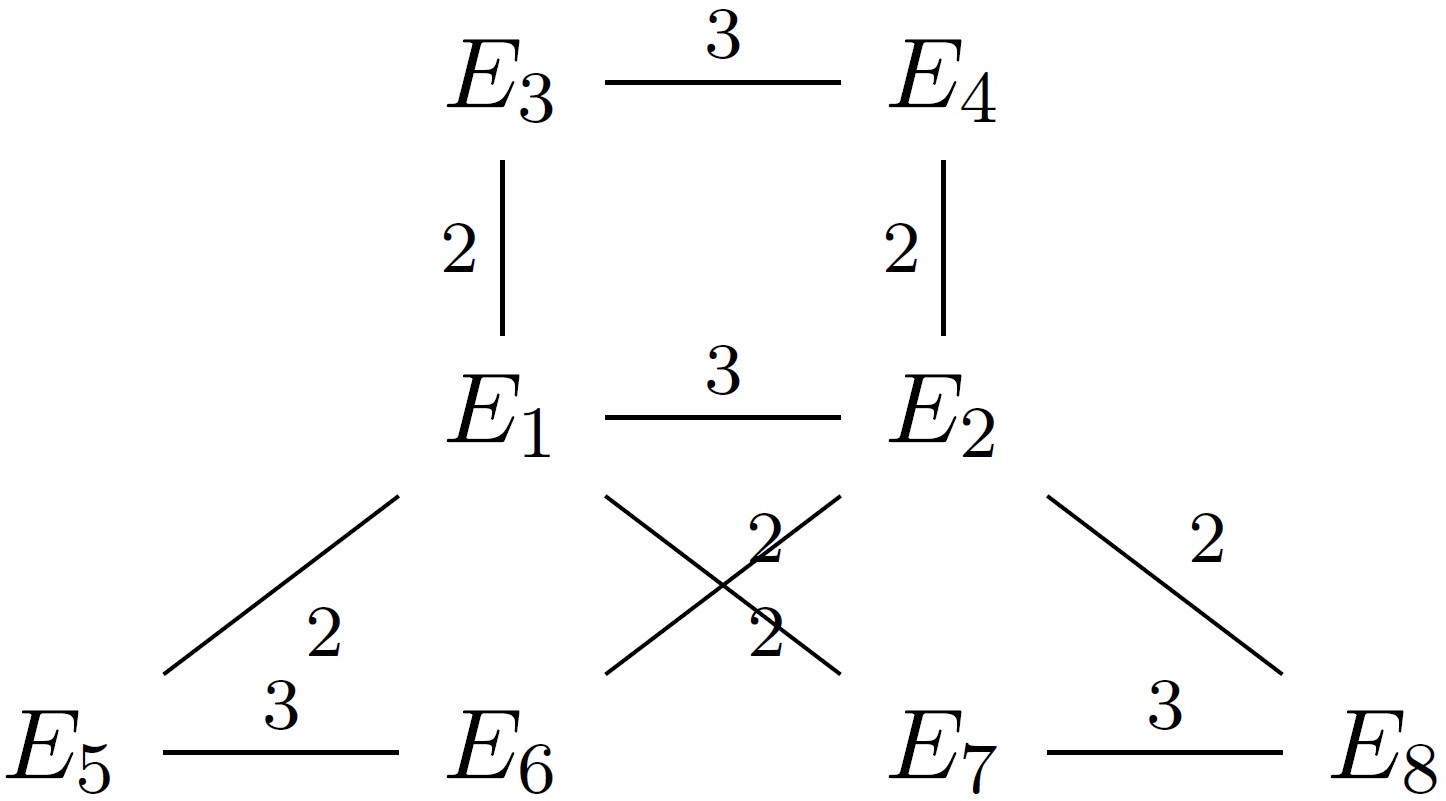}}& \multirow{4}*{$S$} & ([2,2],[2,2],[2],[2],[2],[2],[2],[2]) & 240.b \\
		\cline{3-4}
		& & ([2,2],[2,2],[4],[4],[2],[2],[2],[2]) & 150.b \\
		\cline{3-4}
		& & ([2,6],[2,2],[6],[2],[6],[2],[6],[2]) & 30.a \\
		\cline{3-4}
		& & ([2,6],[2,2],[12],[4],[6],[2],[6],[2]) & 90.c \\
		\hline
	\end{tabular}
	
	\caption{The list of all $S$ rational isogeny-torsion graphs}
	\label{S_k graphs}
\end{table}

 %\begin{center} 
 \begin{table}[h!]
 	\renewcommand{\arraystretch}{1.2}
 	\begin{tabular}{|c|c|c|c|c|c|}
 		\hline
 		$d_K$ & \multicolumn{2}{c|}{$j$} & Type & Torsion config. & LMFDB\\
 		\hline
 		\hline
 		\multirow{8}{*}{$-3$}  & \multirow{6}{*}{$0$} & $y^2=x^3+t^3, t=-3,1$  & $R_4(6)$ & $([6],[6],[2],[2])$ & \texttt{36.a4}\\
 		& &$y^2=x^3+t^3, t\neq -3,1$  & $R_4(6)$ & $([2],[2],[2],[2])$ & \texttt{144.a3}\\
 		& &$y^2=x^3+16t^3, t=-3,1$  & $L_4$ & $([3],[3],[3],[1])$ & \texttt{27.a3}\\
 		& &$y^2=x^3+16t^3, t\neq -3,1$  & $L_4$ & $([1],[1],[1],[1])$ & \texttt{432.e3}\\
 		& & $\!y^2=x^3+s^2,\,s^2\neq t^3,16t^3\!\!$  & $L_2(3)$ & $([3],[1])$ & \texttt{108.a2}\\
 		& & $\!y^2=x^3+s,\,s\neq t^3,16t^3\!\!$  & $L_2(3)$ & $([1],[1])$ & \texttt{225.c1}\\
 		\cline{2-6}
 		& \multirow{2}{*}{$2^4\cdot 3^3\cdot 5^3$} &  $y^2=x^3-15t^2x + 22t^3, t=1,3$ & $R_4(6)$ & $([6],[6],[2],[2])$ & \texttt{36.a1}\\
 		&  & $y^2=x^3-15t^2x + 22t^3, t\neq 1,3$ & $R_4(6)$ & $([2],[2],[2],[2])$ & \texttt{144.a1}\\
 		\cline{2-6}
 		& \multirow{2}{*}{$-2^{15}\cdot 3\cdot 5^3$} & $E^t, t=-3,1$ & $L_4$ & $([3],[3],[3],[1])$ & \texttt{27.a2}\\
 		& & $E^t, t\neq -3,1$  & $L_4$ & $([1],[1],[1],[1])$ & \texttt{432.e1}\\
 		\hline
 		\multirow{5}{*}{$-4$} & \multirow{4}{*}{$2^6\cdot 3^3$} & $y^2=x^3+tx, t=-1,4$  & $T_4$ & $([2,2],[4],[4],[2])$ & \texttt{32.a3}\\
 		& & $y^2=x^3+tx, t=-4,1$  & $T_4$ & $([2,2],[4],[2],[2])$ & \texttt{64.a3}\\
 		& & $y^2=x^3\pm t^2x, t\neq 1,2$  & $T_4$ & $([2,2],[2],[2],[2])$ & \texttt{288.d3}\\
 		& &  $y^2=x^3+sx$, $s\neq \pm t^2$  & $L_2(2)$ & $([2],[2])$ & \texttt{256.b1}\\
 		\cline{2-6}
 		& \multirow{3}{*}{$2^3\cdot 3^3\cdot 11^3$} & $y^2=x^3-11t^2x+14t^3, t=\pm 1$ & \multirow{3}{*}{$T_4$} & $([2,2],[4],[4],[2])$ & \texttt{32.a2}\\
 		& & $y^2=x^3-11t^2x+14t^3, t=\pm 2$ &  & $([2,2],[4],[2],[2])$ & \texttt{64.a1}\\
 		& & $y^2=x^3-11t^2x+14t^3, t\neq \pm 1, \pm 2$ &  & $([2,2],[2],[2],[2])$ & \texttt{288.d1}\\
 		\hline
 		\multirow{2}{*}{$-7$} & \multicolumn{2}{c|}{$-3^3\cdot 5^3$}  & $R_4(14)$ & $([2],[2],[2],[2])$ & \texttt{49.a2}\\
 		& \multicolumn{2}{c|}{$3^3\cdot 5^3\cdot 17^3$}  & $R_4(14)$ & $([2],[2],[2],[2])$  & \texttt{49.a1}\\
 		\hline
 		$-8$ & \multicolumn{2}{c|}{$2^6\cdot 5^3$}  & $L_2(2)$ & $([2],[2])$ & \texttt{256.a1}\\
 		\hline
 		$-11$ & \multicolumn{2}{c|}{$-2^{15}$}  & $L_2(11)$ & $([1],[1])$ & \texttt{121.b1}\\
 		\hline
 		$-19$ & \multicolumn{2}{c|}{$-2^{15}\cdot 3^3$}  & $L_2(19)$ & $([1],[1])$ & \texttt{361.a1}\\
 		\hline
 		$-43$ & \multicolumn{2}{c|}{$-2^{18}\cdot 3^3\cdot 5^3$}  & $L_2(43)$ & $([1],[1])$ & \texttt{1849.b1}\\
 		\hline
 		$-67$ & \multicolumn{2}{c|}{$-2^{15}\cdot 3^3\cdot 5^3\cdot 11^3$}  & $L_2(67)$ & $([1],[1])$ & \texttt{4489.b1}\\
 		\hline
 		$-163$ & \multicolumn{2}{c|}{$-2^{18}\cdot 3^3\cdot 5^3\cdot 23^3\cdot 29^3$}  & $L_2(163)$ & $([1],[1])$ & \texttt{26569.a1}\\
 		\hline
 	\end{tabular}
 	\caption{The list of rational $j$-invariants with CM and the possible isogeny-torsion graphs that occur, where $E^t$ denotes the curve $y^2=x^3-38880t^2x+2950992t^3$.}
 	\label{tab-CMgraphs}
 \end{table}
 %\end{center} 

	\section{Preliminaries}\label{sec-preliminaries}
	In this section we include a number of well-known results about torsion subgroups, isogenies, and Galois representations that we will use throughout the paper. We remark here that all of our elliptic curves are defined over $\Q$ and all our isogenies are $\Q$-rational and have finite cyclic kernels unless otherwise stated. First, we cite Mazur's theorem which classifies the possible torsion subgroups for an elliptic curve defined over $\Q$. 
	
	\begin{thm}[Mazur \cite{mazur1}]\label{thm-mazur}
		Let $E/\Q$ be an elliptic curve. Then
		\[
		E(\Q)_\tor\simeq
		\begin{cases}
		\Z/M\Z &\text{with}\ 1\leq M\leq 10\ \text{or}\ M=12,\ \text{or}\\
		\Z/2\Z\oplus \Z/2M\Z &\text{with}\ 1\leq M\leq 4.
		\end{cases}
		\]
	\end{thm}
	
	We will also use the following result of Kenku on the maximum number of isogenous curves over $\Q$ in an isogeny class. The notation $C(E)$ and $C_p(E)$ introduced in the statement below will be used often in our proofs. The statement below compiles a number of results shown by Kenku in \cite{kenku}. Before we cite the result, we make the following definition.
	
	\begin{defn}
		Let $E/\Q$ be an elliptic curve. We define $C(E)$ as the number of distinct finite $\Q$-rational cyclic subgroups of $E$ (including the trivial subgroup), and we define $C_p(E)$ similarly to $C(E)$ but only counting $\Q$-rational cyclic subgroups of order a power of $p$ (like in the definition of $C(E)$, this includes the trivial subgroup), for each prime $p$. 
	\end{defn}

Notice that it follows from the definition that $C(E)=\prod_p C_p(E)$. 
	
	\begin{thm}[Kenku, \cite{kenku}]\label{thm-kenku} There are at most eight $\Q$-isomorphism classes of elliptic curves in each $\Q$-isogeny class. More concretely, let $E / \Q$ be an elliptic curve, then $C(E)=\prod_p C_p(E)\leq 8$. Moreover, each factor $C_p(E)$ is bounded as follows:
		\begin{center}
			\begin{tabular}{c|ccccccccccccc}
				$p$ & $2$ & $3$ & $5$ & $7$ & $11$ & $13$ & $17$ & $19$ & $37$ & $43$ & $67$ & $163$ & \text{else}\\
				\hline 
				$C_p\leq $ & $8$ & $4$ & $3$ & $2$ & $2$ & $2$ & $2$ & $2$ & $2$ & $2$ & $2$ & $2$ & $1$.
			\end{tabular}
		\end{center}
		Moreover:
		\begin{enumerate}
			\item If $C_{p}(E) = 2$ for a prime $p$ greater than $7$, then $C_{q}(E) = 1$ for all other primes $q$. 
			\item Suppose $C_{7}(E) = 2$, then $C(E) \leq 4$. Moreover, we have $C_{3}(E) = 2$, or $C_{2}(E) = 2$, or $C(E) = 2$.
			\item $C_{5}(E) \leq 3$ and if $C_{5}(E) = 3$, then $C(E) = 3$.
			\item If $C_{5}(E) = 2$, then $C(E) \leq 4$. Moreover, either $C_{3}(E) = 2$, or $C_{2}(E) = 2$, or $C(E) = 2$. 
			\item $C_{3}(E) \leq 4$ and if $C_{3}(E) = 4$, then $C(E) = 4$. 
			\item If $C_{3}(E) = 3,$ then $C(E) \leq 6$. Moreover, $C_{2}(E) = 2$ or $C(E) = 3$.
			\item If $C_{3}(E) = 2$, then $C_{2}(E) \leq 4$.
		\end{enumerate}
	\end{thm}
	
	The next result we quote describes the possible isomorphism types of $\Gal(\Q(E[p])/\Q)$. In particular, fix a $\Z/p\Z$-basis of $E[p]$, and let $\rho_{E,p}\colon \Gal(\overline{\Q}/\Q)\to \GL(E[p])\cong \GL(2,\Z/p\Z)$ be the representation associated to the natural Galois action on $E[p]$, with respect to the chosen basis of $E[p]$. Then, $\Gal(\Q(E[p])/\Q)\cong \rho_{E,p}(\GQ)\subseteq \GL(E[p])\cong \GL(2,\Z/p\Z)$.
	
	\begin{thm}[Serre, \cite{serre1}]\label{thm-serre2}
		Let $E/\Q$ be an elliptic curve. Let $G$ be the image of $\rho_{E,p}$, and suppose $G\neq \GL(E[p])$. Then, there is a $\Z/p\Z$-basis of $E[p]$ such that one of the following possibilities holds:
		\begin{enumerate}
			\item $G$ is contained in the normalizer of a split Cartan subgroup of $\GL(E[p])$.
			\item $G$ is contained in the normalizer of a non-split Cartan subgroup of $\GL(E[p])$.
			\item The projective image of $G$ in $\PGL(E[p])$ is isomorphic to $A_4$, $S_4$ or $A_5$, where $S_n$ is the symmetric group and $A_n$ the alternating group (note: only $S_4$ occurs over $\Q$).
			\item $G$ is contained in a Borel subgroup of $\GL(E[p])$.
		\end{enumerate}
	\end{thm}
	
	Rouse and Zureick-Brown have classified all the possible $2$-adic images of $\rho_{E,2}\colon \GQ\to \GL(2,\Z_2)$ (see previous results of Dokchitser and Dokchitser \cite{dok} on the surjectivity of $\rho_{E,2} \bmod 2^n$). 
	
	\begin{thm}[Rouse, Zureick-Brown, \cite{rouse}]\label{thm-rzb} Let $E$ be an elliptic curve over $\Q$ without complex multiplication. Then, there are exactly $1208$ possibilities for the $2$-adic image $\rho_{E,2^\infty}(\GQ)$, up to conjugacy in $\GL(2,\Z_2)$. Moreover, the index of $\rho_{E,2^\infty}(\Gal(\overline{\Q}/\Q))$ in $\GL(2,\Z_2)$ divides $64$ or $96$.
	\end{thm}
	
%	\begin{conj}[Sutherland, Zywina, \cite{sutherland}]\label{conj-zyw} Let $E/\Q$ be an elliptic curve. Let $G$ be the image of $\rho_{E,p}$. Then, there are precisely $63$ isomorphism types of images. 
%	\end{conj}
	
	The $\Q$-rational points on the modular curves $X_0(N)$ have been described completely in the literature, for all $N\geq 1$. One of the most important milestones in their classification was \cite{mazur1}, where Mazur dealt with the case when $N$ is prime. The complete classification of $\Q$-rational points on $X_0(N)$, for any $N$, was completed due to work by Fricke, Kenku, Klein, Kubert, Ligozat, Mazur  and Ogg, among others (see the summary tables in \cite{lozano0}).
	
	\begin{thm}\label{thm-ratnoncusps} Let $N\geq 2$ be a number such that $X_0(N)$ has a non-cuspidal $\Q$-rational point. Then:
		\begin{enumerate}
			\item $N\leq 10$, or $N= 12,13, 16,18$ or $25$. In this case $X_0(N)$ is a curve of genus $0$ and its $\Q$-rational points form an infinite $1$-parameter family, or
			\item $N=11,14,15,17,19,21$, or $27$. In this case $X_0(N)$ is a curve of genus $1$, i.e.,~$X_0(N)$ is an elliptic curve over $\Q$, but in all cases the Mordell-Weil group $X_0(N)(\Q)$ is finite, or 
			
			\item $N=37,43,67$ or $163$. In this case $X_0(N)$ is a curve of genus $\geq 2$ and (by Faltings' theorem) there are only finitely many $\Q$-rational points, which are known explicitly.
		\end{enumerate} 
	\end{thm}

	\section{Lemmas}\label{sec-lemmas}
	
	In this section we show a number of elementary preliminary lemmas that will be used in the following sections to prove the classification of isogeny-torsion graphs.
	
	\begin{lemma}\label{lem-number-of-groups} Let $p$ be a prime and let $E / \Q$ be an elliptic curve. Let $P \in E$ be a torsion point of order $p^m$ for some $m \geq 1$. Then, there are $p$ cyclic subgroups of $E[p^{m+1}]$ of order $p^{m + 1}$ containing $P$. Additionally, there are $p + 1$ subgroups of $E[p^{m+1}]$ of order $p$.
	\end{lemma}
	% Statement and proof revised by Alvaro 8/22/2019  
	\begin{proof}
		Let $P_k = [p^{m-k}]P_m$ with $P_m=P$ as in the statement. Let $P_{m+1}$ be a torsion point of order $p^{m+1}$ such that $[p]P_{m+1}=P_m$, and let $Q_{m+1}$ be another torsion point of order $p^{m+1}$ such that $\{P_{m+1},Q_{m+1}\}$ form a $\Z/p^{m+1}\Z$-basis of $E[p^{m+1}]$, and let $Q_m=[p]Q_{m+1}$.     Then, the subgroups $H_k = \langle P_{m+1}+[kp^{m}]Q_{m+1}\rangle \subset E[p^{m+1}]$, for $k=0,\ldots,p-1$, are of order $p^{m+1}$, and clearly contain the point $P_m$.
		
		First note that if $P_m = [p]([a]P_{m+1}+[b]Q_{m+1})=[a]P_m + [b]Q_m$, then $a\equiv 1 \bmod p^m$ and $b\equiv 0 \bmod p^m$. Thus if $a'\equiv a^{-1} \bmod p^{m+1}$, and $\frac{a'b}{p^m} \equiv k \bmod p$, then $[a']([a]P_{m+1}+[b]Q_{m+1}) \in H_k$. It follows that if $H$ is cyclic of order $p^{m+1}$ and $P_m\in H$, then $H=H_k$ for some $k\in \{0,\ldots,p-1\}$.
		
		Finally, suppose $H_k = H_j$, for some $0\leq k,j\leq p-1$. Then, there is some $h\geq 1$ such that $[h](P_{m+1}+[kp^m]Q_{m+1})= P_{m+1}+[jp^m]Q_{m+1}$. Since $\{P_{m+1},Q_{m+1}\}$ is a $\Z/p^{m+1}\Z$-basis, it follows that $h\equiv 1\bmod p^{m+1}$ and therefore $h\equiv j\bmod p$. This implies that $h=j$, as desired. Hence, we have shown that there are exactly $p$ cyclic subgroups of order $p^{m+1}$ containing $P$, namely $H_0,\ldots,H_{p-1}$. 
		
		The last statement follows from the fact that if $H\subseteq E[p^{m+1}]$ is of order $p$, then $H\subset E[p]\cong \Z/p\Z \times \Z/p\Z$, which has $p+1$ subgroups of order $p$.
	\end{proof}
	
	\begin{lemma}\label{lem-orbit-of-rational-points}   Let $E / \Q$ be an elliptic curve.
		\begin{enumerate}
			\item If $P,Q \in E(\Q)$ with $\tau(Q) =  P$ for some $\tau \in G_{\Q}=\Gal(\overline{\Q}/\Q)$, then $Q = P$. 
			
			\item A subgroup $H\subseteq E$ of order $2$ is $\Q$-rational if and only if $H\subseteq E(\Q)$. 
			
			\item Moreover, if $P \in E(\Q)[2]$ is a rational point of order 2 and $Q \in E[2]$ is a different point of order 2 (not necessarily rational), then $\sigma(Q) = Q$ or $P + Q$ for all $\sigma \in G_{\Q}$. 
		\end{enumerate}
	\end{lemma}
	% Statement and proof revised by Alvaro 8/22/2019   
	\begin{proof}
		Let $E$, $P$, $Q$, and $H$ be as in the statement. If $P\in E(\Q)$ and $\tau \in G_\Q$ with $\tau(Q)=P$, then $P=\tau^{-1}(P)=Q$, because $\tau^{-1}$ fixes $P\in E(\Q)$. Thus, $P=Q$ as claimed in part (1). For part (2), notice that if $H=\langle P \rangle = \{\mathcal{O},P\}$ is $\Q$-rational, then $\sigma(P)=P$ for all $\sigma \in G_\Q$ because $P$ is the only point of order $2$ in $H$. Thus, $P\in E(\Q)$, and $H\subset E(\Q)$.  For (3), suppose $P \in E(\Q)[2]$ of order $2$, and $Q$ is a different point of order $2$. Then, $E[2] = \langle P, Q\rangle$. Now let $\sigma \in G$ be arbitrary. Then, $\sigma(Q)$ is another point of order $2$, and therefore $\sigma(Q)\in \{P,Q,P+Q\}$.  However, $\sigma(Q)\neq P$, by part (1). Thus, $\sigma(Q_2)=Q_2$ or $P_2+Q_2$, as claimed.
	\end{proof}

	\begin{lemma}\label{lem-subgps-of-rat-groups} Let $E / \Q$ be an elliptic curve and suppose $P \in E$ generates a $\Q$-rational subgroup. If $Q\in \langle P\rangle$, then $\langle Q \rangle$ is also $\Q$-rational. 
	\end{lemma} 
	\begin{proof}
		Let $Q\in \langle P\rangle$, where $\langle P\rangle$ is $\Q$-rational. Then, $Q=[n]P$ for some $n \in \Z$. Let  $\sigma\in G_\Q$ be arbitrary. Since $\langle P\rangle$ is $\Q$-rational, there is an integer $m \in \Z$ such that $\sigma(P)=[m]P$. Then, $\sigma(Q)=\sigma([n]P)=[n]\sigma(P)=[n]([m]P)=[m]([n]P)=[m]Q\in \langle Q\rangle$, and therefore $\langle Q \rangle$ is $\Q$-rational as well, as claimed.
	\end{proof} 
	
	\begin{lemma}\label{orbit-of-supergroups-of-rational-groups}
		Let $M, N \geq 1$ such that $M$ divides $N$. Let $E/\Q$ be an elliptic curve such that $P_{M}, P_{N} \in E$ are elements of order $M$ and $N$ respectively and $P_{M} \in \langle P_{N}\rangle $. Suppose $\langle P_{M}\rangle$ is $\Q$-rational, then for each $\sigma \in G_{\Q}$ there is $c=c(\sigma)\in\Z/N\Z$ such that $\sigma(P_{N}) - [c]P_{N} \in E[\frac{N}{M}]$. If $\langle P_{N} \rangle$ is $\Q$-rational, then $\sigma(P_{N}) - [c]P_{N} \in \langle P_{N} \rangle \cap E[\frac{N}{M}] = \langle [M]P_{N} \rangle$. If $P_{M}$ is defined over $\Q$, then $c = 1$ for all $\sigma \in G_{\Q}$.
	\end{lemma}
	
	\begin{proof}
		For the first statement, let $E/\Q$ be an elliptic curve, $N \geq 1$, and $P \in E$ be a point of order $N$. Suppose $Q_{N} \in E$ such that $\{P_{N}, Q_{N} \}$ is a $\Z/N\Z$-basis of $E[N]$. Then, given an arbitrary $\sigma \in G_{\Q}$, we have that $\sigma(P_{N}) \in E[N]$, so $\sigma(P_{N}) = [a]P_{N} + [b]Q_{N}$ for some $a, b \in \Z / N \Z$. Suppose $M$ divides $N$ and that $P_{M} \in \langle P_{N} \rangle$ is an element of order $M$ that generates a $\Q$-rational group. After reassigning labels for the generators of $\langle P_{M} \rangle,$ we can assume without loss of generality that $P_{M} = [\frac{N}{M}]P_{N}.$ Since $\langle P_M\rangle$ is $\Q$-rational, it follows that $\sigma(P_{M}) = [c]P_{M}$ for some $c \in \Z / M \Z$. Thus, 
		$$[c]P_{M} = \sigma(P_{M}) = \sigma\left(\left[\frac{N}{M}\right]P_{N}\right) = \left[\frac{N}{M}\right]\sigma(P_{N}) = \left[\frac{N}{M}\right]([a]P_{N} + [b]Q_{N}) = [a]P_{M} + [b]\left[\frac{N}{M}\right]Q_{N}$$ 
		for some $a, b \in \Z / N \Z$, as above. For ease of notation, let us denote $[\frac{N}{M}]Q_{N} = Q_{M}$ (note that $E[M] = \langle P_{M}, Q_{M} \rangle$). Thus, $[c]P_{M} = [a]P_{M} + [b]Q_{M}$ and therefore $[a - c]P_{M} + [b]Q_{M} = \mathcal{O}$. Thus, $a-c\equiv b\equiv 0\bmod M$. Let us denote $a = c + c'M$ and $b = b'M$ for some $c', b' \in \Z$. Substituting above yields
		$$\sigma(P_{N}) = [c + c'M]P_{N} + [b'M]Q_{N} = [c]P_N + [c'][M]P_N + [b'][M]Q_N.$$
		Hence, $\sigma(P_N) - [c]P_N \in E[N/M]$ as claimed, because $[M]E[N] = E[N/M]$.
		
		The last two statements are clear because if $\langle P_{N} \rangle$ is $\Q$-rational, $\sigma(P_{N}) \in \langle P_{N} \rangle$ for all $\sigma \in G_{\Q}$ and if $P_{M}$ is defined over $\Q$, we must have $\sigma(P_{M}) = P_{M}$ for all $\sigma \in G_{\Q}$.
	\end{proof}
	
%	\begin{lemma} \label
%		Let $E / \Q$ be an elliptic curve. Let $H$ be a finite $\Q$-rational subgroup of $E$. Then, there is a unique rational elliptic curve $E' / \Q$ and a unique non-constant rational isogeny $\phi \colon E \longrightarrow E'$ with kernel $H$.
%	\end{lemma}
%	
%%		See \cite{Silverman}, Chapter III, Section 4, Proposition 4.12.
%	\end{proof}
	
%	We will use the lemma above  extensively in the paper. All our elliptic curves will be defined over $\Q$ and all our isogenies will be induced by such finite $\Q$-rational subgroups.
	
	\begin{lemma}\label{lem-necessity-for-point-rationality} Let $E / \Q$ be an elliptic curve and let $Q \in E$ such that $\langle Q \rangle$ is a $\Q$-rational subgroup of $E$.
		\begin{enumerate}
			
			\item Let $\phi \colon E \to E'$ be an isogeny with kernel a finite, cyclic, $\Q$-rational group, $H$. Then, for an arbitrary $P \in E$, the point $\phi(P)\in E'$ is defined over $\Q$ if and only if $\sigma(P) - P \in H$ for all $\sigma \in G_{\Q}$. 
			
			\item Let $E[2] = \langle P, Q \rangle$. If $\phi\colon E\to E'$ is a $\Q$-rational isogeny with cyclic kernel, and $P \in \ker(\phi)$, then $\phi(Q)\in E'(\Q)[2]$ is rational of order $2$. 
			
			\item Finally, if $\phi\colon E\to E'$ is an isogeny defined over $\Q$, then $\phi(E(\Q))\subseteq E'(\Q)$, i.e., the isogeny $\phi$ sends rational points of $E$ to rational points of $E'$.
		\end{enumerate}
	\end{lemma}
	
	\begin{proof}
		
		Let $E / \Q$, $H$, and $\phi$ be as in (1). Then, $P \in E$ is such that $\phi(P)$ is defined over $\Q$ if and only if, for all $\sigma \in G_{\Q}$, we have $\phi(P) = \sigma(\phi(P)) = \phi(\sigma(P))$. Equivalently, $\phi(\sigma(P) - P) = \mathcal{O}$, and therefore $\sigma(P) - P \in H$ for all $\sigma \in G_{\Q}$. This shows (1).
		
		With notation as in (2), if $P \in \ker(\phi)$, then $\langle P\rangle$ is $\Q$-rational by Lemma \ref{lem-subgps-of-rat-groups}, and since $P\in E[2]$, we must have $P$ is defined over the rationals by Lemma \ref{lem-orbit-of-rational-points}, part (2). Now, the point $\phi(Q)$ is of order dividing $2$ (because $Q\in E[2]$), but $\phi(Q)\neq \mathcal{O}$ because $Q\not\in \ker(\phi)$ as $P\neq Q$ and the kernel is cyclic by assumption. Thus, $\phi(Q)$ has exact order $2$, and $\sigma(Q) - Q \in \langle P \rangle \subseteq \ker(\phi)$ by Lemma \ref{lem-orbit-of-rational-points} (3). Hence, $\phi(Q)$ is defined over $\Q$ by the first part of Lemma \ref{lem-necessity-for-point-rationality}.
		
		The last statement, part (3), follows immediately from the fact that $\phi$ is defined over $\Q$.\end{proof}
	
	\begin{lemma} \label{lem-2torspt-all-have-2torspt} Let $E/\Q$ be an elliptic curve with a point of order $2$ defined over $\Q$. Then, every elliptic curve over $\Q$ that is $\Q$-isogenous (with cyclic kernel) to $E$ also has a point of order $2$ defined over $\Q$.
	\end{lemma} 
	
	\begin{proof}
		Suppose that $E/\Q$ has a point of order $2$ defined over $\Q$. Denote it $P_{2}$ and suppose $E[2] = \langle P_{2}, Q_{2} \rangle$ for some $Q_{2} \in E[2]$. Let $\phi\colon E\to E'$ be a $\Q$-rational isogeny with a finite, cyclic, $\Q$-rational kernel. If $P_2 \in \ker(\phi)$, then by the second statement of Lemma \ref{lem-necessity-for-point-rationality}, we have that $\phi(Q_2)$ is a point of order $2$ defined over $\Q$ in $E'$. Otherwise, if $P_2\not\in \ker(\phi)$, then by the third statement of Lemma \ref{lem-necessity-for-point-rationality}, we have that $\phi(P_2)$ is point of order $2$ defined over $\Q$ in $E'$. Thus, $E'$ has a point of order $2$ defined over $\Q$ as well.  
	\end{proof}

	\begin{lemma}\label{lem-necessity-for-subgroup-rationality} Let $E / \Q$ be an elliptic curve. Let $P \in E$ such that $\langle P \rangle$ contains a finite, $\Q$-rational, cyclic group, $H$. Let $\phi \colon E \to E'$ be an isogeny with kernel $H$. If $\phi(P)$ is defined over $\Q$, then $\langle P \rangle$ is $\Q$-rational.
	\end{lemma}
	
	\begin{proof}
		Let $E$, $P$, and $H$ be as in the statement of the lemma. Let $\phi \colon E \to E'$ be an isogeny with kernel $H$. Then, by the first statement of Lemma \ref{lem-necessity-for-point-rationality}, the point $\phi(P)\in E'$ is defined over $\Q$ if and only if we have $\sigma(P) - P \in H$, for all $\sigma \in G_{\Q}$. As $H \subseteq \langle P \rangle$, we have $\sigma(P) - P \in \langle P \rangle$ for all $\sigma \in G_{\Q}$ and thus, $\langle P \rangle$ is a $\Q$-rational group.
	\end{proof}
	
	\begin{lemma} \label{p^2 isogeny}
		Let $E / \Q$ be an elliptic curve and let $p = 2, 3,$ or $5$. Suppose that $E$ has a point of order $p$ defined over $\Q$. Further, suppose that the image of $\rho_{E,p}$ is a split Cartan subgroup of $\GL(2,\mathbb{F}_p)$. Then, $E$ is isogenous to an elliptic curve, $E'$ over $\Q$ which has a point of order $p$ defined over $\Q$, which is contained in a cyclic $\Q$-rational subgroup of $E'$ of order $p^2$.
	\end{lemma}
	
	\begin{proof}
		Let $E / \Q$ be an elliptic curve, let $p$ be a prime as above, and let $P\in E[p]$ be a point of order $p$ defined over $\Q$. Choose a basis $\{P,Q\}$ of $E[p]$ such that the image of $\rho_{E,p}$ is a split Cartan subgroup $\left\{\left(\begin{array}{cc}
		1 & 0 \\
		0 & \ast
		\end{array} \right) \right\} \subseteq \GL(2,\mathbb{F}_p)$. Thus, $\langle Q \rangle$ is a $\Q$-rational subgroup of order $p$ that intersects $\langle P \rangle$ trivially. Let $\phi \colon E \to E'$ be the isogeny with kernel $\langle Q \rangle$. Then, $\phi(P)$ is a point of order $p$ because $P \in E[p]$ and $P \not \in \langle Q \rangle$. By the third statement of Lemma  \ref{lem-necessity-for-point-rationality}, $\phi(P)$ is defined over $\Q$. Let $P' \in E[p^2]$ be an element such that $[p]P' = P$. Letting $N = p^2$ and $M = p$ in Lemma  \ref{orbit-of-supergroups-of-rational-groups}, we have $\sigma(P') - P' \in E[p] = \langle P, Q \rangle$ for all $\sigma \in G_{\Q}$. Thus, for an arbitrary $\sigma \in G_{\Q}$, we have $\sigma(\phi(P')) = \phi(\sigma(P')) = \phi(P' + [a]P + [b]Q)$ for some $a, b \in \Z / p \Z$. But as $Q$ is an element of the kernel of $\phi$, we have $\phi(\sigma(P')) = \phi(P' + [a]P) = \phi(P' + [ap]P') = \phi([1 + ap]P') = [1 + ap]\phi(P')$. Thus, for all $\sigma \in G_{\Q}$, we have $\sigma(\phi(P')) \in \langle \phi(P') \rangle$. Hence, $\langle \phi(P')\rangle$ is a cyclic $\Q$-rational subgroup of $E'$ of order $p^2$ that contains $\phi(P)$, as desired.\end{proof}
	
	\begin{remark} \label{p^2 remark} When necessary, we will use Lemma \ref{p^2 isogeny} above to ``reorient'' our isogeny-torsion graph to use the most convenient elliptic curve as our base point for classification.
		
		In other words, if every torsion subgroup that appears in the isogeny-torsion graph is cyclic, then we will start building the isogeny-torsion graph using the elliptic curve whose torsion subgroup is the greatest in order among the torsion subgroups of elements of the isogeny class. If there are two or more such elements in the isogeny class, then we will build the isogeny-torsion graph using an element whose torsion subgroup is greatest in order among the torsion subgroups of the elements of the isogeny class \textit{and} with the greatest, finite, $\Q$-rational subgroup by order among the elements of the isogeny class. Note that then each isogeny-torsion graph with exclusively cyclic rational torsion has one or two ideal base points.
		
		If the isogeny class contains an element with bicyclic rational torsion subgroup, then we will start building our graph using the elliptic curve with the largest bicyclic rational torsion subgroup. If there are two such elements in the isogeny class, then we will build the isogeny-torsion graph using an element whose torsion subgroup is greatest in order among the bicyclic torsion subgroups of the elements of the isogeny class \textit{and} with the largest, finite, $\Q$-rational subgroup by order among the elements of the isogeny class with bicyclic torsion. Note that then each isogeny-torsion graph that contains an elliptic curve with bicyclic rational torsion contains one or two ideal base points.
		
	\end{remark}
	
	\begin{lemma} \label{lem-determinants}
		Let $E / \Q$  be an elliptic curve and fix $M \geq 3.$ Let $E[M] = \langle P, Q \rangle$ and suppose $Q$ is defined over $\Q$. Then, there exists some $\sigma \in G_{\Q}$ such that $\sigma(P) = [a]P + [b]Q$ for some $a,b\in\Z/M\Z$ with  $a\not\equiv 1\bmod M$.
	\end{lemma}
	
	\begin{proof}
		Let  $E / \Q$ be an elliptic curve, and fix $M \geq 3$. Let $E[M] = \langle P, Q \rangle$ and suppose $Q$ is defined over $\Q$. Let $\sigma \in G_{\Q}$ be arbitrary. Then, $\sigma(P) = [a]P + [b]Q$ for some $a, b \in \Z / M \Z.$ If $a = 1$ for all $\sigma \in G_{\Q}$, then the image of $\rho_{E,M}$ in $\GL(2,\Z/M\Z)$ is of the form $\left\{\left(\begin{array}{cc} 1 & 0 \\ \ast & 1 \end{array} \right) \right\}$ and thus, the determinant of $\rho_{E,M}$ is constant equal to $1$. This contradicts the fact that $\det(\rho_{E,M})$ is surjective onto $(\Z/M\Z)^\times$, by the properties of the Weil pairing.
	\end{proof}
	
	\begin{lemma}\label{Maximality-Of-Rational-2-Power-Groups} Let $E / \Q$ be an elliptic curve and let $P$ be an element of $E$ of order $2^{M}$ with $M \geq 1$. Suppose $P$ generates a $\Q$-rational group and the two cyclic groups of order $2^{M+1}$ that contain $P$ are not $\Q$-rational. Let $\phi \colon E \to E'$ be an isogeny with kernel $\langle P \rangle$. Then, $E'(\Q)_{\text{tors}}$ is cyclic.
	\end{lemma}
	
	\begin{proof}
		Let $E / \Q$ and $P \in E$ be as in the statement (note that Lemma \ref{lem-number-of-groups} says there are two cyclic subgroups of $E[2^{M+1}]$ that contain $\langle P \rangle$). Let $\phi \colon E \to E'$ be an isogeny with kernel $\langle P \rangle$. By Mazur's theorem on the possible torsion subgroups of elliptic curves over $\Q$, in order to show that $E'(\Q)_{\text{tors}}$ is cyclic, it suffices to show that there is a $2$-torsion point on $E'$ that is not defined over $\Q$. Suppose $P' \in E$ with $[2]P' = P$. Then, $\phi(P')\in E'$ has order $2$ but it is not rational as if it were, then by Lemma \ref{lem-necessity-for-subgroup-rationality}, the subgroup $\langle P' \rangle$ would be $\Q$-rational, contradicting our hypothesis. Thus, $E'(\Q)_{\text{tors}}$ is cyclic, as claimed.
	\end{proof}
	
	\begin{lemma}\label{lem-4tors-implies-bicyclic-isog} Let $E / \Q$ be an elliptic curve with a cyclic $\Q$-rational subgroup of order $4$. Then, $E$ is $\Q$-isogenous to a curve with full two-torsion.
	\end{lemma}
	
	\begin{proof}
		Let $E / \Q$ be an elliptic curve, and $\langle P_{4}\rangle$ be a $\Q$-rational subgroup of $E$ of order 4. Then, the subgroup of $\langle P_{4} \rangle$ of order 2, namely $\langle [2]P_4\rangle$, is $\Q$-rational by Lemma \ref{lem-subgps-of-rat-groups} and thus, $P_{2}=[2]P_4$ is defined over $\Q$ by Lemma \ref{lem-orbit-of-rational-points}. Let $\phi \colon E \to E'$ be the isogeny with kernel $\langle P_{2}\rangle$. Then, we claim that $E'$ has full $2$-torsion defined over $\Q$, generated by $\phi(P_4)$ and $\phi(Q_2)$, where $E[2]=\langle P_2,Q_2\rangle$.

		Since $\langle P_{4} \rangle$ is $\Q$-rational and $P_2\in E(\Q)$, it follows that $\sigma(P_{4}) - P_{4} \in \langle P_{4} \rangle \cap E[2] = \langle P_2\rangle$ for all $\sigma \in G_{\Q}$, by Lemma \ref{orbit-of-supergroups-of-rational-groups}. The first statement of Lemma \ref{lem-necessity-for-point-rationality} then says that $\phi(P_{4})$ is defined over $\Q$. By the second statement of Lemma \ref{lem-necessity-for-point-rationality}, $\phi(Q_{2})$ is a point of order $2$ defined over $\Q$.
		
		Finally, $\phi(Q_{2})$ is not equal to $\phi(P_{4})$ as if it were, then $\phi(P_{4}) - \phi(Q_{2}) = \phi(P_{4} - Q_{2}) = \mathcal{O}$ so $P_{4} - Q_{2} \in \langle P_{2}\rangle$ but $P_{4} - Q_{2}$ has order $4$, a contradiction. This concludes the proof of the lemma.
	\end{proof}
	
	\begin{lemma}\label{lem-Q-Rational-Grps-Order-8} Let $E / \Q$ be an elliptic curve and let $P_{8} \in E(\Q)[8]$ be a point of order $8$. Then, $E$ is isogenous to an elliptic curve with torsion subgroup $\Z / 2 \Z \times \Z / 4 \Z$.
	\end{lemma} 
	
	\begin{proof}
		Let $E / \Q$ be a curve with a point $P_8$ of order $8$ defined over $\Q$. Thus, $P_{2} = [4]P_{8}$ is defined over $\Q$. Suppose $E[2] = \langle P_{2}, Q_{2} \rangle$ for some $Q_{2} \in E[2]$. Let $\phi \colon E \to E'$ be an isogeny with kernel $\langle P_{2} \rangle$. Then, $\phi(P_{8})\in E'$ is a rational point of order $4$ because $\phi$ is a rational isogeny. Also,  $\phi(Q_{2})\in E'$ is a rational point of order $2$ by Lemma \ref{lem-necessity-for-point-rationality}. Moreover, $\phi(Q_{2})$ is not equal to $[2]\phi(P_{8}) = \phi([2]P_{8})$ as if it were, then $\phi([2]P_{8}) - \phi(Q_{2}) = \phi([2]P_{8} - Q_{2}) = O$ so $[2]P_{8} - Q_{2} \in \langle P_{2} \rangle$ but $[2]P_{8} - Q_{2}$ is of order 4. So $E'$ has full two-torsion and a rational point of order $4$, and the result follows.  \end{proof}
	
	\begin{remark} Lemmas \ref{lem-4tors-implies-bicyclic-isog} and \ref{lem-Q-Rational-Grps-Order-8} show that classifying isogeny graphs with a curve with $E(\Q)_{\text{tors}} = \Z / 2 \Z \times \Z / 2 \Z$ or $\Z / 2 \Z \times \Z / 4 \Z$ is enough to classify the isogeny graphs with a curve with $E(\Q)_{\text{tors}} = \Z / 4 \Z$ or $\Z / 8 \Z$ and classifying isogeny graphs with a curve with $E(\Q)_{\text{tors}} = \Z / 2 \Z \times \Z / 6 \Z$ is enough to classify the isogeny graphs with a curve with $E(\Q)_{\text{tors}} = \Z / 12 \Z$. \end{remark}
	
	\begin{lemma}\label{lem-rectangle-lemma} Let $E / \Q$ be an elliptic curve, let $p,q$ be distinct primes, and let $P,Q,R$ be non-zero points on $E$, such that $P$ and $R$ are of $p$-power order, and $Q$ is of $q$-power order, such that $\langle P + Q \rangle$ is a $\Q$-rational subgroup of $E$. Let $\phi \colon E \to E'$ be a $\Q$-isogeny with kernel $\langle P+Q \rangle$, and let $\psi \colon E \to E''$ be a $\Q$-isogeny with kernel $\langle P\rangle$.  Then: 
		\begin{enumerate}
			\item The point $\phi(R)$ is defined over $\Q$ if and only if $\psi(R)$ is defined over $\Q$.
			\item The order of $\phi(R)$ is equal to the order of $\psi(R)$. In particular, $\phi(R)$ is non-zero of $p$-power order if and only if $\psi(R)$ is non-zero of $p$-power order.
		\end{enumerate}
	\end{lemma} 
	
	\begin{proof}
		Let $P,Q,R$ be as in the statement of the lemma. Note that if $\langle P+Q\rangle$ is a $\Q$-rational subgroup, so is $\langle P\rangle$ by Lemma \ref{lem-subgps-of-rat-groups}.  Also, we note that $\langle P\rangle \subseteq \langle P+Q\rangle$ is the largest subgroup of $p$-power order. Moreover, $\phi(R)$ is rational if and only if $\sigma(R) - R \in \langle P + Q \rangle$ by \ref{lem-necessity-for-point-rationality}, for all $\sigma \in G_\Q$.  Since $\sigma(R)-R$ is $p$-power order (because both $R$ and $\sigma(R)$ are), it follows that $\sigma(R)-R \in \langle P+Q \rangle$ if and only if $\sigma(R)-R\in \langle P\rangle$. Hence, we conclude that $\phi(R)$ is rational if and only if $\psi(R)$ is rational. This proves (1).
		
		For (2), we note that:
		\begin{align*}
		\text{order}(\psi(R)) & = \text{smallest positive } n \text{ such that  } nR \in \langle P \rangle\\
		& = \text{smallest positive } n \text{ such that  } nR \in \langle P+Q \rangle =\text{order}(\phi(R)),
		\end{align*}
		where we have used the fact that $R$ is of $p$-power order. Thus, the result follows.
	\end{proof} 
	
	\begin{lemma}\label{lem-subsequent-rational-pts}
		Let $E / \Q$ be an elliptic curve and let $p$ be an odd prime. Suppose $E$ has a point $P_1$ of order $p$ defined over $\Q$, and suppose $C_{p}(E) = m+1 > 1$ such that $\langle P_{m} \rangle$ is a cyclic $\Q$-rational group of order $p^m$ containing $P_{1}$. For $0 \leq j \leq m$, let $P_{j} \in \langle P_{m} \rangle$ such that $P_0=\mathcal{O}$ and $[p]P_{j} = P_{j-1}$, and  let $\phi_{j} \colon E \to E_j$ be an isogeny with kernel $\langle P_{j} \rangle$. 
		\begin{enumerate}
			\item If $1\leq j \leq m-1$, then $E_j$ has a rational point of order $p$. Further, for $\phi_{{{m-1}}}\colon E \to E_{m-1}$, the curve $E_{m-1}$ has a rational point of order $p$ but no rational points of order $p^{2}$. 
			\item If $j=m$, the curve $E_m$ has no rational points of order $p$.
		\end{enumerate}
	\end{lemma}
	
	\begin{proof}
		Let $p$ be an odd prime and let $E / \Q$ be an elliptic curve. Suppose $C_{p}(E) = m+1 > 1$, with $C_p$ defined as in Theorem \ref{thm-kenku}. Suppose $\langle P_{m} \rangle$ is a cyclic $\Q$-rational subgroup of $E$ of order $p^m$ containing a point $P_1$ of order $p$ defined over $\Q$. Note that $\{\mathcal{O}\},\langle P_1\rangle,\ldots,\langle P_m\rangle$ are the $m+1$ subgroups of $p$-power order that are $\Q$-rational by Lemma \ref{lem-subgps-of-rat-groups}, and  their orders are $1,\ldots,p^m$, respectively.

		For $1 \leq j \leq m - 1$, denote $[p]P_{j+1} = P_{j}$. Then, $\langle P_{j} \rangle$ is the unique subgroup of $\langle P_{m} \rangle$ of order $p^{j}$.  Now let us fix $j$ with $1 \leq j \leq m-1$, and let $\phi_{j} \colon E \to E_j$ be an isogeny with kernel $\langle P_{j} \rangle$.
		
		Suppose $E[p^j] = \langle P_{j}, Q_{j} \rangle$ for some $Q_{j} \in E$. By Lemma \ref{orbit-of-supergroups-of-rational-groups}, $\sigma(P_{{j+1}}) - P_{{j+1}} \in \langle P_{j} \rangle$ for all $\sigma \in G_{\Q}$. By Lemma \ref{lem-necessity-for-point-rationality}, the point $\phi_j(P_{{j+1}})\in E_{j}(\Q)$ is rational. Moreover, it is clear that $\phi_{j}(P_{{j+1}})\in E_{j}$ is a point of order $p$.
		
		Now let us look at the case for $\phi_{{m-1}} \colon E \to E_{m-1}$, i.e., an isogeny with kernel $\langle P_{{m-1}} \rangle$. The point $\phi_{{m-1}}(P_{m})$ is a point of order $p$ defined over $\Q$ by Lemma \ref{orbit-of-supergroups-of-rational-groups} and Lemma \ref{lem-necessity-for-point-rationality}. Let $P_{{m+1}} \in E$ such that $[p]P_{{m+1}} = P_{m}$ and let $Q_{1} \in E$ be a non-rational point of order $p$ (i.e., $E[p] = \langle P_{1}, Q_{1} \rangle$; recall that $P_1\in E(\Q)$ and $p\geq 3$, so $Q_1$ is necessarily not rational). Then, the groups of order $p^2$ that contain $\phi_{{m-1}}(P_{m})$ are $$\langle \phi_{{{m-1}}}(P_{{m+1}}) \rangle, \langle \phi_{{{m-1}}}(P_{{m+1}} + Q_{1}) \rangle, \ldots, \langle \phi_{{{m-1}}}(P_{{m+1}} + [p-1]Q_{1}) \rangle.$$ 
		Suppose $\phi_{{{m-1}}}(P_{{m+1}} + [c]Q_{1})$ is defined over $\Q$ for some $0 \leq c \leq p - 1$. Note that $P_{{m-1}} \in \langle P_{{m+1}} + [c]Q_{1} \rangle$. Then, by Lemma \ref{lem-necessity-for-subgroup-rationality}, we have that $\langle P_{{m+1}} + [c]Q_{1} \rangle$ is a $\Q$-rational group of order $p^{m+1}$, however, we showed above none of the $\Q$-rational subgroups of $E$ have order $p^{m+1}$. Thus, no point on $E_{m-1}$ of order $p^2$ is defined over $\Q$, as desired.
		
		Finally, let us look at the case of  $\phi_{m} \colon E \to E_m$, i.e., an isogeny with kernel $\langle P_{m} \rangle$. By Lemma \ref{lem-number-of-groups}, the curve $E_m$ has $p+1$ subgroups of order $p$, namely, $$\langle \phi_{{m}}(Q_{1}) \rangle, \langle \phi_{{m}}(P_{{m+1}}) \rangle, 
		\langle \phi_{{m}}(P_{{m+1}} + Q_{1}) \rangle, \ldots, \langle \phi_{{m}}(P_{{m+1}} + [p-1]Q_{1}) \rangle.$$ 
		Suppose first that $\phi_{{m}}(Q_{1})$ is rational. Then, by Lemma \ref{lem-necessity-for-point-rationality}, $\sigma(Q_{1}) - Q_{1} \in \langle P_{m} \rangle$ for all $\sigma \in G_{\Q}$. As $\sigma(Q_{1}) - Q_{1} \in E[p] = \langle P_{1}, Q_{1} \rangle$, this means that $\sigma(Q_{1}) - Q_{1} \in \langle P_{m} \rangle \cap \langle P_{1}, Q_{1} \rangle = \langle P_{1} \rangle.$ Thus, for each $\sigma \in G_{\Q}$, we have $\sigma(Q_{1}) = [a]P_{1} + Q_{1}$ for some $a=a(\sigma) \in \Z / p \Z$. It follows from the rationality of $P_{1}$, that the image of $\rho_{E,p}$ is of the form $\left\{\left(\begin{array}{cc} 1 & * \\ 0 & 1 \end{array}\right)\right\}$. However, the determinant map has to surject onto $(\Z / p \Z)^\times$. Thus, we have a contradiction as the determinant is always 1 in our case. Finally, suppose $0 \leq c \leq p - 1$ and assume that $\phi_{{m}}(P_{{m+1}} + [c]Q_{1})$ is defined over $\Q$. Note that $\langle P_{m+1} + [c]Q_{1} \rangle$ contains $\langle P_{m} \rangle$. Then by Lemma  \ref{lem-necessity-for-subgroup-rationality}, the group $\langle P_{m+1} + [c]Q_{1} \rangle$ is $\Q$-rational of order $p^{m+1}$, a contradiction.
	\end{proof}
	
	\begin{lemma}\label{lem-evenness-of-C(E)-new}
		Let $E / \Q$ be an elliptic curve with an element $P\in E$ that generates a $\Q$-rational group of order $2^n$ for some $n \geq 1$. Then, the two cyclic groups of order $2^{n+1}$ that contain $P$ are either both $\Q$-rational or they are both not $\Q$-rational. Thus, for a general elliptic curve $E' / \Q$, $C_{2}(E')$ is either $1$ or even, and $C(E') \leq 8$ and $C(E') \neq 5$ and $C(E') \neq 7$.
	\end{lemma}
	
	\begin{proof}
		Let $P \in E$ be a point that generates a $\Q$-rational group of order $2^n$ with $n \geq 1.$ Let $P_{2} \in \langle P \rangle$ be the element of order $2$, then  $P_{2} = \left[2^{n-1}\right]P$ and moreover, $P_{2}$ is defined over $\Q$ by Lemmas \ref{lem-subgps-of-rat-groups} and \ref{lem-orbit-of-rational-points}. Let $P' \in E$ such that $[2]P' = P$. Let $Q_{2} \in E$ be an element of order $2$ not equal to $P_{2}$, then from Lemma \ref{lem-number-of-groups}, it follows that the subgroups of $E$ of order $2^{n+1}$ that contain $P$ are $\langle P' \rangle$ and $\langle P'+Q_2 \rangle$. If neither $\langle P' \rangle$ nor $\langle P' + Q_2 \rangle$ are $\Q$-rational, then we are done. Otherwise, suppose without loss of generality, that $\langle P' \rangle$ is $\Q$-rational. Then, for an arbitrary $\sigma \in G_{\Q}$, we have $\sigma(P') = [a]P'$ for some odd $a \in \Z$. By Lemma \ref{lem-orbit-of-rational-points}, $\sigma(Q_2) = Q_{2}$ or $P_{2} + Q_{2}$. If $\sigma$ fixes $Q_{2}$, then $\sigma(P' + Q_{2}) = \sigma(P') + \sigma(Q_{2}) = [a]P' + Q_{2} = [a]P' + [a]Q_{2} = [a](P' + Q_{2})$. If $\sigma(Q_{2}) = P_{2} + Q_{2} = [2^{n}]P' + Q_{2}$, then $\sigma(P'+ Q_{2}) = \sigma(P') + \sigma(Q_{2}) = [a]P' + [2^{n}]P' + Q_{2} = [a+2^{n}]P' + [a+2^{n}]Q_{2} = [a+2^{n}](P'+Q_{2})$ and it follows that $\langle P' + Q_2 \rangle$ is also $\Q$-rational.
		
		The last statement of the lemma is clear, because we have just shown that if an elliptic curve $E' / \Q$ has a $\Q$-rational $2^{n}$-isogeny for some $n \geq 1$, then the $\Q$-rational subgroups of $2$-power order come in pairs. We have shown that $C_{2}(E')$ is $1$ or even and by Theorem \ref{thm-kenku}, $C_{q}(E')$ is bounded by $4$ for all odd primes $q$, and hence, $C(E')$ can never equal $5$ or $7$.
		
	\end{proof}
	
	\begin{lemma}\label{lem-21 or 27}
		Let $E/\Q$ be an elliptic curve with a $\Q$-rational cyclic $n$-isogeny with $n=21$ or $27$. Then, $E(\Q)_{\text{tors}} \cong \{\mathcal{O}\}$ or $\Z/3\Z$, and both can occur.
	\end{lemma}
	\begin{proof} 
		Let $E/\Q$ be an elliptic curve with a $\Q$-rational cyclic $n$-isogeny of degree $n=21$ or $27$. Then, by Table 4 of \cite{lozano0}, we have 
		$$j(E) \in \{-3^2\cdot 5^6/2^3,\ 3^3\cdot 5^3/2,\ -3^2\cdot 5^3\cdot 101^3/2^{21},\ -3^3\cdot 5^3\cdot 383^3/27,\ -2^{15}\cdot 3\cdot 5^3 \},$$
		where the first four $j$-invariants have a $21$-isogeny and the last one has a $27$-isogeny.  Examples of elliptic curves with these $j$-invariants and smallest conductor and $E(\Q)_{\text{tors}}\cong \Z/3\Z$, according to LMFDB, include:
		$$\{\texttt{162.c3},\ \texttt{162.b4},\ \texttt{162.c2},\ \texttt{162.b1}, \ \texttt{27.a2}\}.$$
		Examples as above with trivial torsion over $\Q$ include:
		$$\{\texttt{162.b3},\ \texttt{162.c4},\ \texttt{162.b2},\ \texttt{162.c1}, \ \texttt{27.a1}\}.$$
		Further, the classification of degrees of $\Q$-rational isogenies (Theorem \ref{thm-ratnoncusps}) shows that any elliptic curve with an $n$-isogeny, for $n=21$ or $27$, has no other isogenies. If $p$ was a prime such that $E(\Q)$ contains a point of exact order $p^m$, then $E/\Q$ would have a $\Q$-rational $p^m$-isogeny, and therefore $p^m$ divides $21$ or $27$, respectively. Thus, in order to conclude that the only possible torsion subgroups are either $\{\mathcal{O}\}$ or $\Z/3\Z$, it suffices to show that there are no points of order $9$ in $E(\Q)$ when $n=27$, or of order $7$ when $n=21$. 
		
		A Magma calculation of division polynomials show that if $E$ is the curve \texttt{27.a2}, and $P\in E[9]$ is non-zero, then $K=\Q(x(P))$ has degree $3$, $6$, or $27$.  Now, since $j(E)\neq 0, 1728$, every elliptic curve $E'$ with $j=-2^{15}\cdot 3\cdot 5^3$ is a quadratic twist of \texttt{27.a2} and if $P'\in E'[9]$, then $\Q(x(P'))=K=\Q(x(P))$ by Lemma 9.6 of \cite{lozano0}. Thus, no elliptic curve with $j=-2^{15}\cdot 3\cdot 5^3$ has a $9$-torsion point defined over the rationals.
		
		Similarly, Magma shows that if $E$ is one of the curves \texttt{162.c3},\ \texttt{162.b4},\ \texttt{162.c2},\ \texttt{162.b1}, and $P\in E[7]$, then $K=\Q(x(P))$ has degree $3$, or $21$. Since none of the $j$-invariants are $0$ or $1728$, it follows that any elliptic curve with a $21$-isogeny has the $x$-coordinate of $7$-torsion points defined over an extension of either degree $3$ or $21$, and therefore none has a point of order $7$ defined over $\Q$, as we wanted to show. 
	\end{proof} 
	
	\begin{remark}
		Let $E/\Q$ be an elliptic curve with a cyclic $27$-isogeny defined over $\Q$. Then, $E/\Q$ has CM (see \cite{lozano0}, Table 4). By \cite{olson} (see also \cite{clarkcooketal}, Section 4), the torsion subgroup of an elliptic curve over $\Q$ with CM is isomorphic to one of the following groups: $0$, $\Z/2\Z$, $\Z/3\Z$, $\Z/4\Z$, $\Z/6\Z$, or $\Z/2\Z\oplus \Z/2\Z$. In particular, $\Z/9\Z$ does not occur. Note, however, that the $j$-invariants with a $21$-isogeny are not CM curves, so this argument cannot be applied to those curves.
	\end{remark}
	
	Last, we show a corollary of Theorems  \ref{thm-kenku} and \ref{thm-ratnoncusps}.
	
	\begin{cor}\label{cor-uniquecyclicQrationalgroup}
		Let $E/\Q$ be an elliptic curve, and suppose there exists a finite cyclic $\Q$-rational subgroup $H\subseteq E$ of order $n=n(H)$, and every finite cyclic $\Q$-rational subgroup of $E$ is contained in $H$. Then, $C_p=C_p(E)$ is bounded by
		\begin{center}
			\begin{tabular}{c|ccccccccccccc}
				$p$ & $2$ & $3$ & $5$ & $7$ & $11$ & $13$ & $17$ & $19$ & $37$ & $43$ & $67$ & $163$ & \text{else}\\
				\hline 
				$C_p\leq $ & $2$ & $4$ & $3$ & $2$ & $2$ & $2$ & $2$ & $2$ & $2$ & $2$ & $2$ & $2$ & $1$.
			\end{tabular}
		\end{center}
		Moreover, the order $n(H)$ takes one of the following values:
		$$n=n(H) = 2,3,5,6,7,9,10,11,13,14,15,17,18,19,21,25,27,37,43,67, \text{ or } 163.$$
	\end{cor}
	\begin{proof}
		Suppose $E$, $H$, and $n(H)$ are as in the statement. Then, $E$ has a rational cyclic isogeny of degree $n$. Theorem \ref{thm-ratnoncusps} now implies the bounds on $C_p(E)$, except for $C_2$. However, Lemma \ref{lem-4tors-implies-bicyclic-isog} says that if $E/\Q$ has a rational cyclic isogeny of degree $4$, then $E$ is $2$-isogenous to an elliptic curve $E'$ with full two-torsion defined over $\Q$. In particular, there are at least three $2$-isogenous curves to $E'$, say $E$, $E''$, and $E'''$, such that $E\to E''$ and $E \to E'''$ are distinct rational cyclic isogenies of degree $4$. Thus, $H$ would have two distinct subgroups of order $4$, but this is impossible because it is cyclic. Hence, $E$ cannot have isogenies of degree $4$, and $C_2(E)\leq 2$ as claimed.
		
		Finally, the number $n(H)$ is the order of a cyclic $\Q$-rational isogeny, and these are classified by Theorem \ref{thm-ratnoncusps}.
	\end{proof}

	\section{The Possible Isogeny Graphs}\label{sec-whatgraphs}
	
	In this section we show Theorem \ref{thm-mainisogenygraphs}. The isogeny graphs seem to have first appeared in the so-called Antwerp tables \cite{antwerp} where the authors remark that ``the tables in fact illustrate almost all the known ways in which isogenies can occur''. In this section, we shall use Theorem \ref{thm-kenku} to determine the possible isogeny graphs that occur for elliptic curves over $\Q$ and then, in the rest of the paper, we will determine what torsion subgroups over $\Q$ can occur at each vertex of an isogeny-torsion graph.

	Let $E/\Q$ be an elliptic curve. We first distinguish two cases, according to whether the following condition is met: 
	\begin{enumerate}
		\item[$\Diamond$] There exists a finite cyclic $\Q$-rational subgroup $H\subseteq E$ such that every finite cyclic $\Q$-rational subgroup of $E$ is contained in $H$.
	\end{enumerate}
	Suppose first that $E/\Q$ satisfies $\Diamond$, with $H=\langle P\rangle$, and let $n=n(P)=n(H)$ be the order of $P$:
	\begin{itemize}
		\item If $n=1$, then $E/\Q$ does not have any isogenous curves other than itself, and the graph is a single vertex. Notice that in this case $E(\Q)_{\text{tors}}$ is necessarily trivial. The trivial graph is denoted by $L_1$.
		\item If $n=p^k$ is a prime power, then Cor. \ref{cor-uniquecyclicQrationalgroup} says that $p\in \{2,3,5,7,11,13,17,19,37,43,67,163\}$ and $k\leq 3$, where $k=3$ is only possible for $p=3$. In this case, the isogenies correspond to the subgroups $\langle [p^j] P \rangle$, for $0\leq j \leq k$, and therefore the graphs are linear with $\leq 4$ vertices. We shall refer to these graphs as linear,  and denote them by $L_{k+1}(p^k)$. Note that in this case, the graph has $k+1$ vertices and two curves have a cyclic $\Q$-rational isogeny of degree $p^{k}$, the maximal degree among the finite cyclic $\Q$-rational isogenies. The case of $k=1$ (two vertices), i.e., in the case when $C(E) = C_{p}(E) = 2$, occurs for $p\in \{2, 3, 5, 7, 11, 13, 17, 19, 37, 43, 67, 163\}$:
		\begin{center}\begin{tikzcd}
				E_{1} \arrow[r,"p", no head] & E_{2} = E_1/\langle P\rangle
		\end{tikzcd}\end{center}
		The case of an isogeny graph $L_3(p^2)$ with three vertices, i.e., when $C(E) = C_{p}(E) = 3$ and $n(P)=p^2$, can only occur for $p = 3$ or $5$, by Cor. \ref{cor-uniquecyclicQrationalgroup}.
		\begin{center} 
			\begin{tikzcd}
				E_{1} \arrow["p", r, no head] & E_{2} = E_1/\langle [p]P\rangle \arrow["p",r, no head] & E_{3}= E_1/\langle P\rangle
			\end{tikzcd}
		\end{center} 
		Finally, the case of $L_4(p^3)$, a linear isogeny graph with $4$ vertices, where $C(E)=C_p(E)=4$, can only occur for $p=3$, so in this case, we will denote $L_4(27)$ simply by $L_4$.
		\begin{center} 
			\begin{tikzcd}
				E_{1} \arrow["3", r, no head] & E_{2}= E_1/\langle [p^2]P\rangle \arrow["3", r, no head] & E_{3}= E_1/\langle [p]P\rangle \arrow["3", r, no head] & E_{4}= E_1/\langle P\rangle
			\end{tikzcd}
		\end{center} 
	In particular, $E_1$ and $E_{4}$ have rational isogenies of degree $27$, which by the tables in \cite{Alvaro}, implies that $\textit{j}(E_1) = \textit{j}(E_4)=-2^{15}\cdot 3\cdot 5^3$, and so $E_1$ and $E_4$ are CM. This shows the last claim of Theorem \ref{thm-mainisogenygraphs}.
		\item If $n=n(P)$ is composite but not a prime power, then Corollary  \ref{cor-uniquecyclicQrationalgroup} says that $n=6, 10,14,15,18$, or $21$. Thus, in all these cases we have $C_p(E)=2$ and $C_q(E)=2$ or $3$, for some distinct primes $p$ and $q$. We shall call these graphs rectangular, and denote them by $R_v(n)$ where $v=C(E)$ is the number of vertices, and $n=p\cdot q$ or $p\cdot q^2$. When $n=pq=6,10,14,15$ or $21$, then the graph is $R_4(n)$, a square of the form: 
		\begin{center}
			\begin{tikzcd}
				E_{1} \arrow["p",d, no head] \arrow["q",r, no head] & E_{2} = E_1/\langle [p]P \rangle \arrow["p",d, no head] \\
				E_{3} = E_1/\langle [q]P \rangle \arrow["q",r, no head]           & E_{4}  = E_1/\langle P \rangle
			\end{tikzcd}
		\end{center}
		Finally, if $n=pq^2$, then $n=18$ with $p=2$ and $q=3$. In this case the graph is $R_6=R_6(18)$, a rectangle of the form:
		\begin{center}
			\begin{tikzcd}
				E_{1} \arrow["2",d, no head] \arrow["3",r, no head] & E_{3} = E_1/\langle [6]P \rangle \arrow["2",d, no head] \arrow["3",r, no head] & E_{5} = E_1/\langle [2]P \rangle \arrow["2",d, no head] \\
				E_{2} = E_1/\langle [9]P \rangle \arrow["3",r, no head]           & E_{4}  = E_1/\langle [3]P \rangle \arrow["3",r, no head] & E_{6} = E_1/\langle P \rangle .
			\end{tikzcd}
		\end{center}
	\end{itemize}
	
	We remark here that the result below about $3$-isogenies is stated in \cite{kenku} without proof, so we include one here for completeness.

	\begin{lemma}\label{lem-27}
		Let $E/\Q$ be an elliptic curve such that $C_3(E)=C(E)=4$. Then, there is an elliptic curve in the $\Q$-isogeny class of $E$ with a $27$-isogeny.
	\end{lemma}
	\begin{proof} 
	    If $E / \Q$ satisfies the $\Diamond$ property, then we are done. Suppose $E/\Q$ does not satisfy the $\diamond$ property, then there is an elliptic curve $E'$ isogenous to $E$ and such that $E'[3]=\langle P,Q\rangle$, and both $\langle P\rangle$ and $\langle Q \rangle$ are $\Q$-rational subgroups of $E'$. If $E'$ has no other rational $3$-isogenies, then $E'$ must have a rational cyclic $9$-isogeny containing one of the $3$-isogenies, and so the isogeny graph of $E$ must be of type $L_4$, and the elliptic curves in the corner vertices have rational $27$-isogenies. 
		
		Otherwise, suppose for a contradiction that  $E'$ has three independent $3$-isogenies (and the graph would be of type $T_4$ as described below for the case of $C_2(E)=C(E)=4$). Then, there are distinct points $P$, $Q$, and $R$ of order $3$ such that the subgroups generated by them are $\Q$-rational. Since $ P,Q$  generate $E[3]$, they form a basis and $R = [a]P+[b]Q$ for some integers $a$ and $b$ (neither of which is $0 \bmod 3$, because $R$ generates a third independent isogeny by assumption). Suppose $\sigma\in G_\Q$ sends $\sigma(P) = [\alpha] P$, $\sigma(Q)=[\beta] Q$ and $\sigma(R) = [\gamma] R$. Then,
		$$[\gamma a] P + [\gamma b] Q = [\gamma] R = \sigma(R) = \sigma([a] P+[b] Q) = [a\alpha] P + [b\beta] Q,$$
		thus, $a\gamma \equiv  a\alpha \bmod 3$ and $b\gamma \equiv b\beta \bmod 3$. Since $a$ and $b$ are units mod $3$, we conclude $\gamma\equiv\alpha\equiv \beta \bmod 3$. Hence, for every $\sigma\in G_\Q$ there is $r=r(\sigma)$ such that $\sigma(T) = [r(\sigma)] T$, for all $T$ in $E[3]$. In other words, $\rho_{E,3}(\sigma)=r(\sigma)\cdot \operatorname{Id}$ has a diagonal form with identical diagonal entries. But then, the subgroup $\det(\rho_{E, 3}(G_\Q))$ is formed by the squares in $\FF_3^\times$, which does not include $2 \bmod 3$. This means the determinant is not surjective and that is a contradiction.
	\end{proof} 
	
	Now suppose that $E/\Q$ does not satisfy the $\Diamond$ property. Then, it follows that there is an elliptic curve $E'$ isogenous to $E$ and a  prime $p$ such that $E'[p]=\langle P,Q\rangle$, and both $\langle P\rangle$ and $\langle Q \rangle$ are $\Q$-rational subgroups of $E'$. This means that $E'$ has two independent rational $p$-isogenies and $C_p(E')\geq 3$, but this can only occur for $p=2$, $3$, or $5$, by Theorem 6.2 of \cite{lozano0}. If $C_5(E')\geq 3$, then $C_5(E')=C(E)=3$ by Theorem \ref{thm-kenku}, and $E'$ is the curve in the middle of an $L_3(25)$ graph, which we have already considered. If $C_3(E')\geq 3$, then $C_3(E')=3$ or $4$. If $C_3(E')=4$, then $C(E)=4$ by Theorem \ref{thm-kenku}, and Lemma \ref{lem-27} shows that $E'$ is one of the two middle curves in an $L_4$ graph, which has already been considered. 
	
	If $C_3(E')=3$, then Theorem \ref{thm-kenku} says that $C(E')=3$ or $6$, and either $C_3(E')=C(E')=3$, in which case this is an $L_3(9)$ graph, or $C_3(E')=3$ and $C_2(E')=2$, in which case we are dealing with an $R_6$ graph, already dealt with. The final case is when $E'$ has two independent rational $2$-isogenies, in which case, $E'[2]=E'(\Q)[2]$ is defined over $\Q$. Thus, using $E'$ as our new $E$ if necessary, we shall assume our curve $E/\Q$ has all of its $2$-torsion points defined over $\Q$, in which case $C_2(E)\geq 4$. Then, either $C_3(E)=2$ and $C_2(E)=4$, or Lemma \ref{lem-evenness-of-C(E)-new} shows that $C_2(E)$ is even, and therefore either $C(E)=C_2(E)=4$, $6$, or $8$.

	Let us first assume that $E/\Q$ has its full $2$-torsion defined over $\Q$ and $C_3(E)=1$. Then, $C_2(E) = C(E) =k$ with $k=4$, $6$, or $8$, and we shall call these graphs two-isogeny graphs, and  denote them by $T_k$. Since all isogenies in such a graph between consecutive vertices are $2$-isogenies, the number of vertices of the isogeny graph is enough to distinguish a $T_{k}$ graph from every other isogeny graph. Therefore it is not necessary to add a parenthesis denoting the maximal cyclic $\Q$-isogeny degree to determine the graph type. We let $E[2]=\langle P_2,Q_2\rangle$. 
	
	\begin{itemize} 
		\item If $C_{2}(E) = C(E) = 4$, then the only isogenies are the $2$-isogenies that come from $2$-torsion points defined over $\Q$. Then, the graph is of type $T_4$ as follows: 
		\begin{center} 
			\begin{tikzcd}
				& E_{2} = E_1/\langle P_2\rangle                             &       \\
				& E_1 \arrow["2",u, no head] \arrow["2",ld, no head] \arrow["2",rd, no head] &       \\
				E_{3}= E_1/\langle P_2+Q_2\rangle &                                   & E_{4} = E_1/\langle Q_2\rangle
			\end{tikzcd}
		\end{center} 
		
		\item If $C_{2}(E) = C(E) = 6$, then $E$ must have an additional rational cyclic $4$-isogeny. Let $Q_{4} \in E$ such that $[2]Q_4 = Q_2$, then after a change of basis of $E[4]$ or $E[2]$ if necessary, we may assume $Q_4$ generates a $\Q$-rational $4$-isogeny. Then, $\langle Q_4+P_2 \rangle$ is also $\Q$-rational by Lemma \ref{lem-evenness-of-C(E)-new}. Hence, we have accounted for the $6$ curves that are $2^k$-isogenous to $E$, and the graph is of type $T_6$, as follows:  
		\begin{center} 
			\begin{tikzcd}
				E_{2}=E_1/\langle P_2\rangle &                                       &                             & E_{5}=E_1/\langle Q_4\rangle \\
				& E_{1} \arrow["2",lu, no head] \arrow["2",ld, no head] \arrow["2",r, no head] & E_{4}=E_1/\langle Q_2\rangle \arrow["2",ru, no head] \arrow["2",rd] &       \\
				E_{3}=E_1/\langle P_2+Q_2\rangle &                                       &                             & E_{6}=E_1/\langle Q_4+P_2\rangle
			\end{tikzcd}
		\end{center}

		\item Suppose $C_{2}(E) = C(E) = 8$. Let $P_{4} \in E$ such that $[2]P_{4} = P_{2}$ and $Q_{8} \in E$ such that $[2]Q_{8} = Q_{4}$. Then, either (i) $E$ has two independent $4$-isogenies, i.e., after a change of basis of $E[4]$ if necessary, $E[4]=\langle P_4,Q_4\rangle$ and $\langle P_4 \rangle$ and $\langle Q_4 \rangle$ are $\Q$-rational, or (ii) $E$ has an $8$-isogeny corresponding to $\langle Q_8 \rangle$ (again, after a change of basis of $E[8]$ is necessary). We shall assume the latter (the former leads to the same graph, where $E$ would be $E/\langle Q_2 \rangle$ below), and call $E=E_1$. By Lemma \ref{lem-evenness-of-C(E)-new}, $\langle P_2+Q_8\rangle$ is also $\Q$-rational. Moreover, $E_4=E/\langle Q_2\rangle$ also has full two-torsion over $\Q$, by Lemma \ref{lem-4tors-implies-bicyclic-isog}. The curves described so far already account for $8$ curves that are $2^k$-isogenous to $E_1$, and therefore we have a graph of type $T_8$, as follows, where $Q_4=[2]Q_8$, and $Q_2=[4]Q_8=[2]Q_4$.
		\begin{center}
			\begin{tikzcd}
				& \substack{E_2=\\E_1/\langle P_2\rangle}                                &                            & \substack{E_7=\\E_1/\langle Q_8 \rangle}                     &       \\
				& E_{1} \arrow["2",u, no head] \arrow["2",ld, no head] \arrow["2",rd, no head] &                            &  \substack{E_6=\\E_1/\langle Q_4 \rangle} \arrow["2",rd, no head] \arrow["2",u, no head] &       \\
				\substack{E_3=\\E_1/\langle P_2 + Q_2 \rangle}&                                       & \substack{E_4=\\E_1/\langle Q_2 \rangle} \arrow["2",d, no head] \arrow["2",ru, no head] &                            & \substack{E_8=\\E_1/\langle P_2+Q_8 \rangle}\\
				&                                       &  \substack{E_5=\\E_1/\langle P_2+Q_4 \rangle}                     &                            &      
			\end{tikzcd}
		\end{center} 
	\end{itemize}

	Finally, we consider the case when $C_2(E)= 4$ and $C_3(E)= 2$, so that $C(E)=8$. In this case we have $\Q$-rational groups $\langle P \rangle$ and $\langle Q\rangle$ of order $2$ as before, and an additional $\Q$-rational group $\langle A\rangle$ of order $3$. This corresponds to the following type of graph, that we will denote by $S$-type:
	\begin{center} 
		\begin{tikzcd}
			& E_{3}=E_1/\langle P \rangle  \arrow["3",r, no head]                                 & E_{4}=E_1/\langle P+A \rangle                                 &       \\
			& E_{1} \arrow["2",u, no head] \arrow["3",r, no head] \arrow["2",rd, no head] \arrow["2",ld, no head] & E_{2}=E_1/\langle A \rangle \arrow["2",u, no head] \arrow["2",rd, no head] \arrow["2",ld, no head] &       \\
			E_{5}=E_1/\langle P+Q \rangle \arrow["3",r, no head] & E_{6}=E_1/\langle P+Q+A \rangle                                & E_{7}=E_1/\langle Q \rangle \arrow["3",r, no head]                       & E_{8}=E_1/\langle Q+A \rangle
		\end{tikzcd}
	\end{center} 
This completes the proof of Theorem \ref{thm-mainisogenygraphs}. 
	
	In the rest of the paper, we will determine the torsion subgroups of each of the curves in the isogeny graphs of $L_k$, $R_k$, $T_k$, or $S$ type. Since we have already treated elliptic curves with CM in Section \ref{sec-CM}, we will continue our investigation under the assumption that $E/\Q$ does not have CM.

	\section{Linear graphs $L_k$}\label{sec-linear}
	
	In this section, we analyze the possible torsion subgroups arising in linear graphs $L_k$. We will often use the fact that if a prime $p$ divides the order of $E(\Q)_{\text{tors}}$, then $C_p(E)>1$ (indeed, a point $P$ of order $p$ defined over $\Q$ induces a $\Q$-rational $p$-isogeny $E\to E/\langle P\rangle$). Below, $E_1=E$.
	\begin{enumerate}
		\item[$L_1$.] The graph consists of one vertex, the curve $E/\Q$ does not have any isogenous curves (other than itself), and $E(\Q)_{\text{tors}}$ is necessarily trivial as any torsion point $P\in E(\Q)_{\text{tors}}$ would induce a rational isogeny.
		
		\vskip 0.2in
		\item[$L_2$.] The case of two vertices, i.e., the case when $C(E) = C_{p}(E) = 2$ for some prime $p$ and $C_q(E)=1$ for all primes $q$ not equal to $p$, occurs for $p\in \{2, 3, 5, 7, 11, 13, 17, 19, 37, 43, 67, 163\}$:
		\begin{center}\begin{tikzcd}
				E_{1} \arrow["p",r, no head] & E_{2} = E_1/\langle P\rangle
		\end{tikzcd}\end{center}
		Suppose first that $p>7$ and $C_p(E)=2$. If $q$ divides the order of the torsion subgroup of $E_1$ or $E_2$, then $q=p$. Now, Mazur's theorem \ref{thm-mazur} shows that $E_1(\Q)_{\text{tors}}=E_2(\Q)_{\text{tors}}$ must be trivial. Hence, the isogeny-torsion graph $L_2(p)$ is of the form $([1],[1])$.
		
		Now let $p$ be $3, 5,$ or $7$, suppose $C_{p}(E) = C(E) = 2$ and such that $P \in E$ generates a $\Q$-rational subgroup of order $p$. Both $E_1$ and $E_2$ can have trivial torsion over $\Q$, simultaneously, giving $([1],[1])$. If $P\in E_1(\Q)_{\text{tors}}$, however, by Lemma \ref{lem-subsequent-rational-pts}, $E_2 = E / \langle P\rangle$ has no rational points of order $p$ and thus, $E_2(\Q)_{\text{tors}}$ is trivial. Thus, the isogeny-torsion graph is $([p],[1])$.
		
		Finally, if $C_{2}(E) = C(E) = 2$, then both isogenous curves have torsion subgroups isomorphic to $\Z / 2 \Z$, by Lemma \ref{lem-orbit-of-rational-points} and Lemma \ref{lem-2torspt-all-have-2torspt}, so the isogeny-torsion graph is $([2],[2])$.
		
		\vskip 0.2in
		\item[$L_3$.] By our work in Section \ref{sec-whatgraphs} (see also Corollary  \ref{cor-uniquecyclicQrationalgroup}), the case of $L_3(p^2)$, that is, $C(E)=C_p(E)=3$ and $n(P)=p^2$, can only occur for $p=3$ or $5$. In such case, we have
		\begin{center} 
			\begin{tikzcd}
				E_{1} \arrow["p",r, no head] & E_{2} = E_1/\langle [p]P\rangle \arrow["p",r, no head] & E_{3}= E_1/\langle P\rangle
			\end{tikzcd}
		\end{center} 
		Since $E_2$ has two independent $p$-isogenies, for $p=3$ or $5$, it follows that the image of $\rho_{E_2,p}$ is contained in a split Cartan subgroup of $\GL(2,\FF_p)$. If $E_2$ does not have rational $p$-torsion points, then neither $E_1$ or $E_3$ would have rational torsion points, by Lemma \ref{lem-subsequent-rational-pts} (which would imply a rational $p$-torsion point on $E_2$), and the isogeny-torsion graph is $([1],[1],[1])$. Otherwise, if $E_2$ has a rational $p$-torsion point, then Lemma \ref{p^2 isogeny} implies that either $E_1$ or $E_3$ have a rational point of $p$-torsion (and, clearly, a cyclic rational $p^2$-isogeny). Let us assume that $E_1$ has a rational $p$-torsion point. By Lemma \ref{lem-subsequent-rational-pts}, we must have $E_2(\Q)_{\text{tors}}\cong \Z / p \Z$ and $E_3(\Q)_{\text{tors}}$ is trivial. Moreover, $E_1(\Q)_{\text{tors}}$ can be $\Z/3\Z$, $\Z/9\Z$ when $p = 3$ and $E_{1}(\Q)_{\text{tors}}$ is $\Z/5\Z$ when $p = 5$, by Mazur's theorem \ref{thm-mazur}. Thus, the possible isogeny-torsion configurations in this case are $([1],[1],[1])$ for any $L_3(p^2)$, and in addition $([9],[3],[1])$, $([3],[3],[1])$ for $L_3(9)$, and $([5],[5],[1])$ for the $L_3(25)$-type.
		
		\vskip 0.2in
		\item[$L_4$.] The case of $L_4$ can only occur for $C(E)=C_3(E)=4$, i.e., for an elliptic curve $E_1/\Q$ with a $27$-isogeny, induced by a $\Q$-rational subgroup of $E$, $\langle P\rangle$ of order $27$:
		\begin{center} 
			\begin{tikzcd}
				E_{1} \arrow["3",r, no head] & E_{2}= E_1/\langle [9]P\rangle \arrow["3",r, no head] & E_{3}= E_1/\langle [3]P\rangle \arrow["3",r, no head] & E_{4}= E_1/\langle P\rangle
			\end{tikzcd}
		\end{center} 
		By Table 4 of \cite{lozano0}, we must have $\textit{j}(E)=-2^{15}\cdot 3\cdot 5^3$, and by Lemma \ref{lem-21 or 27}, we have $E_1(\Q)_{\text{tors}}\cong \Z/3\Z$ or trivial. Suppose that $E_1$ has non-trivial torsion over $\Q$. Then, by Lemma \ref{lem-subsequent-rational-pts}, we have $E_2(\Q)_{\text{tors}}$ has a point of order $3$, $E_3(\Q)_{\text{tors}}\cong \Z/3\Z$, and $E_4(\Q)_{\text{tors}}$ is trivial. It remains to show that $E_2(\Q)_{\text{tors}} \cong \Z/3\Z$. Indeed, if it was larger, then it would be $\Z/9\Z$, however \cite{olson} shows that there is no elliptic curve with CM and torsion subgroup over $\Q$ containing $\Z/9\Z$, and $E_2$ is a twist of an elliptic curve with $j=0$, therefore CM.
		
		Hence, the only possibilities for the $L_4$ isogeny-torsion graphs that can occur over $\Q$ are $([3],[3],[3],[1])$ and $([1],[1],[1],[1])$.
	\end{enumerate}
	
	\section{Rectangular graphs $R_k$}\label{sec-rectangular}
	
	In this section we determine the possible torsion configurations in graphs of $R_k(n)$-type, for $k=4$ and $6$.
	\begin{enumerate}
		\item[$R_4$.] In this case $C(E)=4$, with $C_p(E)=2$ and $C_q(E)=2$, for some distinct primes $p$ and $q$. As before, let $P \in E$ be a point that generates a $\Q$-rational group of order $pq$. The possibilities for $n=n(P)=pq$ are $6,10,14,15$ or $21$, and we obtain a square graph of the form
		\begin{center}
			\begin{tikzcd}
				E_{1} \arrow["p",d, no head] \arrow["q",r, no head] & E_{2} = E_1/\langle [p]P \rangle \arrow["p", d, no head] \\
				E_{3} = E_1/\langle [q]P \rangle \arrow["q",r, no head]           & E_{4}  = E_1/\langle P \rangle
			\end{tikzcd}
		\end{center}
		If $n$ is odd (i.e., $n=15$ or $21$), then by Mazur's theorem $E_1$ may have a point of order $3$, $5$, or $7$ defined over $\Q$, but none of order $15$ or $21$.  Note, however, that Lemma \ref{lem-21 or 27} shows that if $n=21$, then $E_1(\Q)_{\text{tors}}\cong \Z/3\Z$ or trivial. Thus, $E_1(\Q)_{\text{tors}}\cong \Z/p\Z$ with $p=3$ or $5$ or trivial.
		
		Suppose $E_1(\Q)_{\text{tors}}\cong \Z/p\Z$ with $p=3$ or $5$, and $P$ is a point that generates a $\Q$-rational subgroup of order $n=pq$. Thus, $[p]P$ is of order $q$, and $[q]P$ is of order $p$. Lemma \ref{lem-subsequent-rational-pts} shows that $E_3$ has no points of order $p$ over $\Q$.  Since $E_1\to E_3$ is of degree $p$, $E_3$ has rational $q$-torsion if and only if $E_1$ does, therefore neither one does, and $E_3(\Q)_{\text{tors}}$ is trivial. Similarly, $E_1\to E_2$ is of degree $q$, and so $E_2(\Q)_{\text{tors}}\cong \Z/p\Z$. Finally, $E_2\to E_4$ is also of degree $p$, and therefore $E_4(\Q)_{\text{tors}}$ must be trivial. Thus, it follows that if $n$ is odd, then the possible isogeny-torsion graphs of type $R_4(n)$ are either of the form $([1],[1],[1],[1])$, $([3],[3],[1],[1])$ or  $([5],[5],[1],[1])$ for $R_4(15)$, and of the form $([1],[1],[1],[1])$, or $([3],[3],[1],[1])$ for $R_4(21)$.
		
		Now suppose that $n$ is even, so that $n=6$, $10$, or $14$. That is, $C_2(E)=2=C_p(E)$ for $p=3,5,$ or $7$. Since $C_2(E)=2$, it follows that every elliptic curve in the graph has a rational $2$-torsion point, by Lemma \ref{lem-2torspt-all-have-2torspt}. Note that none of the curves may have a $7$-torsion point, by Mazur's theorem, as the curve would have a rational point of order $14$. Moreover, if $E_1$ has a rational point of order $p$, with $p=3$ or $5$, then our arguments above with $q=2$ show that $E_1(\Q)_{\text{tors}} = E_2(\Q)_{\text{tors}}\cong \Z/2p\Z$, while $E_3(\Q)_{\text{tors}} = E_4(\Q)_{\text{tors}}\cong \Z/2\Z$. Thus, the possible isogeny-torsion configurations of $R_4(n)$ with $n=6,10,14$ are $([2],[2],[2],[2])$, and  $([6],[6],[2],[2])$ if $n=6$, and $([10],[10],[2],[2])$ if $n=10$.\vskip 0.2in
		
		\item[$R_6$.] By Section \ref{sec-whatgraphs}, we have $C_2(E)=2$ and $C_3(E)=3$, with $n=18$. Let $P \in E$ be a point that generates a $\Q$-rational group of order $18$, so that we have the following isogeny graph:
		\begin{center}
			\begin{tikzcd}
				E_{1} \arrow["2",d, no head] \arrow["3",r, no head] & E_{3} = E_1/\langle [6]P \rangle \arrow["2",d, no head] \arrow["3",r, no head] & E_{5} = E_1/\langle [2]P \rangle \arrow["2",d, no head] \\
				E_{2} = E_1/\langle [9]P \rangle \arrow["3",r, no head]           & E_{4}  = E_1/\langle [3]P \rangle \arrow["3",r, no head] & E_{6} = E_1/\langle P \rangle .
			\end{tikzcd}
		\end{center}
		By Lemma \ref{lem-2torspt-all-have-2torspt}, all curves have a $2$-torsion point defined over $\Q$. Note that there cannot be a rational $9$-torsion point by Mazur's theorem. Also note that by Lemma \ref{lem-necessity-for-point-rationality} and the fact that $E_{i} \to E_{i+1}$ is a $2$-isogeny for each odd $i$, $E_{i}$ has a point of order $3$ defined over $\Q$ if and only if $E_{i+1}$ has a point of order $3$ defined over $\Q$. Moreover, if $E_{3}$ and $E_{4}$ have points of order $3$ defined over $\Q$, then by Lemma \ref{lem-subsequent-rational-pts}, one and only one pair $(E_{1}, E_{2})$ or $(E_{5},E_{6})$ have points of order $3$ defined over $\Q$. After a renumbering if necessary, we may assume that $E_{5}(\Q)_{\text{tors}} \cong E_{6}(\Q)_{\text{tors}} \cong \Z / 2 \Z$. Thus, there are two possible configurations of $R_6=R_6(18)$-type, namely $([2],[2],[2],[2],[2],[2])$ or $([6],[6],[6],[6],[2],[2])$.
	\end{enumerate}
	
	\section{Elliptic curves with full rational $2$-torsion, and graph of type $T_4$}\label{sec-T4}
	
	In this section we treat the case where $E=E_1/\Q$ has full $2$-torsion defined over $\Q$ and $C(E)=C_2(E)=4$. Then, the graph is of type $T_4$ as follows: 
	\begin{center} 
		\begin{tikzcd}
			& E_{2} = E_1/\langle P_2\rangle                             &       \\
			& E_1 \arrow["2",u, no head] \arrow["2",ld, no head] \arrow["2",rd, no head] &       \\
			E_{3}= E_1/\langle P_2+Q_2\rangle &                                   & E_{4} = E_1/\langle Q_2\rangle
		\end{tikzcd}
	\end{center} 
	
	Let $E_1(\Q)_{\text{tors}} = \langle P_{2}, Q_{2} \rangle \cong \Z / 2 \Z \times \Z / 2 \Z$, i.e., the only $\Q$-rational groups are the point-wise rational groups, $\langle \mathcal{O} \rangle, \langle P_{2} \rangle, \langle Q_{2} \rangle,$ and $\langle P_{2} + Q_{2} \rangle$, and let $E_2=E_1/\langle P_2\rangle$, and $E_3=E_1/\langle P_2+Q_2\rangle$, and $E_{4} = E_1/\langle Q_2\rangle$. By Lemma \ref{lem-2torspt-all-have-2torspt}, $E_2$, $E_3$, and $E_4$ have a rational $2$-torsion point, namely $Q_2+\langle P_2\rangle$, $P_2 + \langle P_2+Q_2\rangle$, and $P_2+\langle Q_2\rangle$, respectively. Moreover, Lemma \ref{Maximality-Of-Rational-2-Power-Groups} says that the torsion subgroup over $\Q$ of $E_k$ is cyclic, for $k=2,3,4$. Since $C_2(E_1)=4$, by \ref{lem-Q-Rational-Grps-Order-8} we conclude that $E_k(\Q)_{\text{tors}}$ is isomorphic to $\Z/2\Z$ or $\Z/4\Z$, for $k=2,3,4$. We claim that not all three can be isomorphic to $\Z/4\Z$. 
	
	Suppose for a contradiction that $E_k(\Q)_{\text{tors}}\cong \Z/4\Z$ for all $k=2,3,4$. Let us write $E_1[4]=\langle P_4,Q_4\rangle$ such that $[2]P_4=P_2$ and $[2]Q_4=Q_2$. Then, the points of $E_4$ that are of order $4$ and double to be rational are multiples of $P_4+\langle Q_2\rangle$ and $P_4+Q_4+\langle Q_2\rangle$. Let us assume that $T=P_4+\langle Q_2\rangle$ is defined over $\Q$ (otherwise, if $P_4+Q_4+\langle Q_2\rangle$ was rational, we can change basis elements to $P_2'=P_2+Q_2$ and arrive to this case). By Lemma \ref{lem-necessity-for-point-rationality}, $\sigma(P_4)-P_4\in \langle Q_2\rangle$, for all $\sigma \in G_\Q$. 
	
	Now, the points on $E_2$ that double to be rational are multiples of $R_1=Q_4+\langle P_2 \rangle$ and $R_2=P_4+Q_4+\langle P_2\rangle$. Similarly, the points on $E_3$ that double to be rational, are multiples of $S_1=P_4+\langle P_2+Q_2 \rangle$ and $S_2=Q_4+\langle P_2+Q_2\rangle$.
	
	\begin{itemize} 
		\item If $T$ and $R_1$ are both rational, then $\sigma(P_4)=P_4+[a]Q_2=P_4+[2a]Q_4$ and $\sigma(Q_4)=Q_4+[2b]P_4$, for some integers $a=a(\sigma),b=b(\sigma)$, for all $\sigma\in G_\Q$. Then, the image of $\rho_{E,4}$ is of the form 
		$$\left\{\left(\begin{array}{cc} 1 & 2b \\ 2a & 1 \end{array} \right) \right\},$$
		where all the matrices have determinant $\equiv 1 \bmod 4$. Since the determinant of $\rho_{E,4}$ needs to be surjective onto $(\Z/4\Z)^\times$, this case is impossible. 
		\item Let $\sigma\in G_\Q$. If $T$ and $S_1$ are rational, then $\sigma(P_4)=P_4+[a]Q_2$ on one hand, and on the other hand $\sigma(P_4)=P_4+[b](P_2+Q_2)=[2b+1]P_4+[b]Q_2$, then we must have $a,b$ even, and it follows that $\sigma(P_4)=P_4$, and therefore $P_4$ is rational on $E_1$, a contradiction because $E_1(\Q)_{\text{tors}} \cong \Z/2\Z \times \Z/2\Z$. Similarly, $R_1$ and $S_2$ cannot be both rational.
		\item Hence, the only remaining possibility is that $T$, $R_2$, and $S_2$ are rational. Let $\sigma\in G_\Q$. In this case $\sigma(P_4)=P_4+[2a]Q_4$ and $\sigma(Q_4)=Q_4+[b](P_2+Q_2)=[2b+1]Q_4+[2b]P_4$. It follows that the image of $\rho_{E,4}$ is the subgroup 
		$$H=\left\{\left(\begin{array}{cc} 1 & 0 \\ 0 & 1 \end{array} \right), \left(\begin{array}{cc} 1 & 2 \\ 2 & 3 \end{array} \right) \right\},$$
		of order $2$ of $\GL(2,\Z/4\Z)$. 
		
		If $E_1/\Q$ has CM, then we saw in Section \ref{sec-CM} that the torsion configuration $([2,2],[4],[4],[4])$ cannot occur (see also \cite{lozano1} for a proof that $H$ cannot be a mod-$4$ image of an elliptic curve defined over $\Q$ with CM). If $E_1/\Q$ does not have CM, then its $2$-adic image $\rho_{E,2^\infty}(G_\Q)$ must reduce modulo $4$ to $H$. We have used Magma (code available at \cite{lozanoweb}) to search through the Rouse--Zureick-Brown database \cite{rouse} of $2$-adic images for non-CM curves (Theorem \ref{thm-rzb}), and no $2$-adic image reduces to a subgroup that is conjugate to $H$ modulo $4$. Thus, there is no such elliptic curve that would have a $([2,2],[4],[4],[4])$ torsion configuration in its $T_4$-isogeny graph.
		
		One can also note that $\left(\begin{array}{cc} 1 & 2 \\ 2 & 3 \end{array} \right)$ is not conjugate by an element of $\operatorname{GL}(2,\Z/4\Z)$ to either $\left(\begin{array}{cc} 1 & 0 \\ 0 & -1 \end{array} \right)$ or $\left(\begin{array}{cc} 1 & 1 \\ 0 & -1 \end{array} \right)$, the elements that represent complex conjugation in $\operatorname{GL}(2,\Z/4\Z)$ (see Lemma 2.8 in \cite{sutherlandzywina}).
		
	\end{itemize}
	The other possible configurations, namely $([2,2],[4],[4],[2])$, $([2,2],[4],[2],[2])$, and $([2,2],[2],[2],[2])$, all occur for CM and non-CM curves, as our tables show.
	
	\begin{remark}
		We thank the referee for the following note. The fact that there is no elliptic curve over $\Q$ whose image of $\rho_{E, 4}$ is conjugate to the subgroup $H\subseteq \GL(2,\Z/4\Z)$ can be deduced from the fact that the corresponding quotient of $X(4)$ is a conic with affine equation $x^2+y^2=-1$ which has no real (therefore no rational) points.
	\end{remark}

	\section{Intermediary Results for Curves with Torsion $\mathbb{Z} / 2 \mathbb{Z} \times \mathbb{Z} / 2^{N} \mathbb{Z}$}\label{sec-morelemmas}

	For the following lemmas, let $E / \Q$ be an elliptic curve such that $C(E) = C_{2}(E) \geq 4$ so that, in particular, $E$ has no $\Q$-rational subgroups of odd prime order. Then, by Lemma \ref{lem-4tors-implies-bicyclic-isog}, we can assume without loss of generality that $E$ has full two torsion defined over $\Q$. Thus, put  $E[2] = E(\Q)[2] = \langle P_2, Q_2\rangle$. For each $m \geq 1$, let $P_{2^{m+1}}, Q_{2^{m+1}} \in E[2^{m+1}]$ such that $[2]P_{2^{m+1}} = P_{2^{m}}$ and $[2]Q_{2^{m+1}} = Q_{2^{m}}$. Note that then $E[2^{m}] = \langle P_{2^{m}}, Q_{2^{m}}\rangle$ for all $m \geq 1$. Suppose $E(\Q)_{tors} = \langle P_2, Q_{2^{N}}\rangle \cong \Z/2\Z \times \Z/2^{N}\Z$ with $N \geq 1$. In addition, we assume that $E$ is not isogenous to an elliptic curve with a torsion subgroup of order $2^{N+2}$.

	We remark here that since $E$ has its full two-torsion defined over $\Q$, every $\Q$-isogenous curve to $E$ has at least one $2$-torsion point defined over $\Q$, by Lemma \ref{lem-2torspt-all-have-2torspt}.
	
	\begin{remark} \label{split cartan 4} Let $E / \Q$ be as above with $E(\Q)_{\text{tors}} \cong \Z / 2 \Z \times \Z / 2^{N} \Z$ for some $N \geq 1.$ Suppose the image of $\rho_{E, 4}$ is a split Cartan subgroup, i.e., the mod-$4$ Galois image is of the form
		
		$$ \left\{\left(\begin{array}{cc}
		\ast & 0 \\
		0 & \ast
		\end{array} \right) \right\}.$$
		We can assume up to a change of basis of $E[4]$, that $\langle P_{4} \rangle$ and $\langle Q_{4} \rangle$ are $\Q$-rational groups of order $4$. (Note that this is enough to conclude that there are eight $\Q$-rational groups, namely, $\{ \mathcal{O} \}, \langle P_{2} \rangle, \langle Q_{2} \rangle, \langle P_{2} + Q_{2} \rangle, \langle P_{4} \rangle, \langle P_{4} + Q_{2} \rangle, \langle Q_{4} \rangle,$ and $\langle P_{2} + Q_{4} \rangle.)$ Let $\phi \colon E \to E'$ be an isogeny with kernel $\langle P_{2} \rangle.$ 
		
		Note that by Lemma \ref{orbit-of-supergroups-of-rational-groups}, for each $\sigma \in G_{\Q},$ there is $c \in \Z / 8 \Z$ such that $\sigma(Q_{8}) - [c]Q_{8} \in \langle P_{2}, Q_{2} \rangle.$ In other words, $\langle \phi(Q_{8}) \rangle$ is a $\Q$-rational group of order $8$ as for each $\sigma \in G_{\Q},$ we have $a, b, c \in \Z / 8 \Z$ such that
		$$\sigma(\phi(Q_{8})) = \phi(\sigma(Q_{8})) = \phi([a]P_{2} + [b]Q_{2} + [c]Q_{8}) = \phi([4b]Q_{8} + [c]Q_{8}) = \phi([4b + c]Q_{8}) = [4b + c]\phi(Q_{8})$$
		Thus, $\langle \phi(Q_{8}) \rangle$ is a $\Q$-rational group of $E'$ of order $8$. By the fact that $\langle P_{4} \rangle$ is $\Q$-rational and Lemma \ref{orbit-of-supergroups-of-rational-groups}, $\sigma(P_{4}) - P_{4} \in \langle P_{2} \rangle$ for all $\sigma \in G_{\Q}$ and thus, by Lemma \ref{lem-necessity-for-point-rationality}, $\phi(P_{4})$ is a point of order $2$ defined over $\Q.$ Moreover, $\phi(P_{4}) \notin \langle \phi(Q_{8}) \rangle$ as if it were, then $\phi(P_{4}) = \phi(Q_{2})$ and thus, $\phi(P_{4}) - \phi(Q_{2}) = \phi(P_{4} - Q_{2}) = \mathcal{O}$, a contradiction. Note that by Lemma \ref{lem-necessity-for-point-rationality}, $\phi(Q_{2})$ is a point of order $2$ defined over $\Q$ that is not equal to $\phi(P_{4})$. Thus, if $E(\Q)_{\text{tors}} \cong \Z / 2 \Z \times \Z / 2 \Z,$ then $E'(\Q)_{\text{tors}} \cong \Z / 2 \Z \times \Z / 2 \Z.$ If $E(\Q)_{\text{tors}} \cong \Z / 2 \Z \times \Z / 4 \Z$ with $Q_{4}$ being a point of order $4$ defined over $\Q$, then $E'(\Q)_{\text{tors}} \cong \Z / 2 \Z \times \Z / 4 \Z$, with $\phi(Q_{4})$ being the point of order $4$ defined over $\Q$ (remember there are no elements in the isogeny class whose torsion subgroup is greater in order than $E(\Q)_{\text{tors}}$).
		
		Thus, $E(\Q)_{\text{tors}} \cong E'(\Q)_{\text{tors}}$ with the main difference being that the maximum isogeny degree amongst the isogenies generated by finite, cyclic $\Q$-rational subgroups of $E$ is $4$ and the maximum isogeny degree amongst the isognies generate by finite, cyclic $\Q$-rational subgroup of $E'$ is $8.$ With this in mind, if $C(E) = C_2(E) = 6$ or $8$ we will always build our isogeny-torsion graph starting with the curve with the largest bicyclic torsion subgroup. Moreover, if two or more curves have the same largest bicyclic torsion subgroup, start with the curve with the largest bicyclic torsion subgroup with largest finite cyclic isogeny degree amongst the elements of the isogeny class. For some of our isogeny-torsion graphs, this curve is not unique, but that is not a problem.
		
		From now on, our base curve $E / \Q$ will always have full two-torsion. For sections 10, 11, and 12, if $C(E) = 4$, then the finite cyclic $\Q$-rational subgroups of $E$ are the four point-wise rational groups generated by elements of $E[2] = \left\{ \mathcal{O}, P_{2}, Q_{2}, P_{2}+Q_{2} \right\}$. For sections 10, 11, and 12, if $C(E) = 6$, then we shall assume, in addition, that $\langle Q_{4} \rangle$ and $\langle P_{2} + Q_{4} \rangle$ are $\Q$-rational. For sections 10, 11, and 12, if $C(E) = 8$, we shall assume, in addition, that $\langle Q_{8} \rangle$ and $\langle P_{2} + Q_{8} \rangle$ are $\Q$-rational.
		
	\end{remark}

	\begin{lemma} \label{lem-Z2xZ2}
		
		Let $E / \Q$ be as above with $E(\Q)_{\text{tors}} \cong \mathbb{Z} / 2 \mathbb{Z} \times \mathbb{Z} / 2^N \mathbb{Z}$ and $N \geq 1,$ and such that $E$ has the largest bicyclic rational torsion subgroup in its isogeny class. In other words, no element of the isogeny class of $E$ has torsion subgroup of order $2^{N+2}$. Suppose $\langle Q_{4} \rangle$ is $\Q$-rational. 
		
		\begin{enumerate}
			
			\item If $\langle Q_{8} \rangle$ is not $\Q$-rational, then $E / \langle Q_{2}\rangle (\Q)_{\text{tors}} \cong \Z / 2 \Z \times \Z / 2 \Z$.
			
			\item If $\langle Q_{8} \rangle$ is $\Q$-rational, then $E / \langle Q_{4}\rangle (\Q)_{\text{tors}} \cong \Z / 2 \Z \times \Z / 2 \Z$.
			
		\end{enumerate}
	\end{lemma}

	\begin{proof}

	    Let $E$ and $Q_{4}$ be as in the statement. Let $M = 3$ if $Q_{8}$ generates a $\Q$-rational group and let $M = 2$ if $Q_{8}$ does not generate a $\Q$-rational group. In other words, $Q_{2^M} = Q_{8}$ if $Q_{8}$ generates a $\Q$-rational group and $Q_{2^M} = Q_{4}$ if $Q_{8}$ does not generate a $\Q$-rational group. In either case, $\left\langle Q_{2^{M}} \right\rangle$ and $\langle P_2+Q_{2^M}\rangle$ are the largest finite cyclic $\Q$-rational subgroups of $E$ that contain $Q_{4}$. Let $Q_{2^{M-1}} = [2]Q_{2^{M}}$. In other words, $Q_{2^{M-1}} = Q_{2}$ if $Q_{8}$ does not generate a $\Q$-rational group and $Q_{2^{M-1}} = Q_{4}$ if $Q_{8}$ does generate a $\Q$-rational group. We claim that $E / \left\langle Q_{2^{M-1}} \right\rangle(\Q)_{\text{tors}} \cong \Z / 2 \Z \times \Z / 2 \Z$. We first show that $E / \left\langle Q_{2^{M-1}} \right\rangle$ has full two-torsion.
	    
	    Let $\phi\colon E \to E'=E/\langle Q_{2^{M-1}} \rangle$. Then,  $\phi(P_{2})$ is a rational point of order $2$ by Lemma \ref{lem-necessity-for-point-rationality}. From $\langle Q_{2^{M}} \rangle$ being $\Q$-rational and Lemma \ref{orbit-of-supergroups-of-rational-groups}, it follows that $\sigma(Q_{2^{M}}) - Q_{2^{M}} \in \langle Q_{2^{M-1}} \rangle$ for all $\sigma \in G_{\Q}$. By Lemma \ref{lem-necessity-for-point-rationality}, $\phi(Q_{2^M})$ is a rational point of order $2$. Moreover, $\phi(Q_{2^M}) \neq \phi(P_{2})$ as then $Q_{2^M} - P_{2} \in \langle Q_{2} \rangle$ but $Q_{2^{M}} - P_{2}$ has order $4$ or $8$ and $Q_{2}$ has order $2$. It follows that all $2$-torsion elements of $E'$, namely, $\mathcal{O}$, $\phi(P_{2})$, $\phi(Q_{2^M})$, and $\phi(P_{2} + Q_{2^M})$ are defined over $\Q$.
		
		If $E(\Q)_{\text{tors}} \cong \Z / 2 \Z \times \Z / 2 \Z$, then we are done because we assumed we had no isogenous elliptic curve with bicyclic torsion group bigger than the torsion group of the curve in our hypothesis (and so, we must have $E'[2] = E'(\Q)_{\text{tors}} \cong \Z / 2 \Z \times \Z / 2 \Z$). Otherwise, suppose $Q_{4}$ is defined over $\Q$, so that $N\geq 2$. We analyze if any point of $E'$ of order $4$ is rational:
		
		Let $\phi(P)$ and $\phi(P')$ be any pair of unequal elements of $E'$ of order $2$. Remember both of these points of order $2$ are defined over $\Q$. Let $B \in E$ such that $[2]B = P$. We claim that the two subgroups of $E'$ of order $4$ that contain $\phi(P)$ are $\left\langle \phi(B) \right\rangle$ and $\left\langle \phi(P'+B) \right\rangle$. Indeed, both $\left\langle \phi(B) \right\rangle$ and $\left\langle \phi(P'+B) \right\rangle$ are order $4$ and contain $\phi(P)$. If we assume that $\left\langle \phi(B) \right\rangle$ and $\left\langle \phi(P'+B) \right\rangle$ are equal, then either $\phi(B) = \phi(P'+B)$ and hence, $\phi(P') = \mathcal{O}$, a contradiction, or $[3]\phi(B) = \phi([3]B) = \phi(P'+B)$ and so, $\phi([2]B) = \phi(P) = \phi(P')$, another contradiction.
		
		Now we claim that neither $\phi(B)$ nor $\phi(P'+B)$ are defined over $\Q$. By the fact that $\phi(P')$ is defined over $\Q$, showing $\phi(B)$ is not defined over $\Q$ is enough. If $\phi(B)$ is defined over $\Q$, then $\sigma(B) - B \in \left\langle Q_{2^{M-1}} \right\rangle$ for all $\sigma \in G_{\Q}$. If $P = P_{2}$, then by Lemma \ref{orbit-of-supergroups-of-rational-groups}, $\sigma(B) - B \in \left\langle Q_{2} \right\rangle$ for all $\sigma \in G_{\Q}$. Thus, $E[4] = \left\langle B, Q_{4} \right\rangle$ and the fact that $Q_{4}$ is  defined over $\Q$, this causes a contradiction as noted in Lemma \ref{lem-determinants}. If $P = Q_{2^{M}}$ or $P_{2} + Q_{2^{M}}$, then $Q_{2^{M-1}} \in \left\langle B \right\rangle$ and hence, by Lemma \ref{lem-necessity-for-subgroup-rationality}, we have that $\left\langle B \right\rangle$ is $\Q$-rational but that contradicts the condition that $\left\langle Q_{2^{M}} \right\rangle$ and $\langle P_2+Q_{2^M}\rangle$ are the largest finite cyclic $\Q$-rational subgroups of $E$ that contain $Q_{4}$. Hence, none of the groups that contain $\phi(P)$ of order $4$ are defined over $\Q$. As $P$ and $P'$ were arbitrary, we have proven that $E / \left\langle Q_{2^{M}} \right\rangle(\Q)_{\text{tors}} \cong \Z / 2 \Z \times \Z / 2 \Z$ and the proof is complete.
		
	\end{proof}
	
	\begin{lemma} \label{lem-Z2}
		
		Let $E / \Q$ be an elliptic curve having the largest bicyclic torsion subgroup and maximum isogeny degree amongst the isogenies generated by finite cyclic $\Q$-rational groups amongst all the elements of the isogeny class. Suppose $E(\Q)_{\text{tors}} \cong \Z / 2 \Z \times \Z / 4 \Z$ or $\Z / 2 \Z \times \Z / 8 \Z$.
		\begin{enumerate}
			\item If $\langle Q_{8} \rangle$ is $\Q$-rational, then $E / \langle Q_{8}\rangle (\Q)_{\text{tors}}$ and $E /\langle  P_{2} + Q_{8}\rangle  (\Q)_{\text{tors}}$ are cyclic of order $2$. 
			
			\item If $\langle Q_{8} \rangle$ is not $\Q$-rational, then $E / \langle Q_{4}\rangle (\Q)_{\text{tors}}$ and $E / \langle P_{2} + Q_{4}\rangle (\Q)_{\text{tors}}$ are cyclic of order $2$.
		\end{enumerate}
	\end{lemma}
	
	\begin{proof}
		
		Let $E / \Q$ be an elliptic curve with $E[4]=\langle P_4,Q_4\rangle$, such that $P_2=[2]P_4$ and $Q_{4}$ are defined over $\Q$. Let $Q_{2^M} = Q_{8}$ if $Q_{8}$ generates a $\Q$-rational group and let $Q_{2^M} = Q_{4}$ if $Q_{8}$ does not generate a $\Q$-rational group. In either case, $\left\langle Q_{2^M} \right\rangle$ and $\left\langle P_{2} + Q_{2^M} \right\rangle$ are the largest finite cyclic  $\Q$-rational subgroups of $E$ by order. To generalize, let $A = Q_{2^M}$ or $P_{2} + Q_{2^M}$. We need to prove that $E / \left\langle A \right\rangle(\Q)_{\text{tors}} \cong \Z / 2 \Z$. By Lemma \ref{Maximality-Of-Rational-2-Power-Groups}, $E / \langle A \rangle (\Q)_{\text{tors}}$ is cyclic. Let $B \in E$ such that $[2]B = A$. Note that all elements of $\left\langle B \right\rangle$ of order at most $4$ are defined over $\Q$.
		
		Let $\phi \colon E \to E' = E /\langle A \rangle$. Then, $\phi(P_{2})$ is the point of order $2$ defined over $\Q$ on $E'$ by the third statement of Lemma \ref{lem-necessity-for-point-rationality}. The subgroups of order $4$ of $E'$ that contain $\phi(P_{2})$ are $\langle \phi(P_{4}) \rangle$ and $\langle \phi(P_{4} + B) \rangle$. The groups $\left\langle \phi(P_{4}) \right\rangle$ and $\left\langle \phi(P_{4}+B) \right\rangle$ are in fact distinct because if not, then $\phi(P_{4}) = \phi(P_{4} + B)$ and hence, $\phi(B) = \mathcal{O}$ so $B \in \left\langle A \right\rangle$ or $[3]\phi(P_{4}) = \phi([3]P_{4}) = \phi(P_{4} + B)$ and so, $\phi([2]P_{4}) = \phi(P_{2}) = \phi(B)$. This means that $B - P_{2} \in \left\langle A \right\rangle$. But both $\left\langle B \right\rangle$ and $\left\langle B - P_{2} \right\rangle$ properly contain $A$!
		
		Now, if $\phi(P_{4})$ were rational, then by Lemma \ref{lem-necessity-for-point-rationality}, we have $\sigma(P_{4}) - P_{4} \in \langle A \rangle$ for all $\sigma \in G_{\Q}$. By Lemma \ref{orbit-of-supergroups-of-rational-groups}, we have $\sigma(P_{4}) - P_{4} \in \langle Q_{2} \rangle$ for all $\sigma \in G_{\Q}$. But as $Q_{4}$ is defined over $\Q$, this contradicts Lemma \ref{lem-determinants}. If instead $\phi(P_{4} + B)$ were defined over $\Q$, then
		
		\vspace{5mm}
		
		\begin{center} $\sigma(P_{4} + B) - (P_{4} + B) \in \langle A \rangle$ for all $\sigma \in G_{\Q}$. \end{center}
		
		\vspace{5mm}
		
		Now multiply through by $\frac{\lvert B \rvert}{4}$. Note that $\left[\frac{\lvert B \rvert}{4}\right] B$ is an element of $\left\langle B \right\rangle$ of order $4$, hence, is defined over $\Q$. Note that $\left[\frac{\lvert B \rvert}{4}\right] P_{4} = P_{2}$ or $\mathcal{O}$ and thus, is also defined over $\Q$. Finally, note that $\left[\frac{\lvert B \rvert}{4}\right]A = \left[\frac{\lvert B \rvert}{4}\right][2]B = \left[\frac{\lvert B \rvert}{2}\right]B$ is an element of $\left\langle B \right\rangle$ of order $2$. Thus, $\frac{\lvert B \rvert}{4}(\sigma(P_{4}) + B) - (P_{4} + B) = \mathcal{O}$ while $\left[\frac{\lvert B \rvert}{4}\right]A$ is non-zero, so a refinement of the above equation is
		
		\vspace{5mm}
		
		\begin{center} $\sigma(P_{4} + B) - (P_{4} + B) \in \langle [2]A \rangle$ for all $\sigma \in G_{\Q}$. \end{center}
		
		\vspace{5mm}
		
		Note that $[2]A = [4]B = [4](P_{4} + B)$ so by Lemma \ref{lem-necessity-for-subgroup-rationality}, we would have that $\langle P_{4} + B \rangle$ is $\Q$-rational but this is a contradiction as $\left\langle A \right\rangle$ and $\langle P_2+A\rangle$ should be the largest finite cyclic  $\Q$-rational subgroup of $E$ by order. Thus, we can conclude $E / \left\langle Q_{4} \right\rangle (\Q)_{\text{tors}} \cong E / \left\langle P_{2} + Q_{4} \right\rangle (\Q)_{\text{tors}} \cong \Z / 2 \Z$ when $\left\langle Q_{8} \right\rangle$ is not $\Q$-rational and $E / \left\langle Q_{8} \right\rangle (\Q)_{\text{tors}} \cong E / \left\langle P_{2} + Q_{8} \right\rangle (\Q)_{\text{tors}} \cong \Z / 2 \Z$ when $\left\langle Q_{8} \right\rangle$ is $\Q$-rational.
	
	\end{proof}
	%% Garen's proof of part (2) is at the end of this document!
	
	\begin{lemma} \label{lem-Z2xZ4} Let $E(\Q)_{\text{tors}} \cong \Z / 2 \Z \times \Z / 4 \Z$ or $\Z / 2 \Z \times \Z / 8 \Z$, with $E[8]=\langle P_8,Q_8\rangle$ (i.e., $Q_{4}=[2]Q_8$ is defined over $\Q$ but $Q_{8}$ may or may not be). If $\langle Q_{8} \rangle$ is a $\Q$-rational group, then $E/\langle Q_{2}\rangle (\Q)_{\text{tors}} \cong \mathbb{Z} / 2 \mathbb{Z} \times \mathbb{Z} / 4 \mathbb{Z}$.
	\end{lemma}
	
	\begin{proof}
		Let $E/\Q$ be an elliptic curve as in the statement. Suppose $Q_{4}$ is defined over $\Q$ and $Q_{8}$ generates a $\Q$-rational group of order $8$ that is not necessarily point-wise rational. By Lemma \ref{orbit-of-supergroups-of-rational-groups}, $\sigma(Q_{8}) - Q_{8} \in \left\langle Q_{4} \right\rangle$ for all $\sigma \in G_{\Q}$. Let $\phi \colon E \to E / \left\langle Q_{2} \right\rangle$ be an isogeny with kernel $\left\langle Q_{2} \right\rangle$, then by the third statement of Lemma \ref{lem-necessity-for-point-rationality}, the point $\phi(P_{2})$ is rational of order $2$ and by the first statement of Lemma \ref{lem-necessity-for-point-rationality}, $\phi(Q_{8})$ is a point of order $4$ defined over $\Q$. $\phi(Q_{8})$ and $\phi(P_{2} + Q_{8})$ are the rational points of order $4$ that generate the two distinct point-wise rational groups of order $4$.
		
		To generalize, let $A$ be either $Q_{8}$ or $P_{2} + Q_{8}$. The subgroups of $E'$ of order $8$ that contain $\phi(A)$ are $\langle \phi(B) \rangle$ and $\langle \phi(P_{2} + B) \rangle$ where $B \in E$ such that $[2]B = A$. To prove that $\phi(B)$ and $\phi(P_{2}+B)$ in fact generate distinct groups of order $8$, assume there is an integer $n \in \Z$ such that $[n]\phi(B) = \phi([n]B) = \phi(P_{2}+B)$. Then $P_{2}+B - [n]B = P_{2} + [1-n]B \in \left\langle Q_{2} \right\rangle$. Multiplying through by $2$ gets us $[1-n]A = \mathcal{O}$ and hence, $8$ divides $1-n$ and so, $P_{2} + [1-n]B = P_{2}$ or $P_{2} + Q_{2}$, neither of which are elements of $\left\langle Q_{2} \right\rangle$.
		
		Now to prove that $\phi(B)$ and $\phi(P_{2}+B)$ are not defined over $\Q$. By the rationality of $\phi(P_{2})$, all we have to do is check the rationality of $\phi(B)$. Note that $Q_{2} \in \left\langle B \right\rangle$ so if $\phi(B)$ is defined over $\Q$, then $\left\langle B \right\rangle$ is a $\Q$-rational subgroup of $E$ of order $16$ by Lemma \ref{lem-necessity-for-subgroup-rationality}. This means that the groups generated by $[2^{a}]P_{2} + [2^{b}]B$ for $0 \leq a \leq 1$ and $0 \leq b \leq 4$ are all $\Q$-rational and pairwise distinct. This gives us at least $10$ $\Q$-rational groups, contradicting Theorem \ref{thm-kenku}.
	\end{proof}
	
	\begin{lemma} \label{Not-too-many-Z/4}
		
		Let $E / \Q$ be a rational elliptic curve with $E(\Q)_{\text{tors}} \cong \Z / 2 \Z \times \Z / 2 \Z$, with $E[4]=\langle P_4,Q_4\rangle$ and suppose $\langle Q_{4} \rangle$ is $\Q$-rational, with $C(E)=C_2(E)=6$ or $8$. Further, assume that $E$ has the largest bicyclic torsion subgroup and isogeny degree in its class (see Remark \ref{split cartan 4}).
		\begin{enumerate}
			\item $E / \langle P_{2}\rangle  (\Q)_{\text{tors}}$ and $E / \langle P_{2} + Q_{2}\rangle (\Q)_{\text{tors}}$ are cyclic of order $2$ or $4$ but not both order $4$.
			\item If $\langle Q_{8} \rangle$ is not $\Q$-rational, then $E / \langle Q_{4}\rangle  (\Q)_{\text{tors}}$ and $E /\langle  P_{2} + Q_{4}\rangle (\Q)_{\text{tors}}$ are cyclic of order $2$ or $4$ but not both order $4$.
			\item If $\langle Q_{8} \rangle$ is $\Q$-rational, then $E / \langle Q_{8}\rangle  (\Q)_{\text{tors}}$ and $E / \langle P_{2} + Q_{8}\rangle (\Q)_{\text{tors}}$ are cyclic of order $2$ or $4$ but not both order $4$.
		\end{enumerate}
	\end{lemma}
	
	\begin{proof}
	
		Let $E/\Q$ be as in the statement. It follows that the only $\Q$-rational subgroups are $\langle [2^a]P_2+[2^b]Q_4\rangle$ for $0\leq a \leq 1$ and $0\leq b \leq 2$, if $C(E)=6$ and $\langle [2^a]P_2+[2^b]Q_8\rangle$ for $0\leq a \leq 1$ and $0\leq b \leq 3$ if $C(E)=8$.  In particular, by Lemma \ref{Maximality-Of-Rational-2-Power-Groups}, we have that $E / \langle P_{2}\rangle (\Q)_{\text{tors}}$ and $E / \langle P_{2} + Q_{2}\rangle (\Q)_{\text{tors}}$ are both cyclic. By Lemma \ref{lem-2torspt-all-have-2torspt}, both have a point of order $2$ or $4$ defined over $\Q$ but no points of order $8$ by Lemma  \ref{lem-Q-Rational-Grps-Order-8} and by our assumption that E has the largest bicyclic torsion subgroup. So let us assume for a contradiction  that both are of order $4$. Let $\phi\colon E\to E/\langle P_2\rangle$ and $\phi'\colon E\to E/\langle P_2+Q_2\rangle$ be isogenies with kernel $\left\langle P_{2} \right\rangle$ and $\left\langle P_{2} + Q_{2} \right\rangle$, respectively. By the third statement of Lemma \ref{lem-necessity-for-point-rationality},  the points $\phi(Q_{2})$ and $\phi'(Q_{2})$ are both defined over $\Q$. The groups of order $4$ that contain $\phi(Q_{2})$ are $\langle \phi(Q_{4}) \rangle$ and $\langle \phi(P_{4} + Q_{4}) \rangle$. If $\phi(Q_{4})$ were defined over $\Q$, then by Lemma \ref{lem-necessity-for-point-rationality}, $\sigma(Q_{4}) - Q_{4} \in \langle P_{2} \rangle$ for all $\sigma \in G_{\Q}.$ But we already know that $\langle Q_{4} \rangle$ is $\Q$-rational, so that means $Q_{4}$ must be defined over $\Q$, a contradiction. If $\phi(P_{4} + Q_{4})$ is defined over $\Q$, then by Lemma \ref{lem-necessity-for-point-rationality}, we have  $\sigma(P_{4} + Q_{4}) - (P_{4} + Q_{4}) \in \langle P_{2} \rangle$ for all $\sigma \in G_{\Q}.$

		The groups of order $4$ that contain $\phi'(Q_{2})$ are $\langle \phi'(Q_{4}) \rangle$ and $\langle \phi'(P_{4}) \rangle$. If $\phi'(Q_{4})$ were defined over $\Q$, then by Lemma \ref{lem-necessity-for-point-rationality}, we have $\sigma(Q_{4}) - Q_{4} \in \langle P_{2} + Q_{2} \rangle$ for all $\sigma \in G_{\Q}$. But we already know that $\langle Q_{4} \rangle$ is $\Q$-rational, which means that $Q_{4}$ must be defined over $\Q$, a 
		contradiction. If $\phi'(P_{4})$ were defined over $\Q$, then $\sigma(P_{4}) - P_{4} \in \langle P_{2} + Q_{2} \rangle$ for all $\sigma \in G_{\Q}$.
		
		Thus, if both $E / \langle P_{2}\rangle (\Q)_{\text{tors}}$ and $E / \langle P_{2} + Q_{2}\rangle (\Q)_{\text{tors}}$ are both cyclic of order $4$, then 
		$\sigma(P_{4}) - P_{4} \in \langle P_{2} + Q_{2} \rangle$ and $\sigma(P_{4} + Q_{4}) - (P_{4} + Q_{4}) \in \langle P_{2} \rangle$ for all $\sigma \in G_{\Q}$, and, recall that $\langle Q_4\rangle$ is assumed to be $\Q$-rational. Then we have for an arbitrary $\sigma \in G_{\Q}$, that either $\sigma(P_4)=P_4$ and $\sigma(Q_4)=Q_4$, or $\sigma(P_4)=[3]P_4+[2]Q_4$ and $\sigma(Q_4)=[3]Q_4$. Thus, the image of $\rho_{E,4}$ is of the form 
		$$\left\{\left(\begin{array}{cc} 1 & 0 \\ 0 & 1 \end{array} \right), \left(\begin{array}{cc} 3 & 0 \\ 2 & 3 \end{array} \right) \right\},$$
		but the determinant would not be surjective onto $(\Z/4\Z)^\times$, and we reach a contradiction. This proves (1).
		
		For (2), let $Q_{2^M} = Q_{4}$ if $Q_{8}$ does not generate a $\Q$-rational group and let $Q_{2^M} = Q_{8}$ if $Q_{8}$ does generate a $\Q$-rational group. Then the finite cyclic $\Q$-rational groups of maximal order are $\left\langle Q_{2^M} \right\rangle$ and $\left\langle P_{2} + Q_{2^M} \right\rangle$. Let $Q_{2^{M+1}} \in E$ such that $[2]Q_{2^{M+1}} = Q_{2^M}$. $E / \left\langle Q_{2^M} \right\rangle(\Q)_{\text{tors}}$ and $E / \left\langle P_{2} + Q_{2^M} \right\rangle(\Q)_{\text{tors}}$ are cyclic by Lemma \ref{Maximality-Of-Rational-2-Power-Groups}, of order $2$ or $4$ by Lemma \ref{lem-2torspt-all-have-2torspt}, but not $8$ by Lemma \ref{lem-Q-Rational-Grps-Order-8} and our assumption that $E$ has the largest bicyclic torsion subgroup. Let us assume for a contradiction that both groups are of order $4$. If we let $\phi\colon E\to E/\langle Q_{2^M} \rangle$ and $\phi'\colon E\to E/\langle P_2+Q_{2^M}\rangle$ be isogenies with kernel $\left\langle Q_{2^M} \right\rangle$ and $\left\langle P_{2} + Q_{2^M} \right\rangle$ respectively, then $\phi(P_{2})$ and $\phi'(P_{2})$ are the elements of order $2$, respectively by the third statement of Lemma \ref{lem-necessity-for-point-rationality}. The groups of order $4$ that contain $\phi(P_{2})$ are $\langle \phi(P_{4}) \rangle$ and $\langle \phi(P_{4} + Q_{2^{M+1}}) \rangle$. The groups of order $4$ that contain $\phi'(P_{2})$ are $\langle \phi'(P_{4}) \rangle$ and $\langle \phi'(Q_{2^{M+1}}) \rangle$. If $\phi(P_{4})$ is defined over $\Q$, then $\sigma(P_{4}) - P_{4} \in \left\langle Q_{2^{M+1}} \right\rangle$ for all $\sigma \in G_{\Q}$ by Lemma \ref{lem-necessity-for-point-rationality} and by Lemma \ref{orbit-of-supergroups-of-rational-groups}, we have $\sigma(P_{4}) - P_{4} \in \left\langle Q_{2} \right\rangle$ for all $\sigma \in G_{\Q}$ but then $E / \left\langle Q_{2} \right\rangle$ has a point of order $4$, namely, $P_{4} + \left\langle Q_{2} \right\rangle$. Thus, $E / \left\langle Q_{2} \right\rangle(\Q)_{\text{tors}} \cong \Z / 2 \Z \times \Z / 4 \Z$, contradicting our assumption that there are no bicyclic torsion subgroups of elements in the isogeny class of order greater than $4$. For a similar reason, $\phi'(P_{4})$ cannot be defined over $\Q$. Let's assume both $\phi(P_{4}+Q_{2^{M+1}})$ and $\phi'(Q_{2^{M+1}})$ are defined over $\Q$. Thus, $\sigma(Q_{2^{M+1}}) - Q_{2^{M+1}} \in \left\langle P_{2} + Q_{2^{M}} \right\rangle$ and $\sigma(P_{4} + Q_{2^{M+1}}) - (P_{4} + Q_{2^{M+1}}) \in \left\langle Q_{2^{M}} \right\rangle$ for all $\sigma \in G_{\Q}$. By \ref{orbit-of-supergroups-of-rational-groups}, $\sigma(P_{4}) - P_{4} \in \left\langle P_{2}, Q_{2} \right\rangle$. Taking everything together, we have for an arbitrary $\sigma \in G_{\Q}$,
		
		\begin{center} $\sigma(P_{4}) = P_{4}$ or $P_{4} + Q_{2}$ $\iff$ $\sigma(Q_{2^{M+1}}) = [1+4r]Q_{2^{M+1}}$ for some $r \in \Z$ \end{center}
		
		\begin{center} $\sigma(P_{4}) = [3]P_{4}$ or $[3]P_{4} + Q_{2}$ $\iff$ $\sigma(Q_{2^{M+1}}) = P_{2} + [3+4s]Q_{2^{M+1}}$ for some $s \in \Z$ \end{center}
		
		Multiplying the right side of both the above conditions by $\frac{\left\lvert Q_{2^{M+1}}\right\rvert}{4}$, we have
		
		\begin{center} $\sigma(P_{4}) = P_{4}$ or $P_{4} + Q_{2}$ $\iff$ $\sigma(Q_{4}) = Q_{4}$ \end{center}
		
		\begin{center} $\sigma(P_{4}) = [3]P_{4}$ or $[3]P_{4} + Q_{2}$ $\iff$ $\sigma(Q_{4}) = [3]Q_{4}$ \end{center}
		
		Thus, the image mod $4$ of the Galois representation associated to $E$ is conjugate to a subgroup of $G = \left\{\left(\begin{array}{cc}
		    1 & 0 \\
		    0 & 1
		\end{array}\right), \left(\begin{array}{cc}
		    1 & 0 \\
		    2 & 1
		\end{array}\right), \left(\begin{array}{cc}
		    3 & 0 \\
		    0 & 3
		\end{array}\right), \left(\begin{array}{cc}
		    3 & 0 \\
		    2 & 3
		\end{array}\right)\right\}$. As the determinant of every element of $G$ is $1$, we have a contradiction. Hence, if $Q_{8}$ does not generate a $\Q$-rational group, then $E / \left\langle Q_{4} \right\rangle(\Q)_{\text{tors}} \cong \Z / 2 \Z$ or $\Z / 4 \Z$ and $E / \left\langle P_{2} + Q_{4} \right\rangle(\Q)_{\text{tors}} \cong \Z / 2 \Z$ or $\Z / 4 \Z$ but not both of the torsion subgroups are order $4$. Also, if $Q_{8}$ generates a $\Q$-rational group, then $E / \left\langle Q_{8} \right\rangle(\Q)_{\text{tors}} \cong \Z / 2 \Z$ or $\Z / 4 \Z$ and $E / \left\langle P_{2} + Q_{8} \right\rangle(\Q)_{\text{tors}} \cong \Z / 2 \Z$ or $\Z / 4 \Z$ but not both of the torsion subgroups are order $4$.
		
	\end{proof}
	
	\section{Elliptic curves with full rational $2$-torsion, and graph of type $T_6$}\label{sec-T6graphs}
	
	In this section we resume our classification of isogeny-torsion graphs, and we deal with curves $E/\Q$ that have full two-torsion and $C(E)=C_2(E)=6$, and therefore the isogeny graph is of type $T_6$. 
	\begin{center} 
		\begin{tikzcd}
			E_{2}=E_1/\langle P_2\rangle &                                       &                             & E_{5}=E_1/\langle Q_4\rangle \\
			& E_{1} \arrow["2",lu, no head] \arrow["2",ld, no head] \arrow["2",r, no head] & E_{4}=E_1/\langle Q_2\rangle \arrow["2",ru, no head] \arrow["2",rd, no head] &       \\
			E_{3}=E_1/\langle P_2+Q_2\rangle &                                       &                             & E_{6}=E_1/\langle P_2+Q_4\rangle
		\end{tikzcd}
	\end{center} 
	We distinguish two cases, according to whether one of $E_1$ or $E_4$ has torsion subgroup isomorphic to $\Z/2\Z\times \Z/4\Z$. 
	
	\subsection{Elliptic curves with $E(\Q)_{\text{tors}}\cong \Z/2\Z\times \Z/2\Z$, and graph of type $T_6$}
	
	Let $E=E_1/\Q$ be an elliptic curve with $E(\Q)_{\text{tors}}\cong \Z/2\Z \times \Z/2\Z$, with $E[4]=\langle P_4,Q_4\rangle$ such that $C(E)=C_2(E)=6$ and the finite cyclic $\Q$-rational groups of $E$ are $\langle \mathcal{O} \rangle, \langle P_{2} \rangle, \langle Q_{2} \rangle, \langle P_{2} + Q_{2} \rangle, \langle Q_{4} \rangle,$ and $\langle P_{2} + Q_{4} \rangle$. Suppose no torsion subgroup of any elliptic curve $\Q$-isogenous to $E$ has order $8$. By Lemma \ref{lem-Z2xZ2}, we have that $E / \langle Q_{2}\rangle (\Q)_{\text{tors}} \cong \Z / 2 \Z \times \Z / 2 \Z$. We proved in Lemma \ref{Not-too-many-Z/4} that $E / \langle Q_{4}\rangle (\Q)_{\text{tors}}$ and $E /\langle  P_{2} + Q_{4}\rangle (\Q)_{\text{tors}}$ are both cyclic of order $2$ or $4$, but they are not both order $4$ at the same time. In addition, the same lemma also says that $E / \langle P_{2}\rangle (\Q)_{\text{tors}}$ and $E / \langle P_{2} + Q_{2}\rangle (\Q)_{\text{tors}}$ are both cyclic of order $2$ or $4$, but they are not both of order $4$. We shall show that only one out of these four groups can be isomorphic to $\Z/4\Z$. For a contradiction, suppose two of them are of order $4$. Note that by a change of $E[4]$ or $E[2]$ bases if necessary, we can assume that $E / \langle Q_{4}\rangle (\Q)_{\text{tors}}$ and $E / \langle P_{2}\rangle (\Q)_{\text{tors}}$ are isomorphic to $\Z/4\Z$.
	
	Let  $\phi\colon E\to E/\langle P_2\rangle$ and $\phi'\colon E\to E/\langle Q_4\rangle$ be isogenies with kernel $\left\langle P_{2} \right\rangle$ and $\left\langle Q_{4} \right\rangle$, respectively. Then, $\phi(Q_{2})$ is the rational element of order $2$ in $E / \langle P_{2}\rangle $ by Lemma \ref{lem-necessity-for-point-rationality}. The groups of order $4$ that contain $\phi(Q_{2})$ are $\langle \phi(Q_{4}) \rangle$ and $\langle \phi(P_{4} + Q_{4}) \rangle$. If $\phi(Q_{4})$ were defined over $\Q$, then $\sigma(Q_{4}) - Q_{4} \in \langle P_{2} \rangle$ for all $\sigma \in G_{\Q}$, and since $\langle Q_4\rangle$ is $\Q$-rational, we also have $\sigma(Q_4)\in \langle Q_4\rangle$. This would make $Q_{4}$ a rational point of $E$, a contradiction, since we assumed that only the $2$-torsion is rational. Thus, for $E / \langle P_{2}\rangle (\Q)_{\text{tors}}$ to be order $4$, we must have that $\phi(P_{4} + Q_{4})$ is defined over $\Q$.
	
	Similarly, $\phi'(P_{2})$ is the rational point of order $2$ in $ E / \langle Q_{4} \rangle$ by Lemma \ref{lem-necessity-for-point-rationality}. The groups of order $4$ that contain $\phi'(P_{2})$ are $\langle \phi'(P_{4}) \rangle$ and $\langle \phi'(P_{4} + Q_{8}) \rangle$. If $\phi'(P_{4})$ was defined over $\Q$, then $\sigma(P_{4}) - P_{4} \in \langle Q_{4} \rangle$ for all $\sigma \in G_{\Q}$. By Lemma \ref{orbit-of-supergroups-of-rational-groups}, we see that $\sigma(P_{4}) - P_{4} \in E[2]$ for all $\sigma \in G_{\Q}$ and thus, $\sigma(P_{4}) - P_{4} \in \langle Q_{2} \rangle.$ Thus, $E / \langle Q_{2}\rangle (\Q)_{\text{tors}} \cong \Z / 2 \Z \times \Z / 4 \Z$ as $E / \left\langle Q_{2} \right\rangle$ has full two-torsion ($P_{2} + \left\langle Q_{2} \right\rangle$, $Q_{4} + \left\langle Q_{2} \right\rangle$, and $P_{2} + Q_{4} + \left\langle Q_{2} \right\rangle$ are distinct elements of order $2$ defined over $\Q$) and a point of order $4$, $P_4 + \langle Q_2\rangle$, defined over $\Q$. We assumed no elliptic curve isogenous to $E$ has a torsion subgroup of order $8$, thus, if $E / \langle Q_{4}\rangle (\Q)_{\text{tors}}$ were to have order $4$, then $\phi'(P_{4} + Q_{8})$ is defined over $\Q$.
	
	Suppose then that both $\phi'(P_{4} + Q_{8})$ and $\phi(P_{4} + Q_{4})$ are defined over $\Q$. Then, by Lemma \ref{lem-necessity-for-point-rationality}, we have $\sigma(P_{4} + Q_{8}) - (P_{4} + Q_{8}) \in \langle Q_{4} \rangle$ and $\sigma(P_{4} + Q_{4}) - (P_{4} + Q_{4}) \in \langle P_{2} \rangle$ for all $\sigma \in G_{\Q}$. In addition $\sigma(Q_4)=[2a+1]Q_4$ for some $a=a(\sigma)\in\Z$, because $\langle Q_4\rangle$ is $\Q$-rational. Thus, for each $\sigma$, there are $a$ and $b$ integers such that $$\sigma(P_4)+[2a+1]Q_4 = P_4+Q_4+[b]P_2$$ 
	and therefore $\sigma(P_4)=[2b+1]P_4 + [2a]Q_4$. Hence, the image of $\rho_{E,4}$ is a subgroup of  order $4$
	$$G_4=\left\{\left(\begin{array}{cc} 2b+1 & 0 \\ 2a & 2a+1 \end{array} \right) : a,b\in \{0,1\} \right\}.$$
	On the other hand, the condition $\sigma(P_{4} + Q_{8}) - (P_{4} + Q_{8}) \in \langle Q_{4} \rangle$ implies that there is some $d\in \Z$ such that 
	$$\sigma(Q_8) = [4b]P_8 + [2d+1]Q_8,$$
	and since $\sigma(Q_4)=[2a+1]Q_4$ it follows that $a\equiv d \bmod 2$. Thus, the condition imposed on $\sigma(Q_8)$ is a lift of the condition on $Q_4$ modulo $4$. Hence, the image of $\rho_{E,8}$ is contained in the full inverse image of $G_4$ in $\GL(2,\Z/8\Z)$, and more concretely it is a subgroup of the following group of order $32$:
	$$G_8=\left\{\left(\begin{array}{cc} e & 4b \\ f & 2d+1 \end{array} \right) : e\equiv 2b+1 \bmod 4, f\equiv 2d\bmod 4 \right\}.$$
	We have searched the Rouse--Zureick-Brown database \cite{rouse} of $2$-adic images for groups in $\GL(2,\Z/8\Z)$ that are conjugates of a subgroup of $G_8$ above, and found none. Thus, the image cannot be contained in $G_8$ and we have reached a contradiction. Hence, only one of $E/\langle P_2\rangle$ or $E/\langle Q_4\rangle$ (or $E / \left\langle P_{2} + Q_{2} \right\rangle$ or $E / \left\langle P_{2} + Q_{4} \right\rangle$) may have $\Z/4\Z$ torsion defined over $\Q$. Thus, the possible torsion configurations for elliptic curves in a $T_6$ graph, such that no elliptic curve has a torsion subgroup of order $8$ are $([2,2],[2],[2],[2,2],[2],[2])$ or $([2,2],[4],[2],[2,2],[2],[2])$, and examples are shown in our tables. 
	
	\subsection{Elliptic curves with $E(\Q)_{\text{tors}}\cong \Z/2\Z\times \Z/4\Z$, and graph of type $T_6$}
	Let $E=E_1 / \Q$ with $E(\Q)_{\text{tors}} \cong \Z / 2 \Z \times \Z / 4 \Z$ and $C(E)=C_2(E)=6$. Thus, we have six finite cyclic $\Q$-rational groups, namely  $\langle \mathcal{O} \rangle, \langle P_{2} \rangle, \langle Q_{2} \rangle, \langle P_{2} + Q_{2} \rangle, \langle Q_{4} \rangle,$ and $\langle P_{2} + Q_{4} \rangle$ and every element of these six groups is defined over $\Q$.
	
	By Lemma \ref{lem-Z2xZ2}, we have that $E /\langle Q_{2}\rangle (\Q)_{\text{tors}} \cong \Z / 2 \Z \times \Z / 2 \Z$ and by Lemma \ref{lem-Z2}, we have $E / \langle Q_{4}\rangle (\Q)_{\text{tors}} \cong E / \langle P_{2} + Q_{4}\rangle (\Q)_{\text{tors}} \cong \Z / 2 \Z.$ Thus, it remains to determine the torsion subgroups of $E /\langle P_{2}\rangle (\Q)$ and $E /\langle  P_{2} + Q_{2}\rangle$. We shall show that they are both either $\Z/4\Z$ or $\Z/8\Z$, but they are not both $\Z/8\Z$ simultaneously.
	
	The torsion subgroups $E /\langle P_{2}\rangle (\Q)_{\text{tors}}$ and $E /\langle  P_{2} + Q_{2}\rangle 
	(\Q)_{\text{tors}}$ are both cyclic by Lemma \ref{Maximality-Of-Rational-2-Power-Groups}. Let  $\phi\colon E\to E/\langle P_2\rangle$ and $\phi'\colon E\to E/\langle  P_{2} + Q_{2}\rangle$ be isogenies with kernel $\left\langle P_{2} \right\rangle$ and $\left\langle P_{2}+Q_{2} \right\rangle$, respectively. Then, the points  $\phi(Q_{4})$ and $\phi'(Q_{4})$ are of order $4$ defined over $\Q$, respectively, by the third statement in Lemma \ref{lem-necessity-for-point-rationality}. The groups of order $8$ that contain $\phi(Q_{4})$ are $\langle \phi(Q_{8}) \rangle$ and $\langle \phi(P_{4} + Q_{8}) \rangle$. The groups of order $8$ that contain $\phi'(Q_{4})$ are $\langle \phi'(Q_{8}) \rangle$ and $\langle \phi'(P_{4} + Q_{8}) \rangle$. If both $\phi(Q_{8})$ and $\phi'(Q_{8})$ are defined over $\Q$, then $\sigma(Q_{8}) - Q_{8} \in \langle P_{2} \rangle$ and $\sigma(Q_{8}) - Q_{8} \in \langle P_{2} + Q_{2} \rangle$ for all $\sigma \in G_{\Q}$. This forces $Q_{8}$ to be defined over $\Q$, a contradiction. Similarly, if both $\phi(P_{4} + Q_{8})$ and $\phi'(P_{4} + Q_{8})$ are defined over $\Q$, then $\sigma(P_{4} + Q_{8}) - (P_{4} + Q_{8}) \in \langle P_{2} \rangle$ and $\sigma(P_{4} + Q_{8}) - (P_{4} + Q_{8}) \in \langle P_{2} + Q_{2} \rangle$ for all $\sigma \in G_{\Q}$. This would mean $P_{4} + Q_{8}$ is defined over $\Q$, a contradiction.
	
	So suppose $\phi(Q_{8})$ and $\phi'(P_{4} + Q_{8})$ are defined over $\Q$ (note that if we change basis of $E[2]$, replacing $P_2$ by $P_2+Q_2$, then this case also takes care of the case of $\phi(P_{4} + Q_{8})$ and $\phi'(Q_{8})$ being defined over $\Q$). Thus for all $\sigma \in G_{\Q}$, we have $\sigma(Q_{8}) - Q_{8} \in \langle P_{2} \rangle$ and  $\sigma(P_{4} + Q_{8}) - (P_{4} + Q_{8}) \in \langle P_{2} + Q_{2} \rangle$. Thus, for each $\sigma$, there are integers $a=a(\sigma)$ and $b=b(\sigma)$ such that $\sigma(P_4+Q_8)=P_4+Q_8+[a](P_2+Q_2)$ and $\sigma(Q_8)=Q_8+[b]P_2$. In particular, $\sigma(P_4)= [2a-2b+1]P_4+[2a]Q_4$. Hence, the image of $\rho_{E,8}$ is contained in the full inverse image of $G_4$ in $\GL(2,\Z/8\Z)$
	$$G_4=\left\{\left(\begin{array}{cc} 2(a-b)+1 & 0 \\ 2a & 1 \end{array} \right) : a,b\in \{0,1\} \right\},$$
	and more concretely it is conjugate to a subgroup of the following group of order $16$:
	$$G_8=\left\{\left(\begin{array}{cc} 2(a-b)+1 & 4b \\ 2a & 1 \end{array} \right): a,b\in \Z/8\Z  \right\}\subseteq \GL(2,\Z/8\Z).$$
	We have searched the Rouse--Zureick-Brown database \cite{rouse} of $2$-adic images for groups in $\GL(2,\Z/8\Z)$ that are conjugates of a subgroup of $G_8$ above, and found none. Thus, the image cannot be contained in $G_8$ and we have reached a contradiction. Hence, only one of $E/\langle P_2\rangle$ or $E/\langle P_2 + Q_2\rangle$ may have $\Z/8\Z$ torsion defined over $\Q$. Thus, the possible torsion configurations for elliptic curves in a $T_6$ graph, such that the $E(\Q)_\text{tors}\cong \Z/2\Z\times\Z/4\Z$ are $([2,4],[4],[4],[2,2],[2],[2])$ or $([2,4],[8],[4],[2,2],[2],[2])$, and examples are shown in our tables. 
	
	\section{Elliptic curves with full rational $2$-torsion, and graph of type $T_8$}\label{sec-T8graphs}
	
	Let $E=E/\Q$ be an elliptic curve with full two-torsion and $C(E)=C_2(E)=8$. As we saw in Section \ref{sec-whatgraphs}, the non-trivial finite cyclic  $\Q$-rational subgroups are $\left\langle \mathcal{O} \right\rangle$, $\langle P_2 \rangle$, $\langle P_2+Q_2 \rangle$, $\langle Q_2 \rangle$, $\left\langle Q_{4} \right\rangle$, $\langle P_2+Q_4 \rangle$, $\langle Q_8 \rangle$, and $\langle P_2+Q_8 \rangle$. The corresponding isogeny graph is as follows:
	
			\begin{center}
		\begin{tikzcd}
			& \substack{E_2=\\E_1/\langle P_2\rangle}                                &                            & \substack{E_7=\\E_1/\langle Q_8 \rangle}                     &       \\
			& E_{1} \arrow["2",u, no head] \arrow["2",ld, no head] \arrow["2",rd, no head] &                            &  \substack{E_6=\\E_1/\langle Q_4 \rangle} \arrow["2",rd, no head] \arrow["2",u, no head] &       \\
			\substack{E_3=\\E_1/\langle P_2 + Q_2 \rangle}&                                       & \substack{E_4=\\E_1/\langle Q_2 \rangle} \arrow["2",d, no head] \arrow["2",ru, no head] &                            & \substack{E_8=\\E_1/\langle P_2+Q_8 \rangle}\\
			&                                       &  \substack{E_5=\\E_1/\langle P_2+Q_4 \rangle}                     &                            &      
		\end{tikzcd}
	\end{center}

	We will consider cases according to the isomorphism type of $E(\Q)_{\text{tors}}$. 
	
	\subsection{Elliptic curves with $\Z/2\Z\times \Z/8\Z$ rational torsion, and graph of type $T_8$} 
	
	Let $E/\Q$ be an elliptic curve with $E(\Q)_{\text{tors}} = \langle P_2,Q_8\rangle \cong \Z / 2 \Z \times \Z / 8 \Z$, and $C(E)=C_2(E)=8$. Then, we claim that the isogeny-torsion graph is of the form $([2,8],[8],[8],[2,4],[4],[2,2],[2],[2]).$
	
	Indeed, Lemma \ref{lem-Z2xZ4} shows that $E / \langle Q_{2}\rangle (\Q)_{\text{tors}} \cong \Z / 2 \Z \times \Z / 4 \Z$, Lemma \ref{lem-Z2xZ2} shows the isomorphism $E / \langle Q_{4}\rangle (\Q)_{\text{tors}} \cong \Z / 2 \Z \times \Z / 2 \Z$, and Lemma \ref{lem-Z2} shows that $E / \langle Q_{8}\rangle (\Q)_{\text{tors}} \cong E / \langle P_{2} + Q_{8}\rangle (\Q)_{\text{tors}} \cong \Z / 2 \Z.$
	
	The torsion groups $E / \langle P_{2}\rangle (\Q)_{\text{tors}}$ and $E / \langle P_{2} + Q_{2}\rangle (\Q)_{\text{tors}}$ are cyclic by Lemma \ref{Maximality-Of-Rational-2-Power-Groups}. For $R\in E$ that generates a Galois invariant subgroup, let $\phi_R \colon E \to E/\langle R\rangle$ be the corresponding isogeny with kernel $\left\langle R \right\rangle$. Then, by the third statement in Lemma  \ref{lem-necessity-for-point-rationality}, the points  $\phi_{P_{2}}(Q_{8})$ and $\phi_{P_{2} + Q_{2}}(Q_{8})$ are of order $8$ defined over $\Q$. Thus, $E / \langle P_{2}\rangle (\Q)_{\text{tors}} \cong E/ \langle P_{2} + Q_{2}\rangle (\Q)_{\text{tors}} \cong \Z / 8 \Z$. 
	
	It remains to determine $E / \langle P_{2} + Q_{4}\rangle (\Q)_{\text{tors}}$. Again, by Lemma \ref{Maximality-Of-Rational-2-Power-Groups}, $E / \langle P_{2} + Q_{4}\rangle (\Q)_{\text{tors}}$ is cyclic. Let $\phi=\phi_{P_2+Q_4} \colon E \to E / \left\langle P_{2} + Q_{4} \right\rangle$ be an isogeny with kernel $\left\langle P_{2}+Q_{4} \right\rangle$. Then, $\phi(Q_{8})$ is a point of order $4$ defined over $\Q$ by the third statement in Lemma \ref{lem-necessity-for-point-rationality}. The groups of order $8$ that contain $\phi(Q_{8})$ are $\langle \phi(Q_{16}) \rangle$ and $\langle \phi(P_{4} + Q_{16}) \rangle$. To generalize, let $A = Q_{16}$ or $P_{4} + Q_{16}$ and assume $\phi(A)$ is defined over $\Q$, then by Lemma \ref{lem-necessity-for-point-rationality}, $\sigma(A) - A \in \left\langle P_{2} + Q_{4} \right\rangle$ for all $\sigma \in G_{\Q}$. Note that $[2]A$ is a point of order $8$ defined over $\Q$. Note that $[8]A = Q_{2}$. Multiplying through by $2$, we have $\sigma([2]A) - [2]A \in \left\langle Q_{2} \right\rangle$ for all $\sigma \in G_{\Q}$ but as $[2]A$ is defined over $\Q$ we must have $\sigma([2]A) - [2]A = \mathcal{O}$ for all $\sigma \in G_{\Q}$. This means that originally, $\sigma(A) - A \in \left\langle Q_{2} \right\rangle$ for all $\sigma \in G_{\Q}$. As $Q_{2} \in \left\langle A \right\rangle$, this forces $\left\langle A \right\rangle$ to be a $\Q$-rational subgroup of $E$ of order $16$ by Lemma \ref{lem-necessity-for-subgroup-rationality}, but that is a contradiction as all the finite cyclic $\Q$-rational subgroups of $E$ are order at most $8$. Thus, $E / \left\langle P_{2} + Q_{4} \right\rangle(\Q)_{\text{tors}} \cong \Z / 4 \Z$.
	
	Hence, we have determined the torsion subgroups at every vertex and the isogeny-torsion graph is of type $([2,8],[8],[8],[2,4],[4],[2,2],[2],[2])$, as claimed.
	
	\subsection{Elliptic curves with $\Z/2\Z\times \Z/4\Z$ rational torsion, and graph of type $T_8$} 
	
	Let $E/\Q$ be an elliptic curve with $E(\Q)_{\text{tors}} = \langle P_2,Q_4\rangle \cong \Z / 2 \Z \times \Z / 4 \Z$, and $C(E)=C_2(E)=8$. Thus, we have $8$ $\Q$-rational groups, namely, $\langle \mathcal{O} \rangle, \langle P_{2} \rangle, \langle Q_{2} \rangle, \langle P_{2} + Q_{2} \rangle, \langle Q_{4} \rangle, \langle P_{2} + Q_{4} \rangle, \langle Q_{8} \rangle,$ and $\langle P_{2} + Q_{8} \rangle$. Then, we claim that the isogeny-torsion graph is of the form $([2,4],[4],[4],[2,4],[4],[2,2],[2],[2])$, or one of $([2,4],[4],[4],[2,4],[8],[2,2],[2],[2])$ or $([2,4],[8],[4],[2,4],[4],[2,2],[2],[2])$.
	
	As in the previous case, Lemma \ref{lem-Z2xZ4} shows that $E / \langle Q_{2}\rangle (\Q)_{\text{tors}} \cong \Z / 2 \Z \times \Z / 4 \Z$, Lemma \ref{lem-Z2xZ2} shows the isomorphism $E / \langle Q_{4}\rangle (\Q)_{\text{tors}} \cong \Z / 2 \Z \times \Z / 2 \Z$, and Lemma \ref{lem-Z2} shows that $E / \langle Q_{8}\rangle (\Q)_{\text{tors}} \cong E / \langle P_{2} + Q_{8}\rangle (\Q)_{\text{tors}} \cong \Z / 2 \Z.$ It remains to compute $E / \langle P_{2} \rangle (\Q)_{\text{tors}}, E / \langle P_{2} + Q_{2} \rangle (\Q)_{\text{tors}},$ and $E / \langle P_{2} + Q_{4} \rangle (\Q)_{\text{tors}}$, which are all cyclic by Lemma \ref{Maximality-Of-Rational-2-Power-Groups}. We will show they each of these three groups are isomorphic to $\Z / 4 \Z$ or $\Z / 8 \Z$ but at most one is isomorphic to $\Z / 8 \Z$.
	
	Let $\phi_{P_{2}+Q_{4}} \colon E \to E / \left\langle P_{2}+Q_{4} \right\rangle$ be an isogeny with kernel $\left\langle P_{2}+Q_{4} \right\rangle$. By Lemma \ref{orbit-of-supergroups-of-rational-groups}, $\sigma(Q_{8}) - Q_{8} \in \langle Q_{2} \rangle$ for all $\sigma \in G_{\Q}$. Thus, $\phi_{P_{2} + Q_{4}}(Q_{8})$ is a point of order $4$ defined over $\Q$ by Lemma \ref{lem-necessity-for-point-rationality}. The groups of order $8$ that contain $\phi_{P_{2} + Q_{4}}(Q_{8})$ are $\langle \phi_{P_{2} + Q_{4}}(Q_{16}) \rangle$ and $\langle \phi_{P_{2} + Q_{4}}(P_{4} + Q_{16}) \rangle$. So $E / \langle P_{2} + Q_{4} \rangle(\Q)_{\text{tors}} \cong \Z / 4 \Z$ or $\Z / 8 \Z$ depending on whether $\phi_{P_{2} + Q_{4}}(Q_{16})$ or $\phi_{P_{2} + Q_{4}}(P_{4} + Q_{16})$ is defined over $\Q$.
	
	To generalize, let $A$ be $P_{2}$ or $P_{2} + Q_{2}$ and let $\phi_{A} \colon E \to E / \left\langle A \right\rangle$ be an isogeny with kernel $\left\langle A \right\rangle$. The point $\phi_{A}(Q_{4})$ is of order $4$ defined over $\Q$ by the third statement in Lemma \ref{lem-necessity-for-point-rationality}. The groups of order $8$ that contain $\phi_{A}(Q_{4})$ are $\langle \phi_{A}(Q_{8}) \rangle$ and $\langle \phi_{A}(P_{4} + Q_{8}) \rangle$. Note that if $\phi_{A}(Q_{8})$ is defined over $\Q$, then $\sigma(Q_{8}) - Q_{8} \in \langle A \rangle$ for all $\sigma \in G_{\Q}$ by Lemma \ref{lem-necessity-for-point-rationality}. As $\sigma(Q_{8}) - Q_{8} \in \langle Q_{2} \rangle$ for all $\sigma \in G_{\Q}$, $\phi_{A}(Q_{8})$ is defined over $\Q$ if and only if $Q_{8}$ is defined over $\Q$, which is not the case. Thus, $E / \langle P_{2} \rangle (\Q)_{\text{tors}} \cong \Z / 8 \Z$ or $\Z / 4 \Z$ depending on whether or not $\phi_{P_{2}}(P_{4} + Q_{8})$ is defined over $\Q$ and $E / \langle P_{2} + Q_{2} \rangle(\Q)_{\text{tors}} \cong \Z / 8 \Z$ depending on whether or not $\phi_{P_{2} + Q_{2}}(P_{4} + Q_{8})$ is defined over $\Q$.
	
	Let us assume there are two isogenous curves that have torsion subgroups isomorphic to $\Z / 8 \Z$. Assume $E / \langle P_{2} \rangle (\Q)_{\text{tors}}$ and $E / \langle P_{2} + Q_{2} \rangle (\Q)_{\text{tors}}$ are both isomorphic to $\Z / 8 \Z.$ Thus, $\phi_{P_{2}}(P_{4} + Q_{8})$ and $\phi_{P_{2} + Q_{2}}(P_{4} + Q_{8})$ are both defined over $\Q$. Then, $\sigma(P_{4} + Q_{8}) - (P_{4} + Q_{8}) \in \langle P_{2} \rangle$ and $\sigma(P_{4} + Q_{8}) - (P_{4} + Q_{8}) \in \langle P_{2} + Q_{2} \rangle$ for all $\sigma \in G_{\Q}$ by Lemma \ref{lem-necessity-for-point-rationality}. This would force $P_{4} + Q_{8}$ to be defined over $\Q$, a contradiction.
	
	Now let us see what happens when $E / \langle P_{2} \rangle (\Q)_{\text{tors}}$ and $E / \langle P_{2} + Q_{4} \rangle (\Q)_{\text{tors}}$ are both isomorphic to $\Z / 8 \Z$ (interchanging between the two bases, $\langle P_{2}, Q_{2} \rangle$ and $\langle P_{2} + Q_{2}, Q_{2} \rangle$ of $E[2]$ shows that eliminating the case of $E / \langle P_{2} \rangle(\Q)_{\text{tors}} \cong E / \langle P_{2} + Q_{4} \rangle(\Q)_{\text{tors}} \cong \Z / 8 \Z$ also eliminates the case of $E / \langle P_{2} + Q_{2} \rangle(\Q)_{\text{tors}} \cong E / \langle P_{2}  + Q_{4} \rangle(\Q)_{\text{tors}} \cong \Z / 8 \Z$). Either $\phi_{P_{2} + Q_{4}}(Q_{16})$ and $\phi_{P_{2}}(P_{4} + Q_{8})$ are defined over $\Q$ or $\phi_{P_{2} + Q_{4}}(P_{4} + Q_{16})$ and $\phi_{P_{2}}(P_{4} + Q_{8})$ are defined over $\Q$. Let us assume the first case (interchanging again between the two bases, $\langle P_{16}, P_{4} + Q_{16} \rangle$ and $\langle P_{16}, Q_{16} \rangle$ of $E[16]$ shows that eliminating the case of $\phi_{P_{2}}(P_{4} + Q_{8})$ and $\phi_{P_{2} + Q_{4}}(Q_{16})$ being defined over $\Q$ also eliminates the case of $\phi_{P_{2}}(P_{4} + Q_{8})$ and $\phi_{P_{2} + Q_{4}}(P_{4} + Q_{16})$ being defined over $\Q$). So for all $\sigma \in G_{\Q}$, there are integers $a=a(\sigma)$ and $b=b(\sigma)$ such that 
	$$\sigma(Q_{16}) = Q_{16}+[a](P_2+Q_4) = [8a]P_{16} + [1+4a]Q_{16}$$
	and $$\sigma(P_4)= P_4+Q_{8}+[b]P_2 - \sigma(Q_8) = [1+2b]P_4-[2a]Q_4.$$
	Thus, the image of $\rho_{E,16}$ is contained in the following group of order $128$:
	$$G_{16}=\left\{\left(\begin{array}{cc} c & 8a \\ d & 1+4a \end{array} \right) : c\equiv 1\bmod 2, d\equiv -2a\equiv 2a \bmod 4  \right\}\subseteq \GL(2,\Z/16\Z).$$
	
	We searched the Rouse--Zureick-Brown database \cite{rouse} of $2$-adic images for groups in $\GL(2,\Z/16\Z)$ that are conjugates of a subgroup of $G_{16}$ above, and found none. Thus, the image cannot be contained in $G_{16}$ and we have reached a contradiction. Hence, only one of $E/\langle P_2\rangle, E/\langle P_{2} + Q_{2}\rangle,$ or $E/\langle P_{2} + Q_{4}\rangle$ may have $\Z/8\Z$ torsion defined over $\Q$. Thus, the possible torsion configurations for elliptic curves in a $T_8$ graph, such that $E(\Q)_\text{tors}\cong \Z/2\Z\times\Z/4\Z$ and $\left\langle Q_{8} \right\rangle$ is a $\Q$-rational subgroup of $E$ are $([2,4],[4],[4],[2,4],[4],[2,2],[2],[2])$, or one of $([2,4],[4],[4],[2,4],[8],[2,2],[2],[2])$ or $([2,4],[8],[4],[2,4],[4],[2,2],[2],[2])$, and examples are shown in our tables.

	\subsection{Elliptic curves with $\Z / 2 \Z \times \Z / 2 \Z$ rational torsion and graph of type $T_{8}$} Let $E / \Q$ be an elliptic curve with $E(\Q)_{\text{tors}} = \langle P_{2}, Q_{2} \rangle \cong \Z / 2 \Z \times \Z / 2 \Z$ and $C(E) = C_{2}(E) = 8$. Thus, we have $8$  finite cyclic  $\Q$-rational groups, namely, $\langle \mathcal{O} \rangle, \langle P_{2} \rangle, \langle P_{2} + Q_{2} \rangle, \langle Q_{2} \rangle, \langle Q_{4} \rangle, \langle P_{2} + Q_{4} \rangle, \langle Q_{8} \rangle,$ and $\langle P_{2} + Q_{8} \rangle$. Suppose moreover that $E$ is not isogenous to any elliptic curves whose torsion subgroups are of order $8$. We claim that the isogeny-torsion graph is of the form $([2,2], [2], [2], [2,2], [2], [2,2], [2], [2])$ or $([2,2], [4], [2], [2,2], [2], [2,2], [2], [2])$.
	
	For each element $R \in E$ that generates a $\Q$-rational subgroup of $E$, let $\phi_{R} \colon E \to E / \left\langle R \right\rangle$ be an isogeny with kernel $\left\langle R \right\rangle$. Note that $E / \langle Q_{2} \rangle(\Q)_{\text{tors}} \cong \Z / 2 \Z \times \Z / 2 \Z$. This is because $\phi_{Q_{2}}(P_{2})$ is a point of order $2$ defined over $\Q$ by the third statement of Lemma \ref{lem-necessity-for-point-rationality}. By the fact that $\langle Q_{4} \rangle$ is $\Q$-rational and Lemma \ref{orbit-of-supergroups-of-rational-groups}, $\phi_{Q_{2}}(Q_{4})$ is another point of order $2$ defined over $\Q$ by the first statement of Lemma \ref{lem-necessity-for-point-rationality}. Thus, $E / \langle Q_{2} \rangle$ has full two-torsion. If $E / \langle Q_{2} \rangle(\Q)_{\text{tors}} \cong \Z / 2 \Z \times \Z / 4 \Z$, then we are reduced to cases of $T_8$ graphs that we have already covered above, so we shall assume here that $E / \langle Q_{2} \rangle(\Q)_{\text{tors}} \cong \Z / 2 \Z \times \Z / 2 \Z$.
	
	By Lemma \ref{Maximality-Of-Rational-2-Power-Groups}, $E / \langle P_{2} + Q_{4} \rangle(\Q)_{\text{tors}}$ is cyclic. The element of order two $\phi_{P_{2} + Q_{4}}(P_{2})$ is contained in the groups of order $4$ given by  $\langle \phi_{P_{2} + Q_{4}}(P_{4}) \rangle$ and $\langle \phi_{P_{2} + Q_{4}}(Q_{8}) \rangle$. If $\phi_{P_{2} + Q_{4}}(P_{4})$ is defined over $\Q$, then $\sigma(P_{4}) - P_{4} \in \langle P_{2} + Q_{4} \rangle$ for all $\sigma \in G_{\Q}$, and by Lemma \ref{orbit-of-supergroups-of-rational-groups}, $\sigma(P_{4}) - P_{4} \in \langle Q_{2} \rangle$. But then $\phi_{Q_{2}}(P_{4})$ would be a point of order $4$ defined over $\Q$ and $\phi_{Q_{2}}(Q_{4})$ is a point of order $2$ defined over $\Q$ not lying in $\langle \phi_{Q_{2}}(P_{4}) \rangle$. Thus, $E / \langle Q_{2} \rangle(\Q)_{\text{tors}} \cong \Z / 2 \Z \times \Z / 4 \Z$, contradicting the assumption we made that $E$ is not isogenous to any elliptic curve with torsion subgroup of order $8$. Now if we assume $\phi_{P_{2} + Q_{4}}(Q_{8})$ is defined over $\Q$, then $\sigma(Q_{8}) - Q_{8} \in \langle P_{2} + Q_{4} \rangle$ for all $\sigma \in G_{\Q}$. As $\langle Q_{8} \rangle$ is $\Q$-rational, this forces $\sigma(Q_{8}) - Q_{8} \in \langle Q_{2} \rangle$ for all $\sigma \in G_{\Q}$. Multiplying through by $2$ shows that $\sigma(Q_{4}) - Q_{4} = \mathcal{O}$ for all $\sigma \in G_{\Q}$, making $Q_{4}$ defined over $\Q$, a contradiction. Thus, $E / \langle P_{2} + Q_{4} \rangle(\Q)_{\text{tors}} \cong \Z / 2 \Z$.
	
	By Lemma \ref{lem-Z2xZ2}, $E / \langle Q_{4} \rangle(\Q)_{\text{tors}} \cong \Z / 2 \Z \times \Z / 2 \Z$. By Lemma \ref{Not-too-many-Z/4}, $E / \langle P_{2} \rangle(\Q)_{\text{tors}}$ and $E / \langle P_{2} + Q_{2} \rangle(\Q)_{\text{tors}}$ are cyclic of order $2$ or $4$ but not both order $4$ and $E / \langle Q_{8} \rangle(\Q)_{\text{tors}}$ and $E / \langle P_{2} + Q_{8} \rangle(\Q)_{\text{tors}}$ are cyclic of order $2$ or $4$ but not both order $4$. We claim that there is at most one isogenous elliptic curve with torsion subgroup $\Z / 4 \Z$.
	
	By \ref{Not-too-many-Z/4}, if we assume that there are two curves with torsion subgroup $\Z / 4 \Z$, then either one and only one of $E / \langle P_{2} \rangle(\Q)_{\text{tors}}$ or $E / \langle P_{2} + Q_{2} \rangle(\Q)_{\text{tors}}$ is isomorphic to $\Z / 4 \Z$ and one and only one of $E / \langle Q_{8} \rangle(\Q)_{\text{tors}}$ or $E / \langle P_{2} + Q_{8}\rangle(\Q)_{\text{tors}}$ is isomorphic to $\Z / 4 \Z$. Interchanging between the bases $\langle P_{2}, Q_{2} \rangle$ and $\langle P_{2} + Q_{2}, Q_{2} \rangle$ of $E[2]$ shows that switching between $E / \langle P_{2} \rangle$ and $E / \langle P_{2} + Q_{2} \rangle$ is inconsequential. Similarly, interchanging between the bases $\langle P_{8}, Q_{8} \rangle$ and $\langle P_{8}, P_{2} + Q_{8} \rangle$ of $E[8]$ shows that switching between $E / \langle Q_{8} \rangle$ and $E / \langle P_{2} + Q_{8} \rangle$ is inconsequential.
	
	Let us assume $E / \langle P_{2} \rangle(\Q)_{\text{tors}} \cong E / \langle Q_{8} \rangle(\Q)_{\text{tors}} \cong \Z / 4 \Z$. Note that the element of order $2$ of $E / \langle P_{2} \rangle$ is $\phi_{P_{2}}(Q_{2})$, which is contained in the groups of order $4$ given by $\langle \phi_{P_{2}}(Q_{4}) \rangle$ and $\langle \phi_{P_{2}}(P_{4} + Q_{4}) \rangle$. If we assume $\phi_{P_{2}}(Q_{4})$ is defined over $\Q$, then $\sigma(Q_{4}) - Q_{4} \in \langle P_{2} \rangle$. As $\langle Q_{4} \rangle$ is $\Q$-rational, $\sigma(Q_{4}) - Q_{4} \in \langle Q_{2} \rangle$, forcing $Q_{4}$ to be defined over $\Q$, a contradiction. Thus, only $\phi_{P_{2}}(P_{4} + Q_{4})$ has the possibility of being defined over $\Q$. The element of order $2$ of $E / \langle Q_{8} \rangle$ is $\phi_{Q_{8}}(P_{2})$, which is contained in the groups of order $4$ given by $\langle \phi_{Q_{8}}(P_{4}) \rangle$ and $\langle \phi_{Q_{8}}(P_{4} + Q_{16}) \rangle$. If we assume that $\phi_{Q_{8}}(P_{4})$ is defined over $\Q$, then $\sigma(P_{4}) - P_{4} \in \langle Q_{8} \rangle$ for all $\sigma \in G_{\Q}$. By Lemma \ref{orbit-of-supergroups-of-rational-groups}, $\sigma(P_{4}) - P_{4} \in \langle Q_{2} \rangle$ for all $\sigma \in G_{\Q}$. By the first statement of Lemma \ref{lem-necessity-for-point-rationality}, $\phi_{Q_{2}}(P_{4})$ is an element of order $4$ defined over $\Q$. By Lemma \ref{orbit-of-supergroups-of-rational-groups} and the first statement of Lemma \ref{lem-necessity-for-point-rationality}, $\phi_{Q_{2}}(Q_{4})$ is an element of order $2$ defined over $\Q$, not in $\langle \phi_{Q_{2}}(P_{4}) \rangle$ which makes $E / \langle Q_{2} \rangle(\Q)_{\text{tors}} \cong \Z / 2 \Z \times \Z / 4 \Z$. As we assume $E$ is not isogenous to any elliptic curve whose torsion subgroup is order $8$, this is a contradiction.
	
	Let us then assume that both $\phi_{P_{2}}(P_{4} + Q_{4})$ and $\phi_{Q_{8}}(P_{4} + Q_{16})$ are defined over $\Q$. So for all $\sigma \in G_{\Q}$, there are integers $a=a(\sigma)$ and $b=b(\sigma)$ such that 
	$$\sigma(P_4+Q_{4}) = P_4+Q_{4}+[a](P_2) = [1+2a]P_{4} + Q_{4},$$
	and $$\sigma(P_4+Q_{16})= P_4+Q_{16}+[b]Q_8 = P_4+[1+2b]Q_{16}.$$
	It follows that $\sigma(Q_{16}) = [-40a]P_{16}+[1+10b]Q_{16}$ and $\sigma(P_4)=[1+2a]P_4-[10b]Q_4$. Thus, the image of $\rho_{E,16}$ is contained in the following group of order $256$:
	$$G_{16}=\left\{\left(\begin{array}{cc} c & -40a \\ d & 1+10b \end{array} \right) : c\equiv 1+2a\bmod 4, d\equiv -10b\bmod 4  \right\}\subseteq \GL(2,\Z/16\Z).$$
	
	We searched the Rouse--Zureick-Brown database \cite{rouse} of $2$-adic images for groups in $\GL(2,\Z/16\Z)$ that are conjugates of a subgroup of $G_{16}$ above, and found none. Thus, the image cannot be contained in $G_{16}$ and we have reached a contradiction. Hence, only one of $E / \langle P_{2} \rangle(\Q)_{\text{tors}}$ or $E / \langle Q_{8} \rangle(\Q)_{\text{tors}}$ (or $E / \langle P_{2}+Q_{2} \rangle(\Q)_{\text{tors}}$ or $E / \langle P_{2}+Q_{8} \rangle(\Q)_{\text{tors}}$) may be isomorphic to $\Z / 4 \Z$. Therefore, the possible isogeny-torsion graphs are either one of the following forms: $([2,2], [2], [2], [2,2], [2], [2,2], [2], [2])$ or $([2,2], [4], [2], [2,2], [2], [2,2], [2], [2])$, as claimed.
	
	\section{Elliptic curves with a graph of type $S$}\label{sec-Sgraphs}
	
	In this section we consider elliptic curves with $C_2(E)= 4$ and $C_3(E)= 2$, so that $C(E)=8$. In this case, the $2$-torsion is defined over the rationals, $E(\Q)[2]=\langle P_2,Q_2\rangle$, and there is an additional $\Q$-rational group $\langle A_3\rangle$ of order $3$. Then, the  finite cyclic $\Q$-rational groups of $E$ are $\langle \mathcal{O} \rangle, \langle P_{2} \rangle, \langle Q_{2} \rangle, \langle P_{2} + Q_{2} \rangle, \langle A_{3} \rangle, \langle P_{2} + A_{3} \rangle, \langle Q_{2} + A_{3} \rangle,$ and $\langle P_{2} + Q_{2} + A_{3} \rangle$. The isogeny graph is therefore:
	
	\begin{center} 
	\begin{tikzcd}
		& E_{3}=E_1/\langle P \rangle  \arrow["3",r, no head]                                 & E_{4}=E_1/\langle P+A \rangle                                 &       \\
		& E_{1} \arrow["2",u, no head] \arrow["3",r, no head] \arrow["2",rd, no head] \arrow["2",ld, no head] & E_{2}=E_1/\langle A \rangle \arrow["2",u, no head] \arrow["2",rd, no head] \arrow["2",ld, no head] &       \\
		E_{5}=E_1/\langle P+Q \rangle \arrow["3",r, no head] & E_{6}=E_1/\langle P+Q+A \rangle                                & E_{7}=E_1/\langle Q \rangle \arrow["3",r, no head]                       & E_{8}=E_1/\langle Q+A \rangle
	\end{tikzcd}
\end{center} 
	
	We will distinguish two cases, according to whether there is a rational point of order $3$. We remind the reader here that since $E$ has its full two-torsion defined over $\Q$, every $\Q$-isogenous curve to $E$ has at least one $2$-torsion point defined over $\Q$, by Lemma \ref{lem-2torspt-all-have-2torspt}.
	
	\subsection{Elliptic curves with $\Z / 2 \Z \times \Z / 6 \Z$ rational torsion, and graph of type $S$}\label{sec-S1} Let $E = E_{1} / \Q$ be an elliptic curve with $E(\Q)_{\text{tors}} = \langle P_{2}, Q_{2}, A_{3} \rangle \cong \Z / 2 \Z \times \Z / 6 \Z$ where $E(\Q)[2] = \langle P_{2}, Q_{2} \rangle$ and $A_{3} \in E(\Q)[3]$.  
	
	Let $\phi_{A_{3}} \colon E \to E / \langle A_{3} \rangle$ be an isogeny with kernel $\langle A_{3} \rangle$. By Lemma \ref{lem-subsequent-rational-pts}, $E / \langle A_{3} \rangle$ has no points of order $3$ defined over $\Q$ and by the third statement of Lemma \ref{lem-necessity-for-point-rationality}, the points $\phi_{A_{3}}(P_{2})$ and $\phi_{A_{3}}(Q_{2})$ are distinct of order $2$ defined over $\Q$, hence $E / \langle A_{3} \rangle(\Q)_{\text{tors}} \cong \Z / 2 \Z \times \Z / 2 \Z$.
	
	Let $\phi_{P_{2}} \colon E \to E / \langle P_{2} \rangle$ be an isogeny with kernel $\langle P_{2} \rangle$. By Lemma \ref{Maximality-Of-Rational-2-Power-Groups}, $E / \langle P_{2} \rangle(\Q)_{\text{tors}}$ is cyclic. By the third statement of Lemma \ref{lem-necessity-for-point-rationality}, we have that $\phi_{P_{2}}(Q_{2})$ and $\phi_{P_{2}}(A_{3})$ are points of order $2$ and $3$ respectively defined over $\Q$. Thus, $E / \langle P_{2} \rangle(\Q)_{\text{tors}} \cong \Z / 6 \Z$ or $\Z / 12 \Z$. 
	
	Let $\phi_{P_{2} + A_{3}} \colon E \to E / \langle P_{2} + A_{3} \rangle$ be an isogeny with kernel $\left\langle P_{2} + A_{3} \right\rangle$. Then,  $\phi_{P_{2}}(P_{4})$ is a point of order $2$ not defined over $\Q$, so by Lemma \ref{lem-rectangle-lemma}, the point $\phi_{P_{2} + A_{3}}(P_{4})$ is of order $2$ not defined over $\Q$. Thus, $E / \langle P_{2} + A_{3} \rangle (\Q)_{\text{tors}}$ is cyclic. $E / \langle A_{3} \rangle$ has no points of order $3$ defined over $\Q$, so by Lemma \ref{lem-rectangle-lemma}, $E / \langle P_{2} + A_{3} \rangle$ has no points of order $3$ defined over $\Q$. Thus, $E / \langle P_{2} + A_{3} \rangle(\Q)_{\text{tors}} \cong \Z / 2 \Z$ or $\Z / 4 \Z$. A final application of Lemma \ref{lem-rectangle-lemma} shows that $E / \langle P_{2} \rangle(\Q)_{\text{tors}} \cong \Z / 12 \Z$ if and only if $E / \langle P_{2} + A_{3} \rangle(\Q)_{\text{tors}} \cong \Z / 4 \Z$ and $E / \langle P_{2} \rangle(\Q)_{\text{tors}} \cong \Z / 6 \Z$ if and only if $E / \langle P_{2} + A_{3} \rangle(\Q)_{\text{tors}} \cong \Z / 2 \Z$.
	
	We can use the work we did above to arrive at identical results when rewriting $Q_{2}$ or $P_{2} + Q_{2}$ in place of $P_{2}$. More specifically, $E / \langle Q_{2} \rangle(\Q)_{\text{tors}} \cong \Z / 12 \Z$ if and only if $E / \langle Q_{2} + A_{3} \rangle \cong \Z / 4 \Z$ and $E / \langle Q_{2} \rangle(\Q)_{\text{tors}} \cong \Z / 6 \Z$ if and only if $E / \langle Q_{2} + A_{3} \rangle \cong \Z / 2 \Z$. And finally, $E / \langle P_{2} + Q_{2} \rangle(\Q)_{\text{tors}} \cong \Z / 12 \Z$ if and only if $E / \langle P_{2} + Q_{2} + A_{3} \rangle \cong \Z / 4 \Z$ and $E / \langle P_{2} + Q_{2} \rangle(\Q)_{\text{tors}} \cong \Z / 6 \Z$ if and only if $E / \langle P_{2} + Q_{2} + A_{3} \rangle \cong \Z / 2 \Z$.
	
	Hence, we have shown that if $E(\Q)_\text{tors}\cong \Z/2\Z\times \Z/6\Z$, then $E/\langle A_3 \rangle (\Q))_\text{tors} \cong \Z/2\Z \times \Z/2\Z$, and the rest of the vertices come in pairs, and the torsion groups in each pair are either $([12],[4])$ or $([6],[2])$. Suppose that the torsion subgroups in all three pairs were of the form $([12],[4])$. Then, if we consider $E'=E/\langle A_3 \rangle$, it follows that the $2$-power isogenies of $E'$ form a $T_4$ graph with $\Z/4\Z$ at each corner. However, we showed in Section \ref{sec-T4} that this is impossible. Thus, at most two pairs of elliptic curves at the corners of the isogeny-torsion graph of $E$ may have torsion subgroups $([12],[4])$.
	
	We have found examples of isogeny-torsion graphs of the form $([2,6],[2,2],[6],[2],[6],[2],[6],[2])$ and $([2,6],[2,2],[12],[4],[6],[2],[6],[2])$, that is, with zero or one $\Z/12\Z$ torsion subgroups. In Section \ref{sec-elusive} we will show that $([2,6],[2,2],[12],[4],[12],[4],[6],[2])$, with two copies of $\Z/12\Z$ in the graph, is impossible over $\Q$.

	\subsection{Elliptic curves with $\Z / 2 \Z \times \Z / 2 \Z$ rational torsion, and graph of type $S$}\label{sec-S2}
	
	Let $E = E_{1} / \Q$ with $E(\Q)_{\text{tors}} = \langle P_{2}, Q_{2} \rangle \cong \Z / 2 \Z \times \Z / 2 \Z$. Let $A_{3}$ be an element of $E$ of order $3$ that generates a $\Q$-rational subgroup of $E$ but is not defined over $\Q$. Then, again, the  finite cyclic $\Q$-rational groups of $E$ are $\langle \mathcal{O} \rangle, \langle P_{2} \rangle, \langle Q_{2} \rangle, \langle P_{2} + Q_{2} \rangle, \langle A_{3} \rangle, \langle P_{2} + A_{3} \rangle, \langle Q_{2} + A_{3} \rangle,$ and $\langle P_{2} + Q_{2} + A_{3} \rangle$. As $E$ does not have a point of order $3$ defined over $\Q$, neither do $E / \langle P_{2} \rangle, E / \langle Q_{2} \rangle,$ or $E / \langle P_{2} + Q_{2} \rangle$ by Lemma \ref{lem-necessity-for-point-rationality}. If any one of $E / \langle A_{3} \rangle, E / \langle P_{2} + A_{3} \rangle, E / \langle Q_{2} + A_{3} \rangle,$ or $E / \langle P_{2} + Q_{2} + A_{3} \rangle$ has a point of order $3$ defined over $\Q$, then $E / \langle A_{3} \rangle (\Q)_{\text{tors}} \cong \Z / 2 \Z \times \Z / 6 \Z$ by Lemma \ref{lem-necessity-for-point-rationality} and we are back to the case already considered in Section \ref{sec-S1}. Thus, we may assume that none of the isogenous curves have an element of order $3$ defined over $\Q$. 
	
	All of the work we did for graphs for elliptic curves with $\Z / 2 \Z \times \Z / 6 \Z$ rational torsion apply for this case, except no torsion subgroup has a point of order $3$. In other words, the torsion subgroups are of order $2$ or $4$. Thus, $E / \langle A_{3} \rangle (\Q)_{\text{tors}} \cong \Z / 2 \Z \times \Z / 2 \Z$, $E / \langle P_{2} \rangle(\Q)_{\text{tors}} \cong E / \langle P_{2} + A_{3} \rangle(\Q)_{\text{tors}} \cong \Z / 4 \Z$ or $\Z / 2 \Z$, and $E / \langle Q_{2} \rangle(\Q)_{\text{tors}} \cong E / \langle Q_{2} + A_{3} \rangle(\Q)_{\text{tors}} \cong \Z / 4 \Z$ or $\Z / 2 \Z$, and $E / \langle P_{2} + Q_{2} \rangle(\Q)_{\text{tors}} \cong E / \langle P_{2} + Q_{2} + A_{3} \rangle(\Q)_{\text{tors}} \cong \Z / 4 \Z$ or $\Z / 2 \Z$. Thus, as in the previous case in Section \ref{sec-S1}, there are three pairs that have torsion $([4],[4])$ or $([2],[2])$. Using the same reasoning as before (using our work from the $T_4$ graphs in Section \ref{sec-T4}), it is impossible for all three pairs to be of the form $([4],[4])$. 
	
	We have found examples of isogeny-torsion graphs of the form $([2,2],[2,2],[2],[2],[2],[2],[2],[2])$ and $([2,2],[2,2],[4],[4],[2],[2],[2],[2])$, that is, with either none or one pair having torsion $([4],[4])$. In Section \ref{sec-elusive} we will show that $([2,2],[2,2],[4],[4],[4],[4],[2],[2])$, that is, with two pairs having torsion $([4],[4])$, is impossible over $\Q$.
	
	\subsection{Two $S$-type graphs that do not exist over $\Q$}\label{sec-elusive}
	
	In order to complete the classification of $S$-type isogeny-torsion graphs, it remains to show that $([2,2],[2,2],[4],[4],[4],[4],[2],[2])$ and  $([2,6],[2,2],[12],[4],[12],[4],[6],[2])$ do not occur over $\Q$.
	
	Let $E/\Q$ be an elliptic curve with full two-torsion and $C(E)=8$, such that $C_2(E)=4$ and $C_3(E)=2$. We note here that these conditions on isogenies imply that $E/\Q$ is necessarily a non-CM elliptic curve, by our work in Section \ref{sec-CM}. In addition, let us assume that the isogeny-torsion graph of $E$ is of the type $([2,2],[2,2],[4],[4],[4],[4],[2],[2])$ or  $([2,6],[2,2],[12],[4],[12],[4],[6],[2])$. In either case, $E(\Q)[2]=\langle P_2,Q_2\rangle$, there are no other $2$-power isogenies other than those coming from the $2$-torsion, and there is a $\Q$-rational subgroup $\langle A_3\rangle$ of order $3$ (and in the second case of isogeny-torsion graph, we have $A_3$ defined over $\Q$). In other words, $E/\Q$ satisfies two properties: (i) the $2$-isogeny-torsion graph is of type $T_4$, with two $\Z/4\Z$ torsion subgroups in corner vertices, and (ii) there is a $3$-isogeny. For $R \in E$ be an element that generates a $\Q$-rational subgroup fo $E$, let $\phi_{R} \colon E \to E / \left\langle R \right\rangle$ be an isogeny with kernel $\left\langle R \right\rangle$.
	
First, we shall characterize elliptic curves that satisfy (i), i.e., $E/\Q$ with $2$-isogeny-torsion graph of type $T_4$ with two $\Z/4\Z$ groups in corner vertices. After a relabeling of a basis of $E[2]$, if necessary, we may assume that the curves with $\Z/4\Z$ torsion are $E / \langle P_{2} \rangle$, and  $E / \langle Q_{2} \rangle$. Note that the rational element of order $2$ of $E / \langle P_{2} \rangle$ is $\phi_{P_{2}}(Q_{2})$, which is contained in the groups of order $4$ given by $\langle \phi_{P_{2}}(Q_{4}) \rangle$ and $\langle \phi_{P_{2}}(P_{4} + Q_{4}) \rangle$. Similarly, the rational element of order $2$ of $E / \langle Q_{2} \rangle$ is $\phi_{Q_{2}}(P_{2})$, which is contained in the groups of order $4$ given by $\langle \phi_{Q_{2}}(P_{4}) \rangle$ and $\langle \phi_{Q_{2}}(P_{4} + Q_{4}) \rangle$.

If we assume that both $\phi_{P_{2}}(P_{4} + Q_{4})$ and $\phi_{Q_{2}}(P_{4} + Q_{4})$ are defined over $\Q$, then by Lemma \ref{lem-necessity-for-point-rationality}, we have $\sigma(P_{4} + Q_{4}) - (P_{4} + Q_{4}) \in \langle P_{2} \rangle$ and $\sigma(P_{4} + Q_{4}) - (P_{4} + Q_{4}) \in \langle Q_{2} \rangle$ for all $\sigma \in G_{\Q}$. Thus, $\sigma(P_{4} + Q_{4}) - (P_{4} + Q_{4}) = \mathcal{O}$ for all $\sigma \in G_{\Q}$. Therefore, $P_{4} + Q_{4}$ is defined over $\Q$ and thus, $\langle P_{4} + Q_{4} \rangle$ is a $\Q$-rational group of order $4$, but $E/\Q$ does not have $4$-isogenies (only $2$-isogenies), otherwise $C(E)>8$. If we assume that both $\phi_{Q_{2}}(P_{4})$ and $\phi_{P_{2}}(Q_{4})$ are defined over $\Q$, then by the first statement of Lemma \ref{lem-necessity-for-point-rationality}, $\sigma(P_{4}) - P_{4} \in \langle Q_{2} \rangle$ and $\sigma(Q_{4}) - Q_{4} \in \langle P_{2} \rangle$ for all $\sigma \in G_{\Q}$. But then the image of the mod $4$ Galois representation attached to $E$ is a subgroup of 
$$ \left\{\left(\begin{array}{cc} 1 & 2a \\ 2b & 1 \end{array} \right) : a,b\in \Z / 4 \Z \right\}.$$
However, the determinant map would not surject onto $(\Z / 4 \Z)^{\times}$ so we have a contradiction.

	Thus, we may assume $\phi_{Q_{2}}(P_{4})$ and $\phi_{P_{2}}(P_{4} + Q_{4})$ are defined over $\Q$ (the other possibility is obtained by exchanging the roles of $P_2$ and $Q_2$). Then, by Lemma \ref{lem-necessity-for-point-rationality}, we have $\sigma(P_{4}) - P_{4} \in \langle Q_{2} \rangle$ and $\sigma(P_{4} + Q_{4}) - (P_{4} + Q_{4}) \in \langle P_{2} \rangle$ for all $\sigma \in G_{\Q}$. Thus, $\sigma(P_{4}) = P_{4}+[2a]Q_4$ for some $a=a(\sigma)$. Since $\sigma(P_{4} + Q_{4}) - (P_{4} + Q_{4}) \in \langle P_{2} \rangle$, we have $\sigma(P_4+Q_4)=[1+2b]P_4+Q_4$ for some $b=b(\sigma)$. Thus, $\sigma(Q_{4})=[2b]P_4+[1-2a]Q_4$  and the image of the mod $4$ Galois representation is a subgroup of the Klein $4$-group:
	$$H = \left\{\left(\begin{array}{cc} 1 & 0 \\ 0 & 1 \end{array} \right), \left(\begin{array}{cc} 1 & 2 \\ 0 & 1 \end{array} \right), \left(\begin{array}{cc} 1 & 0 \\ 2 & 3 \end{array} \right), \left(\begin{array}{cc} 1 & 2 \\ 2 & 3 \end{array} \right) \right\} .$$
	Let $\widehat{H}\subseteq \GL(2,\Z_2)$ be the subgroup such that $\widehat{H}\equiv H \bmod 4$. Using the Rouse--Zureick-Brown database of $2$-adic images, we find that the (non-CM) elliptic curves with $2$-adic image contained in $\widehat{H}$ are parametrized by the points on $X_{24e}$ in the notation of \cite{rouse}. Further, the modular curve $X_{24e}$ is of genus $0$, isomorphic to $\PP^1$ (with coordinate $t$), and a parametrization of the elliptic curves in the family is given by
	\[E_{24e}: y^2 = x^3 -(27t^{4} + 27t^{2} + 27)x + (54t^{6} + 81t^{4} - 81t^{2} - 54),\]
	with $j$-invariant equal to $j_{E_{24e}}(t) = \frac{(t^4+t^2+1)^3}{t^4(t^2+1)^2}.$
	
	On the other hand, the $j$-invariants of curves satisfying (ii) above, i.e., the elliptic curves $E/\Q$ with a $3$-isogeny, are parametrized by the non-cuspidal points of $X_0(3)$, a curve of genus $0$. The $j$-invariant of such a family can be found for example in \cite{lozano0}, and it is given by $j(s)=(s+27)(s+243)^3/s^3$. Since our assumptions imply that $E/\Q$ is associated to a point in both $X_{24e}$ and $X_0(3)$, it follows that there are rational numbers $t$ and $s$ such that
	$$\frac{(t^4+t^2+1)^3}{t^4(t^2+1)^2} = \frac{(s+27)(s+243)^3}{s^3}.$$
	Let $C$ be the curve (of genus $13$) defined by this equation, that is,
	$$C: (t^4+t^2+1)^3s^3-t^4(t^2+1)^2(s+27)(s+243)^3=0.$$
	If we let $C': (t^2+t+1)^3s^3-t^2(t+1)^2(s+27)(s+243)^3=0$ instead, then there is clearly a map $\phi\colon C\to C'$ that sends $(s,t)\mapsto (s,t^2)$. The curve $C'$ is of genus $6$, and a computation with Magma (code available at \cite{lozanoweb}) shows that $C'$ has an automorphism group of order $6$. One of the automorphisms of $C'$ is the map that sends $(t,s)\mapsto (-1-1/t,s)$, which corresponds to the  automorphism of $C'$ in projective coordinates given by
	$$\psi\colon (t,s,z) \mapsto (-tz-z^2,ts,tz).$$
	The quotient $C'' = C'/\langle \psi\rangle$ is a genus $2$  hyperelliptic curve with equation
	$$C'':  y^2 + x^2y = -x^5 - x^4 + 4x^3 - 2x^2 - 9x  + 2.$$
	Again using Magma, a descent shows that the jacobian $J(C'')/\Q$ is of rank $0$, and therefore one can use the Chabauty method to compute the rational points on $C''$. These are $[-2,-2,1]$ and $[1,0,0]$ in projective coordinates. Pulling these points back to $C'$ via the map $C'\to C'/\langle \psi\rangle = C''$, we obtain the rational points on $C'$, which are
	$$[t,s,z]\in \{ [-1,0,1], [0,0,1], [ 0,1,0], [1,0,0]\}. $$
	Note that all of these points have either $t=0$ or $s=0$ and are therefore cusps in $C$, so they do not correspond to any elliptic curve with the properties we sought. 
	
	Hence, there are no elliptic curves $E/\Q$ that satisfy (i) and (ii), and we conclude that there are no elliptic curves over $\Q$
	whose isogeny-torsion graph is one of  $([2,2],[2,2],[4],[4],[4],[4],[2],[2])$ or  $([2,6],[2,2],[12],[4],[12],[4],[6],[2])$. This completes the proof of Theorem \ref{thm-main2}.

\bibliography{alvaro-garen}{}
\bibliographystyle{plain}

\end{document}